\documentclass[11pt]{article}
\usepackage[a4paper]{geometry}
\usepackage[latin1]{inputenc}
\usepackage{amsmath}
\usepackage{amssymb}
\usepackage{amsthm}
\usepackage{enumerate}
\usepackage{graphicx}
\usepackage{color}
\usepackage{cite}
\usepackage{dsfont} 

\numberwithin{equation}{section}

\newcommand{\N}{\mathbb{N}}
\newcommand{\Z}{\mathbb{Z}}
\newcommand{\Q}{\mathbb{Q}}
\newcommand{\K}{\mathcal{K}}
\newcommand{\B}{\mathcal{B}}
\newcommand{\R}{\mathbb{R}}
\newcommand{\PR}{\mathbb{P}}
\newcommand{\ESP}{\mathbb{E}}

\newcommand{\stas}{\mathcal{S}t\alpha\mathcal{S}}

\newtheorem{thm}{Theorem}[section]
\newtheorem{lemma}{Lemma}[section]
\newtheorem{proposition}{Proposition}[section]
\newtheorem{corollary}{Corollary}[section]
\newtheorem{definition}{Definition}[section]
\newtheorem{rem}{\textbf{Remark}}[section]

\def\og{\leavevmode\raise.3ex\hbox{$\scriptscriptstyle\langle\!\langle$~}}
\def\fg{\leavevmode\raise.3ex\hbox{~$\!\scriptscriptstyle\,\rangle\!\rangle$}}

\let\vers=\rightarrow

\def\ce{{\cal C}}

\def\o{\omega}
\def\eps{\epsilon}
\def\al{\alpha}
\def\ga{\gamma}
\def\th{\theta}

\def\nodiv{\mathrel{\mathchoice{\not|}{\not|}
{\kern-.2em\not\kern.2em|}{\kern-.2em\not\kern.2em|}}}

\def\ii{+\infty}

\def\zmat{\mathop{\raise 0.1mm\hbox{\bf Z}}\nolimits}

\newcommand{\ind}{\mathds{1}}

\begin{document}

\title{Linear Multifractional Stable Motion: fine path properties}

\author{Antoine Ayache \footnote{Corresponding author} \\UMR CNRS 8524, Laboratoire Paul Painlev\'e, B\^at. M2\\
  Universit\'e Lille 1\\ 59655 Villeneuve d'Ascq Cedex, France\\
E-mail: \texttt{Antoine.Ayache@math.univ-lille1.fr}\\
\ 
\and
 Julien Hamonier\\ UMR CNRS 8524, Laboratoire Paul Painlev\'e, B\^at. M2,\\
  Universit\'e Lille 1,\\ 
  59655 Villeneuve d'Ascq Cedex, France\\
  \&\\
  FR CNRS 2956,  LAMAV, \\
  Institut des Sciences et Techniques de Valenciennes,\\
  Universit\'e de Valenciennes et du Hainaut Cambr\'esis, \\
  F-59313 - Valenciennes Cedex 9, France\\
E-mail: \texttt{Julien.Hamonier@univ-valenciennes.fr}
}

\maketitle

\begin{abstract}
Since at least one decade, there is a considerable interest in the study of applied and theoretical issues related to Multifractional random models. Yet, only a few results about them, are known in the framework of heavy-tailed stable distributions; in the latter framework, a paradigmatic example of such models, is Linear Multifractional Stable Motion (LMSM), denoted by $\{Y(t):t\in\R\}$. It has been introduced by Stoev and Taqqu in~\cite{stoev2004stochastic,stoev2005path}, by substituting to the constant Hurst parameter of 
a classical Linear Fractional Stable Motion (LFSM), a deterministic function $H(\cdot)$ depending on the time variable $t$; we always suppose $H(\cdot)$ to be continuous and with values in $(1/\al,1)$, also, in general we restrict its range to a compact interval. The main goal of our article is to make a comprehensive study of the local and asymptotic behavior of $\{Y(t):t\in\R\}$; 
to this end, one needs to derive fine path properties of $\{X(u,v) : (u,v)\in\R \times (1/\alpha,1)\}$, the field generating the latter process (i.e. one has $Y(t)=X(t,H(t))$ for all $t\in\R$). This leads us to introduce random wavelet series representations of $\{X(u,v) : (u,v)\in\R \times (1/\alpha,1)\}$ as well as of all its pathwise partial derivatives of any order with respect to $v$. Then our strategy 
consists in using wavelet methods which are, more or less, reminiscent of those in \cite{ayache2009linear,ayache2005asymptotic}. Among other things, we solve a conjecture of Stoev and Taqqu (see Remark~1 on page 166 in \cite{stoev2005path}), concerning the existence for LMSM of a modification (in other words, a version) with almost surely continuous paths; moreover we significantly improve Theorem~4.1 in \cite{stoev2005path}, which provides some bounds for the local H\"older exponent (in other words, the uniform pointwise H\"older exponent) of LMSM: namely, we obtain a quasi-optimal global modulus of continuity for it, and also an optimal local one. It is worth noticing that,
even in the quite classical case of LFSM, the latter optimal local modulus of continuity provides a new result which was unknown so far. 
\end{abstract}

\medskip

      {\it Running head}:  Linear Multifractional Stable Motion  \\




\section{Introduction}
\label{sec:intro}
Since at least one decade, there is a considerable interest in the study of applied and theoretical issues related to Multifractional random models (among many other references on this topic, one can for instance see \cite{ayache2005multifractional,ayache2007wavelet,ayache2011multiparameter,roux1997elliptic,Bi,BiP,BiPP,dozzi2011,
falconer2002tangent,falconer2003local,falconer2009localizable,falconer2009multifractional,HLS12,lacaux04,Lope,LRMT11,meerschaert2008local,peltier1995multifractional,stoev2004stochastic,stoev2005path,stoev2006rich,Sur}). These fractal nonstationary increments stochastic processes/fields, are natural extensions of the well-known Fractional Brownian Motion (FBM, for brevity); they have a more rich path behavior than it and they offer a larger spectrum of applicability, because their local properties, typically the index governing self-similarity as well as the degree of path roughness, can be controlled via a nonconstant functional Hurst parameter and thus are allowed to change with location. In the Gaussian case, and more generally when all their moments are finite, many results concerning path behavior of such random models have been derived in the literature; yet, much less is known about it, in the framework of heavy-tailed stable distributions. A paradigmatic example of a Multifractional Process in 
such a setting, is the so called Linear Multifractional Stable Motion (LMSM, for brevity), which was introduced by Stoev and Taqqu in~\cite{stoev2004stochastic,stoev2005path}; according to these two authors (see page 1086 in~\cite{stoev2004stochastic}): "a LMSM model is a good candidate to adequately describe some features of traffic traces on telecommunication networks, typically changes in operating regimes and burstiness (the presence of rare but extremely busy periods of activity)".

In order to precisely define LMSM, first, we need to fix some notations to be used throughout the article.
\begin{itemize}
\item Recall that heaviness of the tail of a stable distribution is governed by a constant parameter belonging to the open interval $(0,2)$, usually denoted by $\al$; the smaller $\al$ is, the more heavy is the tail. In the present article, we always assume that $\al\in (1,2)$, since it has been shown in \cite{stoev2004stochastic}, that the latter assumption is actually a necessary condition for the paths of LMSM to be, with probability $1$, continuous functions.
\item $H(\cdot)$ denotes an arbitrary deterministic continuous function defined on the real line and with values in an arbitrary fixed compact interval $[\underline{H},\overline{H}]\subset(1/\al,1)$; similarly to the constant Hurst parameter of FBM, this function will be an essential parameter for LMSM.
\item $Z_{\alpha}(ds)$ is an independently scattered strictly $\al$ stable ($\stas$) random measure on $\R$, with Lebesgue measure as its control measure and an arbitrary Borel function $\beta(\cdot):\R\rightarrow [-1,1]$ as its skewness intensity. Many information on such random measures and the corresponding stochastic integrals can be found in the book \cite{SamTaq}.
\end{itemize}
LMSM's are generated by the $\stas$ random field $\widetilde{X}=\{\widetilde{X}(u,v) : (u,v)\in\R \times (1/\alpha,1)\}$, defined for all $(u,v)$ as the stochastic integral,
\begin{equation}\label{def:chpX}
\widetilde{X}(u,v)=\int_{\R} \Big\{ (u-s)_+^{v-1/\alpha} - (-s)_+^{v-1/\alpha} \Big\} Z_{\alpha}(ds),
\end{equation}
with the convention that, for each real numbers $x$ and $\kappa$, 
\begin{equation}
\label{eq:adpat+}
(x)_+^{\kappa}:=
\left\{
\begin{array}{l}
x^{\kappa}, \mbox{if $x\in (0,+\infty)$},\\
\\
0, \mbox{if $x\in (-\infty,0]$}.
\end{array}
\right.
\end{equation}
Actually, $\widetilde{Y}=\{\widetilde{Y}(t) : t\in\R\}$, the LMSM of functional Hurst parameter $H(\cdot)$, is defined for every $t\in\R$, as,
\begin{equation}\label{def:LMSM}
\widetilde{Y}(t)=\widetilde{X}(t,H(t)).
\end{equation}
Observe that, assuming $\beta(\cdot)$ to be a constant, then for each fixed $v\in (1/\al,1)$, the process $\widetilde{X}(\cdot,v):=\{\widetilde{X}(u,v) : (u,v)\in\R\}$ is the usual Linear Fractional Stable Motion (LFSM, for brevity) of Hurst parameter $v$; therefore LMSM reduces to the latter process when one also assumes $H(\cdot)$ to be a constant. We note in passing that LFSM and Harmonisable Fractional Stable Motion (HFSM, for brevity) are two very classical self-similar stable processes with stationary increments; they are considered to be the most two natural extensions of FBM, to the setting of heavy-tailed distributions. Also, we note that in contrast with moving average and harmonisable representations of FBM in the Gaussian framework, path behavior of LFSM is considerably more irregular and more complex than that of HFSM; we refer to \cite{SamTaq,Tak89,EmMa} for a detailed presentation of the latter two process, as well as other classical examples of stable processes. Before ending this paragraph, let us mention that an Harmonisable Multifractional Stable Process, which extends HFSM and thus behaves very differently from LMSM, has been quite recently introduced in \cite{dozzi2011}.

The main goal of our paper is to make a comprehensive study of the local and asymptotic behavior of LMSM, under the quite general condition that its parameter $H(\cdot)$ is an arbitrary deterministic continuous function with values in an arbitrary fixed compact interval $[\underline{H},\overline{H}]\subset(1/\al,1)$; this study mainly relies on wavelet methods which are, more or less, reminiscent of those in \cite{ayache2009linear,ayache2005asymptotic}. As we will explain it more precisely very soon, among other things, we significantly improve two earlier results of Stoev and Taqqu \cite{stoev2005path}, concerning continuity and path behavior of LMSM. Also, it is worth noting that, even in the quite classical case of Linear Fractional Stable Motion (LFSM) (in other words, in the particular case where the functional parameter $H(\cdot)$ of LMSM is a constant), the optimal lower bound of the power of the logarithmic factor in a local modulus of continuity, was unknown so far; Corollary~\ref{cor:locmodcont1Y}~and Theorem~\ref{locoptim:modcont} in our article, show that, in the more general case of LMSM, this optimal lower bound is in fact $1/\al$. 

Let us now give the precise statements of the two results in \cite{stoev2005path}, we have just mentioned.
\begin{enumerate}
\item {\em {\bf Theorem~3.2 in \cite{stoev2005path}} (the existence for LMSM of a modification (in other words, a version) whose paths are, with probability $1$, H\"older continuous functions). Let $I'\subset I$ be two arbitrary nonempty bounded intervals of the real line, which are respectively closed and open; suppose that $1/\al<H(t)<1$, $t\in I$ and that, for all for $t',t''\in I$,
\begin{equation}
\label{eq1:stoevtaqquH}
\big|H(t')-H(t'')\big|\le c|t'-t''|^\rho,\quad\mbox{with $1/\al<\rho$,}
\end{equation}
where $c>0$ does not depend on $t'$ and $t''$. Then, the LMSM $\{\widetilde{Y}(t):t\in\R\}$ has a modification $\{Y(t):t\in\R\}$, whose paths are, with probability $1$, continuous functions on $I$; moreover, they are H\"older functions on $I'$, with a uniform H\"older exponent (see (\ref{eq1:ppintro})) $\rho_{Y}^{\mbox{{\tiny unif}}}\big(I'\big)$, satisfying,
$$
\rho_{Y}^{\mbox{{\tiny unif}}}\big(I'\big)\ge \left(\rho\wedge\min_{t\in I'} H(t)\right)-1/\al.
$$
}
\item {\em {\bf Theorem~4.1 in \cite{stoev2005path}} (local H\"older exponent (in other words, uniform pointwise H\"older exponent) of LMSM). Assume that $H(\cdot)$ is continuous, with values in $(1/\al,1)$ and satisfies, $\rho_{H}^{\mbox{{\tiny unif}}}(t)>1/\al$ for all $t\in\R$, 
where $\rho_{H}^{\mbox{{\tiny unif}}}(t)$ denotes the local H\"older exponent (see (\ref{eq2:ppintro})) of $H(\cdot)$ at $t$. Then, $\rho_{Y}^{\mbox{{\tiny unif}}}(t_0)$, the local H\"older exponent of 
the LMSM $\{Y(t):t\in\R\}$ at an arbitrary point $t_0\ne 0$, can be almost surely bounded, in the following way: 
\begin{equation}
\label{eq5:stoevtaqquLHE}
\rho_{H}^{\mbox{{\tiny unif}}}(t_0)\wedge H(t_0)-1/\al\le \rho_{Y}^{\mbox{{\tiny unif}}}(t_0)\le \rho_H(t_0)\wedge H(t_0),
\end{equation}
where $\rho_H(t_0)$ denotes the pointwise H\"older exponent at $t_0$ (see e.g. Definition~4.1 in \cite{stoev2005path}) of the function $H(\cdot)$.
}
\end{enumerate}
In \cite{stoev2005path}, the proof of the first one of these two theorems, and that of the first inequality in (\ref{eq5:stoevtaqquLHE}), mainly rely on the strong version of the Kolmogorov's continuity criterion (see, for example, Theorem~3.3.16 in \cite{stroock1993}); while, the main three ingredients of the proof given in the latter article, for the second inequality in (\ref{eq5:stoevtaqquLHE}), are the inequality $\rho_{Y}^{\mbox{{\tiny unif}}}(t_0)\le \rho_{Y}(t_0)$ and Relations (4.11) and (4.12) in \cite{stoev2005path}. Using a different strategy, namely wavelet methods which are, more or less, reminiscent of those in \cite{ayache2009linear,ayache2005asymptotic}, in our present work, we have been able to improve Theorems~3.2~and~4.1 in \cite{stoev2005path}. More precisely:
\begin{enumerate}
\item The condition (\ref{eq1:stoevtaqquH}) seems to be too strong if one is only interested in the existence of a modification of LMSM with almost surely continuous paths; namely, in their Remark~1 on page 166 in \cite{stoev2005path}, Stoev and Taqqu have conjectured that such a modification should exist as long as $H(\cdot)$ is a continuous function with values in $(1/\al,1)$; the latter conjecture is solved in our article. To do so, we construct $X=\{X(u,v) : (u,v)\in\R \times (1/\alpha,1)\}$ a modification with almost surely continuous paths, of the field $\{\widetilde{X}(u,v) : (u,v)\in\R \times (1/\alpha,1)\}$ which generates LMSM's; in fact $\{X(u,v) : (u,v)\in\R \times (1/\alpha,1)\}$ is obtained as a random series of functions, resulting from the decomposition of the kernel in (\ref{def:chpX}) into a Daubechies wavelet basis (see Theorem~\ref{TWSE}). Thus, denoting by $\{Y(t): t\in\R\}$ the modification of LMSM defined for each $t\in\R$, as $Y(t):=X(t,H(t))$, it is clear that the paths of the process $\{Y(t): t\in\R\}$ are continuous with probability $1$, as long as $H(\cdot)$ is a continuous function on the real line and with values in $(1/\al,1)$; observe that at this stage, we do not need to restrict the range of $H(\cdot)$ to the compact interval $[\underline{H},\overline{H}]$.
\item Theorem~\ref{theo:lheLMSM} in our article shows that, almost surely, for any $t_0\in\R$ satisfying $\rho_{H}^{\mbox{{\tiny unif}}}(t_0)>1/\al$,
one has, $\rho_{Y}^{\mbox{{\tiny unif}}}(t_0)=H(t_0)-1/\al$. Observe that the exceptional negligible event on which the latter equality fails to be true, actually does not depend on $t_0$. Also observe that this equality remains valid even in the case where $t_0=0$.
\end{enumerate}


The remaining of the paper is structured in the following way. Section~\ref{sec:wavelets} is devoted to the construction of the modification $\{X(u,v) : (u,v)\in\R \times (1/\alpha,1)\}$ of the field $\{\widetilde{X}(u,v) : (u,v)\in\R \times (1/\alpha,1)\}$ which generates LMSM's; as we have already pointed out, the latter modification
is in fact a random series of functions, resulting from the decomposition of the kernel in (\ref{def:chpX}) into a Daubechies wavelet basis. In Section~\ref{sec:Holdsp}, we show that this series and all its term by term pathwise partial derivatives of any order with respect to $v$, are convergent in a very strong sense: with probability 1, in the space ${\cal E}_{\gamma} (a,b,M):=\ce^{1}\big( [a,b], \ce^{\gamma}([-M,M],\R) \big)$, where the real numbers $M>0$, $0<1/\al<a<b<1$ and $0\le \ga<a-1/\al$ are arbitrary and fixed, and where $\ce^{\lambda}(I,\mathbb{B})$ denotes the space of the $\lambda$-H\"older functions defined on an interval $I$ and with values in a Banach space $\mathbb{B}$. Notice that an important consequence of the latter result is that, for each $q\in\Z_+$, a typical path of the field $\{(\partial_v^q X) (u,v) : (u,v)\in \R \times (1/\alpha,1)\}$ belongs to ${\cal E}_{\gamma} (a,b,M)$; thus, not only such a path is a continuous function but also it has much better properties. In Section~\ref{sec:Xpath}, fine path properties of the field $\{(\partial_v^q X) (u,v) : (u,v)\in \R \times (1/\alpha,1)\}$, are derived thanks to wavelet methods; namely we determine a global modulus of continuity on the rectangle $[-M,M]\times [a,b]$, also we give an upper bound for $\big|(\partial_v^q X) (u,v)\big|$, on the domain $(u,v)\in\R\times [a,b]$. The latter two results are used in Section~\ref{subsec:modcont}, in order to obtain global and local moduli of continuity for the LMSM $\{Y(t): t\in\R\}$. The optimality of some of these moduli of continuity is discussed in Sections~\ref{subsec:quasiopt}~and~\ref{subsec:optlocmod}; under some H\"older conditions on $H(\cdot)$, it turns out that the global one is quasi-optimal (it provides, up to a logarithmic factor, a sharp estimate of the behavior of 
$\{Y(t): t\in\R\}$, on an arbitrary fixed compact interval) and the local one is optimal (it provides, without any logarithmic gap, a sharp estimate of the behavior of $\{Y(t): t\in\R\}$ on a neighborhood of an arbitrary fixed point). In Section~\ref{subsec:lheLMSM}, by making use of the quasi-optimality of the global modulus of continuity of LMSM, we determine its local H\"older exponent. Finally, some technical lemmas as well as their proofs are given in Section~\ref{sec:appendix} (the Appendix).

\section{Wavelet series representation of the field generating LMSM's}
\label{sec:wavelets}

Let $\widetilde{X}=\{\widetilde{X}(u,v) : (u,v)\in\R \times (1/\alpha,1)\}$ be the $\stas$ stochastic field introduced in (\ref{def:chpX}), the goal of this section 
is to construct a modification of $\widetilde{X}$, denoted by $X$, which is defined as a random wavelet series. We note in passing that, random wavelet series representations of LFSM and other self-similar stable fields with stationary increments, have been introduced in~\cite{Dury01}.

First, we need to fix some notations related to wavelets that will be extensively used throughout the article.
\begin{itemize}
\item The real-valued function $\psi$ defined on the real line, denotes a 3 times continuously differentiable compactly supported Daubechies mother wavelet \cite{Dau92,Meyer90,Meyer92};
 observe that $\psi$ has $Q\ge 15 $ vanishing moments i.e.:
\begin{equation}
\label{eq:vanmopsi}
\int_{\R}t^m\psi(t)  dt = 0, \mbox{ for all $m=0,\ldots,Q-1$, and }\int_{\R} t^Q  \psi(t)dt \neq 0.
\end{equation}
The fact that $\psi$ is a compactly supported function will play a crucial role; for the sake of convenience,
we assume that $R$ is a fixed real number strictly bigger than $1$, such that
\begin{equation}
\label{eq:supp-psi-ant}
\mbox{supp}\,\psi\subseteq [-R,R].
\end{equation}
\item The real-valued function $\Psi$ is defined for all $(x,v)\in \R\times (1/\alpha,1)$ as, 
\begin{equation}\label{PSI}
\Psi(x,v):=\int_{\R} (s)_{+}^{v-1/\alpha} \psi(x-s)ds=\int_{\R} (x-s)_{+}^{v-1/\alpha} \psi(s) ds;
\end{equation}
recall that the definition of $(\cdot)_{+}^{v-1/\alpha}$ is given in (\ref{eq:adpat+}).
Denoting by $\Gamma$ the usual Gamma function: 
$$
\Gamma(u):=\int_0^{+\infty}  t^{u-1}\,e^{-t}dt, 
\mbox{ for all $u\in (0,+\infty)$,}
$$
 and denoting, for each fixed $v$, by 
$\widehat{\Psi}(\cdot,v)$ the Fourier transform of the function $\Psi(\cdot,v)$: 
$$
\widehat{\Psi}(\xi,v):=\int_{\R}e^{-i\xi x}\Psi(x,v)dx, \mbox{ for all $\xi\in\R$,} 
$$
one has,
\begin{equation} \label{FPSI}
\widehat{\Psi}(\xi,v)= \Gamma(v+1-1/\alpha) \frac{e^{-i\textrm{sgn}(\xi)(v+1-1/\alpha)\frac{\pi}{2}}\widehat{\psi}(\xi) }{|\xi|^{v+1-1/\alpha}},\mbox{ for all $\xi\in\R\setminus\{0\}$;} 
\end{equation}
the latter equality can be obtained by using a result in \cite{Samko93} concerning Fourier transforms of left-sided fractional derivatives.


\item $\lbrace \epsilon_{j,k} : (j,k)\in\Z^2 \rbrace$ is the sequence of the real-valued $\stas$ random variables defined as,
\begin{equation}
\epsilon_{j,k} := 2^{j/\alpha}\int_{\R}  \psi(2^js-k) Z_{\alpha}(ds). \label{def:ejk}
\end{equation}

\end{itemize}

Now we are in position to state the main result of this section.

\begin{thm}
\label{TWSE}
Let $\Psi$ be the function defined in (\ref{PSI}), let $\lbrace \epsilon_{j,k} : (j,k)\in\Z^2 \rbrace$ be the sequence of the real-valued $\stas$
random variables defined in (\ref{def:ejk}), and let $\Omega_{0}^*$ be the event of probability $1$ introduced in Lemma~\ref{omega0} below. The following two results hold.
\begin{itemize}
\item[(i)] For all fixed $\o\in\Omega_{0}^*$ and $(u,v)\in\R\times (1/\al,1)$, one has 
\begin{equation}
\label{eq1:absconc}
\sum_{(j,k)\in\Z^2} 2^{-jv}\big|\epsilon_{j,k}(\o)\big|\big|\Psi (2^j u-k,v)-\Psi (-k,v)\big|<\infty.
\end{equation}
Therefore, the series of real numbers: 
\begin{equation}
\label{eqmain:WSE}
\sum_{(j,k)\in\Z^2} 2^{-j v}\epsilon_{j,k}(\o)\big (\Psi (2^j u-k,v)-\Psi (-k,v)\big),
\end{equation}
converges to a finite limit which does not depend on the way the terms of the series are ordered; this limit is denoted by $X(u,v,\o)$.
Moreover for each $\o\notin \Omega_0^*$ and every $(u,v)\in\R\times (1/\al,1)$, one sets  $X(u,v,\o)=0$.
\item[(ii)] The field $\{X(u,v) : (u,v)\in\R \times (1/\alpha,1)\}$ is a modification of the $\stas$ field $\{\widetilde{X}(u,v) : (u,v)\in\R \times (1/\alpha,1)\}$
defined in (\ref{def:chpX}).
\end{itemize}
\end{thm}

In order to prove Theorem~\ref{TWSE}, we need some preliminary results.

\begin{rem}
\label{prop:gcjk}
\begin{itemize}
\item[(i)] $\|\epsilon_{j,k}\|_{\alpha}$,  the scale parameter of $\epsilon_{j,k}$, does not depend on $(j,k)$, since classical computations, allow to show that, 
\begin{equation}
\label{eq:spepsjk}
\|\epsilon_{j,k}\|_{\alpha} =  \|\epsilon_{0,0}\|_{\alpha}=\Bigg\{ \int_{\R} |\psi(t)|^{\alpha} dt \Bigg\}^{1/\alpha}.
\end{equation}
\item[(ii)] The skewness parameter of $\epsilon_{j,k}$, is denoted by $\beta_{j,k}$ and is given by,
$$
\beta_{j,k} =  \|\epsilon_{0,0}\|_{\alpha}^{-\alpha}\int_{\R}\psi^{<\alpha >}(x)\beta(2^{-j}x+2^{-j}k) dx , \nonumber
$$
where $z^{<\alpha >}:= |z|^{\alpha} \textrm{sgn}(z)$ for all $z\in\R$, and where $\beta(\cdot)$ is the skewness intensity function of the 
$\stas$ measure $Z_{\alpha}(ds)$; notice that, when the latter function is a constant, then the random variables $\eps_{j,k}$ become identically distributed, since, not 
only they have the same scale parameter, but also the same skewness parameter.
\item[(iii)] Property~1.2.15 on page 16 in \cite{SamTaq}, as well as the fact that $\|\epsilon_{j,k}\|_{\alpha}$ does not vanish and does not depend on $(j,k)$, imply 
that there exist two constants $0<c'\le c''$ non depending on $(j,k)$, such that, one has for all real number $x\ge 1$,
\begin{equation}
\label{eq4:l3optim}
c' x^{-\al} \le \PR\big(|\eps_{j,k}|>x\big)\le c''x^{-\al}.
\end{equation}
\item[(iv)] In view of (\ref{def:ejk}), (\ref{eq:supp-psi-ant}) and the fact that $Z_{\alpha}(ds)$ is independently scattered, for each fixed integers $p>2R$ and $j\in\Z$, one has that $\lbrace \epsilon_{j,pq} : q\in \Z \rbrace$ is a sequence of independent random variables.
\end{itemize}
\end{rem}


The following lemma, which has been derived in \cite{ayache2009linear}, gives rather sharp estimates of the asymptotic behavior of the sequence 
$\big\{|\epsilon_{j,k}|\,:\, (j,k)\in\Z^2\big\}$. It can be proved by showing that for every fixed real number $\eta>0$, one has,
$$
\ESP\left(\sum_{(j,k)\in\Z^2}\ind_{\big\{|\epsilon_{j,k}|>(1+|j|)^{1/\al+\eta}(1+|k|)^{1/\al}\log^{1/\al+\eta}(2+|k|)\big\}}\right)<\infty;
$$
the latter fact, easily results from the second inequality in (\ref{eq4:l3optim}).

\begin{lemma}\label{omega0}\cite{ayache2009linear}
There exists an event of probability 1, denoted by $\Omega_0^{*}$, such that for every fixed real number $\eta >0$, one has, for all $\omega \in \Omega_0^{\ast}$ and for each $(j,k)\in\Z^2$,
\begin{equation}
\label{eq1:omega0}
\big| \epsilon_{j,k}(\omega) \big| \leq C(\omega) \big(1+|j| \big)^{1/\alpha+\eta}  \big(1+|k| \big)^{1/\alpha} \log^{1/\alpha+\eta}\big(2+ |k|\big)
\le C'(\omega) \big(3+|j| \big)^{1/\alpha+\eta}\big(3+|k| \big)^{1/\alpha+\eta}, 
\end{equation}
where $C$ and $C'$ are two positive and finite random variables only depending on $\eta$.
\end{lemma}

The following proposition, which shows that the function $\Psi$ and its partial derivatives of any order, have nice smoothness and localization properties, will also play an important role throughout our article.

\begin{proposition}\label{regulpsi}
The function $\Psi$ satisfies the following two properties.
\begin{itemize}
\item[(i)] For all $(p,q) \in \{ 0,1,2,3 \}\times\Z_+$ and $(x,v)\in\R\times (1/\al,1)$, the partial derivative $(\partial_x^p \partial_v^q \Psi)(x,v)$ exists and is given by,
\begin{equation}
\label{DPSI}
(\partial_x^p \partial_v^q \Psi)(x,v) 
= \int_{\R} (s)_{+}^{v-1/\alpha} \log^q((s)_{+})\psi^{(p)}(x-s)ds= \int_{\R} (x-s)_+^{v-1/\alpha} \log^q((x-s)_+) \psi^{(p)}(s) ds,
\end{equation}
where $\psi^{(p)}$ is the derivative of $\psi$ of order $p$ and $0\log^q (0):=0$. Moreover the function $\partial_x^p \partial_v^q \Psi$ is continuous 
on $\R\times (1/\al,1)$.
\item[(ii)] For each $(p,q) \in \{ 0,1,2,3 \}\times\Z_+$ and for every real numbers
  $a,b$ satisfying $1>b>a>1/\alpha$, the function $\partial_x^p \partial_v^q \Psi$ is well-localized in the variable $x$ uniformly in the variable $v\in [a,b]$; namely one has
\begin{equation}
\label{localisation}
\sup_{(x,v)\in\R\times [a,b]} (3+|x|)^2 \big|(\partial_x^p \partial_v^q \Psi)(x,v)\big|< \infty.
\end{equation}
\end{itemize}
\end{proposition}

\begin{proof}[Proof of Proposition \ref{regulpsi}]
Let us first show that Part $(i)$ holds. In view of (\ref{PSI}), the function $\Psi$ can be expressed, for all 
$(x,v)\in\R\times (1/\alpha,1)$ as,
\begin{equation*}
\label{eq0:regulpsi-ant-j}
\Psi(x,v)= \int_{\R} L(x,v,s)\,ds.
\end{equation*}
where $L(x,v,s):=(s)_{+}^{v-1/\alpha} \psi(x-s)$. Also observe that for all $(p,q)\in \{0,\ldots, 3\}\times\Z_+$ and $(x,v,s)\in\R\times (1/\alpha,1)\times\R$, the partial derivative $(\partial_x^p \partial_v^q L)(x,v,s)$ exists and is given by
\begin{equation}
\label{eq1:regulpsi-ant}
(\partial_x^p \partial_v^q L)(x,v,s)=(s)_{+}^{v-1/\alpha} \log^q((s)_{+})\psi^{(p)}(x-s).
\end{equation}
Therefore, to show that the partial derivative $(\partial_x^p \partial_v^q \Psi)(x,v)$ exists 
and is given by (\ref{DPSI}), it is sufficient to prove that for all real numbers $M$, $a$ and
$b$, satisfying 
\begin{equation}
\label{eq2:regulpsi-ant}
M>0 \mbox{ and } 1/\alpha<a<b<1, 
\end{equation}
one has
\begin{equation}
\label{eq3:regulpsi-ant}
\int_{\R}\sup_{(x,v)\in [-M,M]\times [a,b]}\big |(\partial_x^p \partial_v^q L)(x,v,s)\big |\,ds <\infty.
\end{equation}
This is true, since Relations (\ref{eq1:regulpsi-ant}), (\ref{eq:supp-psi-ant}) and (\ref{eq2:regulpsi-ant}), imply that
$$
\int_{\R}\sup_{(x,v)\in [-M,M]\times [a,b]}\big |(\partial_x^p \partial_v^q L)(x,v,s)\big |ds\le 
\|\psi^{(p)}\|_{L^\infty (\R)}\int_{-M-R}^{M+R}\Big ((s)_{+}^{a-1/\alpha}+(s)_{+}^{b-1/\alpha}\Big) 
\big|\log((s)_{+})\big|^q ds < \infty.
$$
Finally, observe that it follows from (\ref{DPSI}), (\ref{eq1:regulpsi-ant}), (\ref{eq3:regulpsi-ant}) and the dominated convergence Therorem, that for all $(p,q)\in \{0,\ldots, 3\}\times\Z_+$, the function $\partial_x^p \partial_v^q \Psi$ is continuous over $\R\times (1/\alpha, 1)$.

Let us show that Part $(ii)$ of the proposition holds. Relations (\ref{eq:supp-psi-ant}) and (\ref{DPSI}), imply that for all $(p,q)\in \{0,\ldots, 3\}\times\Z_+$ and for each $(x,v)\in (-\infty, -R)\times (1/\al, 1)$, one has
\begin{equation}
\label{eq4:regulpsi-ant}
(\partial_x^p \partial_v^q \Psi)(x,v)=0.
\end{equation}
Combining (\ref{eq4:regulpsi-ant}) with the fact $\partial_x^p \partial_v^q \Psi$ is a continuous function over the compact set $[-R,2R]\times [a,b]$, it follows that 
$$
\sup_{(x,v)\in (-\infty,2R]\times [a,b]} (3+|x|)^2 \big|(\partial_x^p \partial_v^q \Psi)(x,v)\big|< \infty.
$$
Therefore, it remains to show that
\begin{equation}
\label{eq5:regulpsi-ant}
\sup_{(x,v)\in (2R,+\infty)\times [a,b]} (3+x)^2 \big|(\partial_x^p \partial_v^q \Psi)(x,v)\big|< \infty.
\end{equation}
In view of (\ref{DPSI}) and (\ref{eq:supp-psi-ant}), one has for each $(x,v)\in (2R,+\infty)\times [a,b]$,
$$
(\partial_x^p \partial_v^q \Psi)(x,v)=\int_{-R}^{R} K_q (x,v,s)\psi^{(p)}(s)\,ds,
$$
where $K_q (x,v,s):=(x-s)^{v-1/\al}\log^q (x-s)$. For each $l\in\{1,2,3\}$ and real number $s$, one sets $\psi^{(p-l)}(s)=\int_{-\infty} ^s \psi^{(p+1-l)}(t)\,dt$; observe that, in view of (\ref{eq:vanmopsi}) and (\ref{eq:supp-psi-ant}), the 
supports of the latter three functions are included in $[-R,R]$. Thus integrating three times by parts, one 
gets that, 
\begin{equation}
\label{eq6:regulpsi-ant}
(\partial_x^p \partial_v^q \Psi)(x,v)=-\int_{-R}^{R} (\partial_{s}^3 K_q )(x,v,s)\psi^{(p-3)}(s)\,ds.
\end{equation}
Next standard computations, allow to show that there is a constant $c_{q,\al}>0$, only depending on $q$ and $\al$, such
that for all $(x,v,s)\in (2R,+\infty)\times [a,b]\times [-R,R]$, one has,
\begin{equation}
\label{eq7:regulpsi-ant}
\big |(\partial_{s}^3 K_q )(x,v,s)\big |\le c_{q,\al} (x-s)^{-2}\le  4 c_{q,\al} x^{-2}.
\end{equation}
Finally, putting together (\ref{eq6:regulpsi-ant}) and (\ref{eq7:regulpsi-ant}), one obtains (\ref{eq5:regulpsi-ant}).
\end{proof}

Now we are in position to prove Theorem~\ref{TWSE}.

\begin{proof}[Proof of Theorem~\ref{TWSE} Part $(i)$] Let $\o\in\Omega_{0}^*$ and $(u,v)\in\R\times (1/\al,1)$ be arbitrary and fixed. Also,
we assume that $\eta$ is an arbitrarily small fixed positive real number.
By using the triangle inequality, (\ref{localisation}) (in which one takes $p=q=0$ and $a,b$ such that $v\in [a,b]$), and (\ref{eq1:omega0}), it follows that for all fixed $j\in\N$,
\begin{eqnarray}
\label{eq2:absconc}
&& \sum_{k\in\Z} \big|\eps_{j,k}(\o)\big|\big|\Psi (2^j u-k,v)-\Psi (-k,v)\big|\nonumber\\
&& \le C_1 (\o) \big(3+j\big)^{1/\al+\eta}\sum_{k\in\Z} \left(\frac{\big(3+|k|\big)^{1/\al+\eta}}{\big (3+|2^j u-k|\big)^{2}}+\frac{\big(3+|k|\big)^{1/\al+\eta}}{\big (3+|k|\big)^{2}}\right)\nonumber\\
&& \le C_2 (\o) \big(3+j\big)^{1/\al+\eta}\big(3+2^j |u|\big)^{1/\al+\eta}\sum_{k\in\Z} \left(\frac{\big(3+|k|\big)^{1/\al+\eta}}{\big (3+|2^j u-[2^j u]-k|\big)^{2}}+\frac{\big(3+|k|\big)^{1/\al+\eta}}{\big (3+|k|\big)^{2}}\right),\nonumber\\
\end{eqnarray}
where $[2^j u]$ denotes the integer part of $2^j u$ and where $C_1(\o)$ and $C_2(\o)$ are two finite constants non depending on $j$ and $u$. Then, noticing that,
\begin{equation}
\label{eq3:absconc}
\sup_{x\in [0,1]}\left\{ \sum_{k\in\Z}\frac{\big(3+|k|\big)^{1/\al+\eta}}{\big (3+|x-k|\big)^{2}}\right\}\le \sum_{k\in\Z}\frac{\big(3+|k|\big)^{1/\al+\eta}}{\big (2+|k|\big)^{2}}<\infty,
\end{equation}
it follows from (\ref{eq2:absconc}) that 
\begin{equation}
\label{eq4:absconc}
\sum_{j\in\N} \sum_{k\in\Z} 2^{-jv}\big|\eps_{j,k}(\o)\big|\big|\Psi (2^j u-k,v)-\Psi (-k,v)\big|<\infty.
\end{equation}
Let us now prove that,
\begin{equation}
\label{eq5:absconc}
\sum_{j\in\Z_{-}} \sum_{k\in\Z} 2^{-jv}\big|\eps_{j,k}(\o)\big|\big|\Psi (2^j u-k,v)-\Psi (-k,v)\big|<\infty.
\end{equation}
Applying the Mean Value Theorem, one has for all $(j,k)\in\Z_{-}\times\Z$, 
\begin{equation}
\label{eq6:absconc}
\Psi (2^j u-k,v)-\Psi (-k,v)=2^{j} u (\partial_{x} \Psi)(\nu-k,v),
\end{equation}
where $\nu\in [-2^j|u|,2^j|u|]\subseteq[-|u|,|u|]$. Then putting together (\ref{eq6:absconc}), (\ref{eq1:omega0}) and (\ref{localisation}) (in which one takes $p=1$, $q=0$ and $a,b$ such that $v\in [a,b]$), one obtains that,
\begin{eqnarray}
\label{eq7:absconc}
&& \sum_{|k|\le |u|} \big|\eps_{j,k}(\o)\big|\big|\Psi (2^j u-k,v)-\Psi (-k,v)\big|\nonumber\\
&& \le C_3(\o)|u|\big (2 |u|+1\big)\big (3+|u|\big)^{1/\al +\eta}\left(\sup_{x\in\R} \big|(\partial_{x} \Psi)(x,v)\big|
\right) 2^{j} (3+|j|\big)^{1/\al +\eta}
\end{eqnarray}
and 
\begin{eqnarray}
\label{eq8:absconc}
&& \sum_{|k|>|u|} \big|\eps_{j,k}(\o)\big|\big|\Psi (2^j u-k,v)-\Psi (-k,v)\big|\nonumber\\
&& \le C_4(\o) |u|\big (3+|u|\big)^{1/\al +\eta}\left(\sum_{k\in\Z}\big(3+|k|\big)^{1/\al +\eta-2}\right) 2^{j}\big(3+|j|\big)^{1/\al +\eta},
\end{eqnarray}
where $C_3 (\o)$ and $C_4 (\o)$ are two positive finite constants non depending on $j$ and $u$. Next combining (\ref{eq7:absconc}) and (\ref{eq8:absconc}), with
the fact that $v\in (1/\al,1)$, one gets (\ref{eq5:absconc}). Finally (\ref{eq4:absconc}) and (\ref{eq5:absconc}) show that (\ref{eq1:absconc}) holds.
\end{proof}

\begin{proof}[Proof of Theorem~\ref{TWSE} Part $(ii)$] For all $(j,k) \in \Z^2$ and any $s \in \R$, we set
\begin{equation}\label{def:psijk}
\psi_{j,k}(s) = 2^{j/\alpha} \psi(2^j s-k),
\end{equation}
where $\psi$ is the Daubechies mother wavelet introduced at the very beginning of this section; observe that the sequence $\lbrace \psi_{j,k}: (j,k)\in\Z^2 \rbrace$  forms an unconditional basis of $L^{\alpha}(\R)$ and the sequence $\lbrace 2^{j(1/2-1/\alpha)}\psi_{j,k}: (j,k)\in\Z^2 \rbrace$ is an orthonormal basis of $L^{2}(\R)$ (see \cite{Meyer90,Meyer92}). Therefore, noticing that for any fixed $(u,v)\in\R\times(1/\alpha,1)$, the function $s \mapsto (u-s)_{+}^{v-1/\alpha} - (-s)_{+}^{v-1/\alpha}$ belongs to $L^{\alpha}(\R) \cap L^{2}(\R)$, it follows that,
\begin{equation}\label{noyau:wse}
(u-s)_{+}^{v-1/\alpha} - (-s)_{+}^{v-1/\alpha} = \sum_{j\in\Z} \sum_{k\in\Z} w_{j,k}(u,v) \psi_{j,k}(s),
\end{equation}
where
\begin{align}\label{coef:ker}
w_{j,k}(u,v) & := 2^{j(1-1/\alpha)} \int_{\R} \big\{ (u-s)_{+}^{v-1/\alpha} - (-s)_{+}^{v-1/\alpha}  \big\} \psi(2^js-k) ds \nonumber \\
& = 2^{-jv} \big\{ \Psi(2^ju-k,v) - \Psi(-k,v) \big\},
\end{align}
and where the convergence of the series, as a function of $s$, holds in $L^{\alpha}(\R)$ as well as in $L^{2}(\R)$; observe that the limit of the series does not depend on the way its terms are ordered. Next, using (\ref{noyau:wse}), (\ref{coef:ker}), (\ref{def:chpX}), a classical property of the stochastic integral $\int_{\R} \big(\cdot\big) Z_{\al}(ds)$, (\ref{def:psijk}) and (\ref{def:ejk}), we get that the random series 
$$
\sum_{(j,k)\in\Z^2} 2^{-j v}\epsilon_{j,k}\big (\Psi (2^j u-k,v)-\Psi (-k,v)\big),
$$
converges in probability to the random variable $\widetilde{X}(u,v)$; observe that the terms of the latter series can be 
ordered in an arbitrary way. Finally, combining the latter result with Part $(i)$ of Theorem~\ref{TWSE}, we obtain 
that the random variables $\widetilde{X}(u,v)$ and $X(u,v)$ are equal almost surely. 
\end{proof}

\section{Convergence of the wavelet series in H\"older spaces}
\label{sec:Holdsp}

The goal of this section is to show that when the terms of the series in (\ref{eqmain:WSE}), viewed as a random series of functions of the variable $(u,v)$, are ordered in an appropriate way, then not only this series
converges almost surely for every fixed $(u,v)\in\R\times (1/\al,1)$, but also, it is, as well as all its term by term pathwise partial derivatives of any order with 
respect to $v$, almost surely convergent in some H\"older spaces. Let us first precisely define these spaces.

\begin{definition}\label{def:EG}
Let $(\mathbb{B},\|\cdot \| )$ be a Banach space and $\K$ a subset of
$\R$. For every $\gamma \in [0,1]$, the Banach space of $\gamma$-H\"{o}lder
functions from $\K$ to $\mathbb{B}$, is denoted by
$\ce^{\gamma}(\K,\mathbb{B})$ and defined as,
\begin{equation}
\ce^{\gamma}(\K,\mathbb{B}):= \Big\{ f:\K \rightarrow \mathbb{B} : {\cal N}_{\gamma}(f)< \infty \Big\}, \nonumber
\end{equation}
where 
$$
{\cal N}_{\gamma}(f):= \sup_{x\in\K} \|f(x)\| + 
\sup_{x,y\in\K} \dfrac{ \|f(x)-f(y)\|}{|x-y|^{\gamma}},
$$ 
is the natural norm on this space. Notice that in the definition of ${\cal N}_{\gamma}(f)$, we assume that $0/0=0$. Also notice that
$\ce^{1}(\K,\mathbb{B})$ is usually called the space of the Lipschitz functions from $\K$ to $\mathbb{B}$.
\end{definition}

\begin{definition}
\label{de:spaceE}
Let $\gamma$, $M$, $a$ and $b$ be arbitrary and fixed real numbers satisfying $\gamma\in [0,1]$,
$M>0$ and $a < b$. We denote by ${\cal E}_{\gamma} (a,b,M)$, the Banach space
$$
{\cal E}_{\gamma} (a,b,M):=\ce^{1}\big( [a,b], \ce^{\gamma}([-M,M],\R) \big),
$$
of the Lipschitz functions defined on $[a,b]$ and with values in the H\"{o}lder space $\ce^{\gamma}([-M,M],\R)$.
Observe that each function $f$ belonging to ${\cal E}_{\gamma} (a,b,M)$, can be viewed as a bivariate real-valued function $(u,v)\mapsto f(u,v):=\big(f(v)\big)(u)$ 
 on the rectangle $[-M,M]\times [a,b]$; moreover, the natural norm on ${\cal E}_{\gamma} (a,b,M)$, is equivalent to the norm $|||\,.\,|||$ defined as,
\begin{align}
\label{eq1:normequiv}
|||f|||:=\sup_{(u,v) \in [-M,M]\times [a,b]} |f(u,v)|  & +\sup_{(u_1,u_2,v)\in [-M,M]^2\times  [a,b]}
  \frac{\big|(\Delta_{u_1-u_2}^{1,\cdot} f ) ( u_2,v)\big| } {|u_1-u_2|^\gamma} \nonumber\\
  & +\sup_{(u,v_1,v_2)\in [-M,M]\times  [a,b]^2}
  \frac{\big|( \Delta_{v_1-v_2}^{\cdot,1} f )( u,v_2)\big|} {|v_1-v_2|}\\
 & +\sup_{(u_1,u_2,v_1,v_2) \in [-M,M]^2\times [a,b]^2}\frac{\big|( \Delta_{(u_1-u_2,v_1-v_2)}^{1,1} f ) (u_2,v_2)\big|}{|u_1-u_2|^\gamma |v_1-v_2|},\nonumber
\end{align}
where, 
\begin{align}
\label{eq1:normequiv-diff}
& ( \Delta_{u_1-u_2}^{1,\cdot} f ) ( u_2,v):= f(u_1,v) - f( u_2,v),\nonumber\\
& ( \Delta_{v_1-v_2}^{\cdot,1} f )( u,v_2):= f(u,v_1) - f( u,v_2),\\
& (\Delta_{(u_1-u_2,v_1-v_2)}^{1,1} f ) (u_2,v_2):= f(u_1,v_1) - f(u_1,v_2) - f(u_2,v_1) + f(u_2,v_2).\nonumber
\end{align}
Notice that in (\ref{eq1:normequiv}), we assume that $0/0=0$.
\end{definition}

Now we are in position to state the main result of this section.
\begin{thm}
\label{TWSEbis}
We use the same notations as in Theorem~\ref{TWSE}. The following two results hold for all $\o\in\Omega_{0}^*$, the event of probability $1$ introduced in Lemma~\ref{omega0}.
\begin{itemize}
\item[(i)] For each fixed $u\in\R$, the function $X(u,\cdot,\omega): v\mapsto X (u,v,\omega)$ is infinitely differentiable over $(1/\al,1)$; its derivative of any order $q\in\Z_+$ at all $v\in (1/\al,1)$, is given by 
\begin{equation}
\label{eq1:TWSEbis}
(\partial_v^q X)(u,v,\omega)=\sum_{p=0}^q \dbinom{q}{p} \big( - \log 2 \big)^p \sum_{(j,k)\in\Z^2} j^p 2^{-jv}\epsilon_{j,k}(\omega)  \left (\big(\partial_v^{q-p}\Psi\big)(2^ju-k,v)-\big(\partial_v^{q-p}\Psi\big)(-k,v)\right),
\end{equation}
where, $0^0:=1$, for every fixed $(u,v)$ the series is absolutely convergent (its terms can therefore be ordered in an arbitrary way), and $\dbinom{q}{p}$ denotes the binomial coefficient $\frac{q!}{p!(q-p)!}$.
\item[(ii)] For each fixed $q\in\Z_+$ and $M,a,b\in\R$ satisfying $M>0$ and $1/\al <a<b<1$, the function $(\partial_v^q X)(\cdot,\cdot,\omega):(u,v)\mapsto (\partial_v^q X)(u,v,\o)$ belongs to the space ${\cal E}_{\gamma} (a,b,M)$ for all $\gamma\in [0,a-1/\al)$.
\end{itemize}
\end{thm}

The proof of Theorem~\ref{TWSEbis} mainly relies on the following proposition.

\begin{proposition}
\label{prop:cauchy}
Let $M$ be an arbitrary and fixed positive real number. For every $n\in\Z_+$, denote by 
$X_{M,n}=\big\{X_{M,n} (u,v)\,:\, (u,v)\in\R\times (1/\al,1)\big\}$ the $\stas$ random field defined for every $(u,v) \in \R\times (1/\al,1)$, as the finite sum,
\begin{equation}
\label{eq1:PWSEXN}
 X_{M,n}(u,v)= \sum_{(j,k)\in D_{M,n}} 2^{-jv} \epsilon_{j,k} \big( \Psi(2^ju-k,v) - \Psi(-k,v) \big),
\end{equation}
where 
\begin{equation}
\label{eq2:PWSEXN}
D_{M,n}:=\left\{(j,k)\in \Z^2 : |j| \leq n \mbox{ and } |k| \leq M2^{n+1}\right\}.
\end{equation} 
Then, the following three results hold.
\begin{itemize}
\item[(i)] For all fixed $\o\in\Omega$ (the underlying probability space) and $u\in\R$, the function $X_{M,n}(u,\cdot,\o):v\mapsto X_{M,n}(u,v,\o)$ is infinitely 
differentiable over $(1/\al,1)$; its derivative of any order $q\in\Z_+$ at any point $v\in (1/\al,1)$, is denoted by $(\partial_v^q X_{M,n})(u,v,\omega)$.
\item[(ii)] For all fixed $\o\in\Omega$, $q,n\in\Z_+$ and $a,b\in\R$ satisfying $1/\alpha < a < b <1$, the function $(\partial_v^q X_{M,n})(\cdot,\cdot,\omega)$
belongs to the Banach space ${\cal E}_{1} (a,b,M)$.
\item[(iii)] For each fixed $\omega\in\Omega_{0}^*$, $q\in\Z_{+}$, and $a,b,\ga\in\R$ satisfying $1/\alpha < a < b <1$ and $0\le\ga<a-1/\al$, $\big((\partial_v^q X_{M,n})(\cdot,\cdot,\omega)\big)_{n\in\Z_+}$ is a Cauchy sequence in the Banach space ${\cal E}_{\gamma} (a,b,M)$.
\end{itemize}
\end{proposition}

\begin{proof}[Proof of Proposition~\ref{prop:cauchy}] Parts $(i)$ and $(ii)$ of Proposition~\ref{prop:cauchy} are more or less straightforward consequences of Proposition~\ref{regulpsi}. In view of Definition~\ref{de:spaceE}, Part $(iii)$ of Proposition~\ref{prop:cauchy} results from the following four lemmas.
\end{proof}



\begin{lemma}\label{PXN1}
Let $M$, $a$ and $b$ be fixed real numbers satisfying $M>0$ and $1/\al<a<b<1$. For all fixed $q\in\Z_+$ and $\omega \in \Omega_0^{*}$, when $n$ goes to infinity,
\begin{equation}\label{eq1:PXN1}
\Big| \big( \partial_v^q X_{M,n+l} \big)(u,v,\omega) -  \big( \partial_v^q X_{M,n} \big)(u,v,\omega) \Big| 
\end{equation}
converges to $0$, uniformly in $(u,v)\in [-M,M]\times [a,b]$ and in $l\in\Z_+$.
\end{lemma}

\begin{lemma}\label{PXN2}
Let $M$, $a$, $b$ and $\ga$ be fixed real numbers satisfying $M>0$, $1/\al<a<b<1$ and $\gamma < a-1/\alpha$. For all fixed $q\in\Z_+$ and $\omega \in \Omega_0^{*}$, when $n$ goes to infinity,
\begin{equation}\label{eq1:PXN2}
\frac{ \Big| \big(\Delta_{u_1-u_2}^{1,\cdot} \big( \partial_v^q X_{M,n+l}\big)\big)(u_2,v,\omega) - \big(\Delta_{u_1-u_2}^{1,\cdot} \big( \partial_v^q X_{M,n}\big)\big) (u_2,v,\omega)\Big|}{ |u_1-u_2|^{\gamma} }  
\end{equation}
converges to $0$ uniformly in $(u_1,u_2,v)\in [-M,M]^2\times [a,b]$ and in $l\in\Z_+$.
\end{lemma}

\begin{lemma}\label{PXN3}
Let $M$, $a$ and $b$ be fixed real numbers satisfying $M>0$ and $1/\al<a<b<1$. For all fixed $q\in\Z_+$ and $\omega \in \Omega_0^{*}$, when $n$ goes to infinity,
\begin{equation}\label{eq1:PXN3}
\frac{ \Big| \big(\Delta_{v_1-v_2}^{\cdot,1} \big( \partial_v^q X_{M,n+l}\big)\big)(u,v_2,\omega) - \big(\Delta_{v_1-v_2}^{\cdot,1} \big( \partial_v^q X_{M,n}\big)\big)(u,v_2,\omega)\Big|}{ |v_1-v_2|} 
\end{equation}
converges to $0$ uniformly in $(u,v_1,v_2)\in [-M,M]\times [a,b]^2$ and in $l\in\Z_+$.
\end{lemma}

\begin{lemma}\label{PXN4}
Let $M$, $a$, $b$ and $\ga$ be fixed real numbers satisfying $M>0$, $1/\al<a<b<1$ and $\gamma < a-1/\alpha$. For all fixed $q\in\Z_+$ and $\omega \in \Omega_0^{*}$, when $n$ goes to infinity,
\begin{equation}\label{eq1:PXN4}
\frac{\Big| \Big( \Delta_{(u_1-u_2,v_1-v_2)}^{1,1} \big(\partial_v^q X_{M,n+l}\big)\Big) (u_2,v_2,\omega) - \Big(\Delta_{(u_1-u_2,v_1-v_2)}^{1,1}  \big(\partial_v^q X_{M,n}\big)\Big)(u_2,v_2,\omega) \Big|}{ |u_1-u_2|^{\gamma} |v_1-v_2| }
\end{equation}
converges to $0$ uniformly in $(u_1,u_2,v_1,v_2)\in [-M,M]^2\times [a,b]^2$ and in $l\in\Z_+$.
\end{lemma}

The proofs of the previous four lemmas are quite similar, so we will only give that of Lemma~\ref{PXN4}.

\begin{proof}[Proof of Lemma \ref{PXN4}]
In view of the convention that $0/0=0$, there is no restriction to assume that $u_1\ne u_2$ and $v_1\ne v_2$. By using (\ref{eq1:PWSEXN}), (\ref{eq1:normequiv-diff}) and Leibniz formula, one can rewrite (\ref{eq1:PXN4}) as, 
\begin{equation}\label{eq2:PXN4}
\frac{ \Big| \sum_{p=0}^q \dbinom{q}{p} \big( - \log 2 \big)^p \sum_{(j,k)\in D_{M,n+l} \setminus D_{M,n}} j^p \epsilon_{j,k}(\omega) \Big( \Delta_{(u_1-u_2,v_1-v_2)}^{1,1} \Theta^{q-p}_{j,k}\Big) (u_2,v_2) \Big| } { |u_1-u_2|^{\gamma} |v_1-v_2| },
\end{equation}
where for all $(u,v) \in \R\times (1/\alpha,1)$, 
\begin{equation}
\label{eq2bis:PXN4}
\Theta_{j,k}^{q-p}(u,v)=2^{-jv} \big(\partial_v^{q-p}\Psi\big)(2^ju-k,v).
\end{equation}
In the sequel, we denote by $D_{M,n}^c$ the set defined as $D_{M,n}^c=\big\{(j,k)\in\Z^2\,:\,(j,k)
\notin D_{M,n}\big\}$; recall that $D_{M,n}$ has been introduced in (\ref{eq2:PWSEXN}). Using (\ref{eq2:PXN4}), Taylor formula with respect to the variable $v$, (\ref{eq1:normequiv-diff}), and the triangle inequality, one obtains that
\begin{align}
\label{eq3:PXN4}
& \frac{\Big| \Big( \Delta^{1,1}_{(u_1-u_2,v_1-v_2)} \big(\partial_v^q X_{n+l}\big)\Big) (u_2,v_2,\omega) - \Big(\Delta^{1,1}_{(u_1-u_2,v_1-v_2)} \big(\partial_v^q X_{n}\big)\Big)(u_2,v_2,\omega) \Big|}{ |u_1-u_2|^{\gamma} |v_1-v_2| } \nonumber \\
&\hspace{7.5cm}\le G_{M,n}^{1,q}(u_1,u_2,v_2,\omega) + |v_1-v_2| G_{M,n}^{2,q}(u_1,u_2,v_1,v_2,\omega), \nonumber
\end{align}
where
\begin{equation}
\label{eq3:PXN4-ant1}
G_{M,n}^{1,q}(u_1,u_2,v_2,\omega):=\sum_{p=0}^q \dbinom{q}{p} (\log 2)^p \sum_{(j,k)\in D_{M,n}^c} |j|^p |\epsilon_{j,k}(\omega)| \frac{ \Big| \Big(\Delta^{1,\cdot}_{u_1-u_2} \big( \partial_v \Theta_{j,k}^{q-p} \big) \Big) (u_2,v_2) \Big|}{ |u_1-u_2|^{\gamma} }
\end{equation}
and 
\begin{align}
\label{eq3:PXN4-ant2}
& G_{M,n}^{2,q}(u_1,u_2,v_1,v_2,\omega)\nonumber\\
&:=\sum_{p=0}^q \dbinom{q}{p} (\log 2)^p \sum_{(j,k)\in D_{M,n}^c} |j|^p |\epsilon_{j,k}(\omega)| \frac{ \Big| \int_0^1 (1-s)  \Big( \Delta^{1,\cdot}_{u_1-u_2}\big( \partial_v^2 \Theta_{j,k}^{q-p} \big) \Big) (u_2, v_2+s(v_1-v_2)) ds \Big|}{ |u_1-u_2|^{\gamma} }.
\end{align}
Thus, for proving the lemma, it is sufficient to show that, when $n\rightarrow +\infty$, $G_{M,n}^{1,q}(u_1,u_2,v_2,\omega)$ and 
$G_{M,n}^{2,q}(u_1,u_2,v_1,v_2,\omega)$ converge to $0$, uniformly in $(u_1,u_2,v_1,v_2)\in [-M,M]^2 \times [a,b]^2$.

First, let us study $G_{M,n}^{1,q}(u_1,u_2,v_2,\omega)$. It follows from (\ref{eq2bis:PXN4}), that,
\begin{equation}
\label{eq2ter:PXN4}
(\partial_v \Theta_{j,k}^{q-p})(u,v)= 2^{-jv} \big(\partial_v^{q+1-p}\Psi\big)(2^ju-k,v)-(\log 2) j 2^{-jv} \big(\partial_v^{q-p}\Psi\big)(2^ju-k,v).
\end{equation}
Next, putting together (\ref{eq2ter:PXN4}), (\ref{eq1:normequiv-diff}), the triangle inequality, Lemma~\ref{omega0}, (\ref{SAN}) and (\ref{SBN}), one has,
\begin{align*}
& \sum_{(j,k)\in D_{M,n}^c} |j|^p |\epsilon_{j,k}(\omega)| \frac{ \Big| \Big(\Delta_{u_1-u_2}^{1,\cdot} \big( \partial_v \Theta_{j,k}^{q-p} \big) \Big) (u_2,v_2) \Big|}{ |u_1-u_2|^{\gamma} }\\
&\le \sum_{(j,k) \in D_{M,n}^c} 2^{-jv_2} |j|^p |\epsilon_{j,k}(\omega)| \frac{ \Big| \big(\partial_v^{q+1-p}\Psi \big)(2^ju_1-k,v_2) - \big(\partial_v^{q+1-p}\Psi \big)(2^ju_2-k,v_2) \Big|}{ |u_1-u_2|^{\gamma} } \nonumber \\
& \hspace{0.3cm}+ (\log 2) \sum_{(j,k) \in D_{M,n}^c} 2^{-jv_2} |j|^{p+1} |\epsilon_{j,k}(\omega)| \frac{ \Big| \big(\partial_v^{q-p}\Psi \big)(2^ju_1-k,v_2) - \big(\partial_v^{q-p}\Psi \big)(2^ju_2-k,v_2) \Big|}{ |u_1-u_2|^{\gamma} } \nonumber \\
& \leq C_{1}(\omega)\Big(A_n\big(u_1,u_2,v_2;M,\gamma,\eta,p,\partial_v^{q+1-p}\Psi\big) +  B_n\big(u_1,u_2,v_2;M,\gamma,\eta,p,\partial_v^{q+1-p}\Psi\big)\Big) \nonumber \\
& \hspace{0.3cm}+ C_{1}(\omega) (\log 2) \Big( A_{n}\big(u_1,u_2,v_2;M,\gamma,\eta,p+1,\partial_v^{q-p}\Psi\big) + B_{n}\big(u_1,u_2,v_2;M,\gamma,\eta,p+1,\partial_v^{q-p}\Psi\big)\Big),
\end{align*}
where $C_1$ denotes the random variable $C'$ introduced in Lemma~\ref{omega0}. Then Lemma~\ref{LA5} and (\ref{eq3:PXN4-ant1}) imply that, when $n\rightarrow +\infty$, $G_{M,n}^{1,q}(u_1,u_2,v_2,\omega)$ converges to $0$, uniformly 
in $(u_1,u_2,v_1, v_2)\in [-M,M]^2 \times [a,b]^2$.

Let us now study $G_{M,n}^{2,q}(u_1,u_2,v_1,v_2,\omega)$. It follows from (\ref{eq2ter:PXN4}) that,
\begin{eqnarray}
\label{eq2quat:PXN4}
(\partial_{v}^2 \Theta_{j,k}^{q-p})(u,v)= 2^{-jv}\big(\partial_v^{q+2-p}\Psi\big)(2^ju-k,v) &-& 2(\log 2) j 2^{-jv} \big(\partial_v^{q+1-p}\Psi\big)(2^ju-k,v)\nonumber\\
&+&(\log 2)^2 j^2 2^{-jv}\big(\partial_v^{q-p}\Psi\big)(2^ju-k,v).
\end{eqnarray}
Next, putting together (\ref{eq2quat:PXN4}), (\ref{eq1:normequiv-diff}), the triangle inequality, Lemma~\ref{omega0}, (\ref{SAN}) and (\ref{SBN}), one has,
\begin{align*}
& \sum_{(j,k)\in D_{M,n}^c} |j|^p |\epsilon_{j,k}(\omega)| \frac{ \Big| \int_0^1 (1-s)  \Big( \Delta_{u_1-u_2}^{1,\cdot}\big( \partial_v^2 \Theta_{j,k}^{q-p} \big) \Big) (u_2, v_2+s(v_1-v_2)) ds \Big|}{ |u_1-u_2|^{\gamma} }\\
& \le C_1(\omega)\int_{0}^1 \Big( A_n\big(u_1,u_2,v_2+s(v_1-v_2);M,\gamma,\eta,p,\partial_v^{q+2-p}\Psi\big) +\\   &\hspace{7cm} B_n\big(u_1,u_2,v_2+s(v_1-v_2);M,\gamma,\eta,p,\partial_v^{q+2-p}\Psi\big) \Big) ds\nonumber \\
& \hspace{0.1cm}+C_{2}(\omega) \int_{0}^1\Big( A_{n}\big(u_1,u_2,v_2+s(v_1-v_2);M,\gamma,\eta,p+1,\partial_v^{q+1-p}\Psi\big) +\\ 
&\hspace{7cm} B_{n}\big(u_1,u_2,v_2+s(v_1-v_2);M,\gamma,\eta,p+1,\partial_v^{q+1-p}\big) \Big)ds\\
&\hspace{0.1cm}+C_{2}(\omega)\int_{0}^1\Big(A_{n}\big(u_1,u_2,v_2+s(v_1-v_2);M,\gamma,\eta,p+2,\partial_v^{q-p}\Psi\big) + \\
&\hspace{7cm}B_{n}\big(u_1,u_2,v_2+s(v_1-v_2);M,\gamma,\eta,p+2,\partial_v^{q-p}\Psi\big) \Big)ds,
\end{align*}
where $C_2(\omega)=(2\log 2)C_1(\omega)$. Then Lemma~\ref{LA5} and (\ref{eq3:PXN4-ant2}) imply that, when $n\rightarrow +\infty$, $G_{M,n}^{2,q}(u_1,u_2,v_1,v_2,\omega)$ converges to $0$, uniformly 
in $(u_1,u_2,v_1,v_2)\in [-M,M]^2 \times [a,b]^2$.
\end{proof}

Now we are in position to prove Theorem~\ref{TWSEbis}.

\begin{proof}[Proof of Theorem~\ref{TWSEbis}]  Let $\o\in\Omega_{0}^*$ be arbitrary and fixed. First we show that Part $(i)$ of the theorem holds. By using Lemma~\ref{omega0}, Proposition~\ref{regulpsi} and a method similar to the one which allowed to derive (\ref{eq1:absconc}), we can prove that, for all 
fixed $q\in\N$ and $(u,v)\in\R\times (1/\al,1)$, one has,
$$
\sum_{p=0}^q \dbinom{q}{p} \big(\log 2 \big)^p \sum_{(j,k)\in\Z^2} |j|^p 2^{-jv}\big|\epsilon_{j,k}(\omega)\big|  \left |\big(\partial_v^{q-p}\Psi\big)(2^ju-k,v)-\big(\partial_v^{q-p}\Psi\big)(-k,v)\right|<\infty.
$$
Therefore, the series of real numbers,
$$
\sum_{p=0}^q \dbinom{q}{p} \big(-\log 2 \big)^p \sum_{(j,k)\in\Z^2} j^p 2^{-jv}\epsilon_{j,k}(\omega) \left (\big(\partial_v^{q-p}\Psi\big)(2^ju-k,v)-\big(\partial_v^{q-p}\Psi\big)(-k,v)\right),
$$
is convergent, and its finite limit, denoted by $\check{X}^{(q)}(u,v,\o)$, does not depend on the way the terms of the series are ordered. Let us now assume 
that $u\in\R$ is arbitrary and fixed and that the variable $v$ belongs to an arbitrary fixed compact interval $[a,b]$ contained in $(1/\al,1)$. 
We denote by $M$ an arbitrary fixed positive real number such that $u\in [-M,M]$. In view of Theorem~\ref{TWSE} Part $(i)$, Proposition~\ref{prop:cauchy} Part $(iii)$, and (\ref{eq1:normequiv}), when $n$ goes to infinity, the following two results are satisfied:
\begin{itemize}
\item the function $v\mapsto X_{M,n}(u,v,\o)$ converges to the function $v\mapsto X(u,v,\o)$, uniformly in $v\in [a,b]$;
\item for each fixed $q\in\N$, the function $v\mapsto (\partial_{v}^q X_{M,n})(u,v,\o)$ converges to the function $v\mapsto \check{X}^{(q)}(u,v,\o)$, uniformly in $v\in [a,b]$.
\end{itemize}
The latter two results imply that $v\mapsto X(u,v,\o)$ is an infinitly differentiable function over $[a,b]$ and one has, for all $q\in\N$ and $v\in [a,b]$,
\begin{equation}
\label{eq2:TWSEbis}
(\partial_{v}^q X)(u,v,\o)=\lim_{n\rightarrow +\infty} (\partial_{v}^q X_{M,n})(u,v,\o)=\check{X}^{(q)}(u,v,\o);
\end{equation}
these equalities mean that (\ref{eq1:TWSEbis}) is satisfied. Thus, it remains to show that Part $(ii)$ of the theorem holds. In fact, the equality $X(u,v,\o)=\lim_{n\rightarrow +\infty}  X_{M,n}(u,v,\o)$, (\ref{eq2:TWSEbis}), and Proposition~\ref{prop:cauchy} Part $(iii)$, imply that this is indeed the case.
\end{proof}

Before ending this section, let us stress that for each fixed $\o\in \Omega_{0}^*$, $q\in\Z_+$ and $M,a,b\in\R$ satisfying $M>0$ and $1/\al <a<b<1$,  Theorem~\ref{TWSEbis} Part $(ii)$, allows to derive, uniformly in $v\in [a,b]$, a global modulus of continuity of the function $u\mapsto (\partial_v^q X)(u,v,\o)$, on the interval $[-M,M]$; also, it allows to derive, uniformly in $u\in [-M,M]$, a global modulus of continuity of the function $v\mapsto (\partial_v^q X)(u,v,\o)$, on the interval $[a,b]$. More precisely, in view of Definition~\ref{de:spaceE}, a straightforward consequence of Theorem~\ref{TWSEbis} Part $(ii)$, is the following:

\begin{corollary}
\label{cor:TWSEbis}
For each fixed $\o\in \Omega_{0}^*$, $q\in\Z_+$ and $M,a,b,\eta\in\R$ satisfying $M>0$, $1/\al <a<b<1$ and $\eta>0$, one has,
\begin{equation}
\label{eq1:CTWSEbis}
\sup_{(u_1,u_2,v) \in [-M,M]^2\times [a,b]} \left\{\frac{\big| \big( \partial_v^q X \big)(u_1,v,\omega) - \big( \partial_v^q X \big)(u_2,v,\omega)\big|}{ |u_1-u_2|^{a-1/\alpha- \eta} }  \right\}< \infty, 
\end{equation}
and 
\begin{equation}
\label{eq2:CTWSEbis}
\sup_{(u,v_1,v_2) \in [-M,M]\times [a,b]^2} \left\{\frac{\big| \big( \partial_v^q X \big)(u,v_1,\omega) - \big( \partial_v^q X \big)(u,v_2,\omega)\big|}{ |v_1-v_2|}\right\}  < \infty.
\end{equation}
\end{corollary}

\section{Fine path properties of the field generating LMSM's}
\label{sec:Xpath}
The main two goals of this section are the following:
\begin{itemize}
\item to give an improved version of the global modulus of continuity (\ref{eq1:CTWSEbis});
\item to derive, an upper bound of $\big| \big( \partial_v^q X \big)(u,v,\omega)\big|$, for all $\o\in\Omega_{0}^*$, $q\in\Z_+$, $v\in [a,b]\subset (1/\al,1)$ and $u\in\R$.
\end{itemize}
More precisely, we will show that the following two results hold.

\begin{proposition}
\label{pp1:dX}
For each fixed $\o\in \Omega_{0}^*$, $q\in\Z_+$ and $M,a,b,\eta\in\R$ satisfying $M>0$, $1/\al <a<b<1$ and $\eta>0$, one has,
\begin{align}
\label{eq1:dX}
& \nonumber\sup_{(u_1,u_2,v) \in [-M,M]^2\times [a,b]} \left\{\frac{\big| \big( \partial_v^q X \big)(u_1,v,\omega) - \big( \partial_v^q X \big)(u_2,v,\omega)\big|}{ |u_1-u_2|^{v-1/\alpha} \big( 1+\big| \log |u_1-u_2| \big| \big)^{q+2/\alpha+\eta}} \right\} \\
\nonumber
&\le\sup_{(u_1,u_2,v) \in [-M,M]^2\times [a,b]}\\
&\nonumber\hspace{1.5cm}\left\{\frac{\sum_{p=0}^q \dbinom{q}{p} \big(\log 2 \big)^p \sum_{(j,k)\in\Z^2} |j|^p 2^{-jv}\big|\epsilon_{j,k}(\omega)\big|  \left |\big(\partial_v^{q-p}\Psi\big)(2^j u_1-k,v)-\big(\partial_v^{q-p}\Psi\big)(2^j u_2-k,v)\right|}{|u_1-u_2|^{v-1/\alpha} \big( 1+\big| \log |u_1-u_2| \big| \big)^{q+2/\alpha+\eta}} \right\}\\
& < \infty.
\end{align}
\end{proposition}

\begin{proposition}
\label{prop:asympinf}
For each fixed $\o\in \Omega_{0}^*$, $q\in\Z_+$ and $a,b,\eta\in\R$ satisfying $1/\al <a<b<1$ and $\eta>0$, one has,
\begin{align}
\label{eq1:pasymp}
& \nonumber \sup_{(u,v) \in \R\times [a,b]} \left\{\frac{ \big|\big(\partial_v^q X \big)(u,v,\omega) \big|}{ |u|^{v} \big(1+\big|\log |u| \big|\big)^{q+1/\alpha+\eta}} \right\}\\
\nonumber
& \le \sup_{(u,v) \in \R\times [a,b]} \left\{\frac{\sum_{p=0}^q \dbinom{q}{p} \big(\log 2 \big)^p \sum_{(j,k)\in\Z^2} |j|^p 2^{-jv}\big|\epsilon_{j,k}(\omega)\big|  \left |\big(\partial_v^{q-p}\Psi\big)(2^ju-k,v)-\big(\partial_v^{q-p}\Psi\big)(-k,v)\right|}{|u|^{v} \big(1+\big|\log |u| \big|\big)^{q+1/\alpha+\eta}} \right\}\\
& < \infty.
\end{align}
\end{proposition}

The proofs of Propositions~\ref{pp1:dX}~and~\ref{prop:asympinf} are, to a certain extent, inspired by that of Theorem~1~in~\cite{ayache2009linear}. 

\begin{proof}[Proof of Proposition \ref{pp1:dX}]
Let $(u_1,u_2,v)\in[-M,M]^2\times [a,b]$ be arbitrary and fixed; in all the sequel we assume that $u_1\ne u_2$. Observe that,
in view of (\ref{localisation}), there is a constant $c_1>0$, non depending on $(u_1,u_2,v)$, such that for all $p\in\{0,\ldots,q\}$ and $(j,k)\in\Z^2$, one has,
\begin{equation}\label{eq1:pp1dX}
\big| \big(\partial_v^{q-p} \Psi\big)(2^j u_1-k,v) - \big(\partial_v^{q-p} \Psi\big)(2^j u_2-k,v) \big| \leq c_1 \left ( \big( 3+|2^j u_1-k| \big)^{-2} + \big( 3+|2^j u_2-k| \big)^{-2} \right).
\end{equation}
Also notice that $\big| \big(\partial_v^{q-p} \Psi \big)(2^j u_1-k,v) -  \big(\partial_v^{q-p} \Psi \big)(2^j u_2-k,v) \big|$ can be bounded more sharply when the condition 
\begin{equation}\label{eq1bis:pp1dX}
2^j|u_1-u_2| \leq 1 
\end{equation}
holds, namely using the Mean Value Theorem  and (\ref{localisation}), one has, 
\begin{align}\label{eq2:pp1dX}
\big| \big(\partial_v^{q-p} \Psi\big)(2^j u_1-k,v) - \big(\partial_v^{q-p} \Psi\big)(2^j u_2-k,v) \big| & \leq 2^j |u_1-u_2| \sup_{(u,v)\in [u_1\wedge u_2, u_1\vee u_2]\times [a,b]}\big| \big(\partial_x\partial_v^{q-p} \Psi\big)(2^j u-k,v)\big|\nonumber \\
& \leq c_1 2^j |u_1-u_2| \sup_{u\in [u_1\wedge u_2, u_1\vee u_2]} \big(3+|2^ju-k|\big)^{-2} \nonumber \\
& \leq c_1 2^j |u_1-u_2| \big(2+|2^j u_1-k|\big)^{-2},
\end{align}
where the last inequality results from the triangle inequality and (\ref{eq1bis:pp1dX}). Denote by $j_0 > -\log_2(4M)$ the unique integer satisfying
\begin{equation}\label{eq3:defj0}
2^{-1} < 2^{j_0} |u_1-u_2| \leq 1. 
\end{equation}
Then, the first inequality in (\ref{eq1:omega0}), (\ref{eq1:pp1dX}) and (\ref{eq2:pp1dX}), entail that, for all $\eta>0$ and $\o\in\Omega_{0}^*$,
\begin{align}
\label{eq1:decAB}
& \sum_{(j,k)\in\Z^2} |j|^p 2^{-jv} \big|\epsilon_{j,k}(\omega)\big| \left |\big(\partial_v^{q-p}\Psi\big)(2^ju_1-k,v)-\big(\partial_v^{q-p}\Psi\big)(u_2-k,v)\right| \le \nonumber\\
& C(\o)\sum_{(j,k) \in \Z^2} 2^{-jv} (1+|j|)^{p+1/\alpha+\eta}(1+|k|)^{1/\alpha} \log^{1/\alpha+\eta}(2+|k|) \big| \big(\partial_v^{q-p} \Psi\big)(2^ju_1-k,v) - \big(\partial_v^{q-p} \Psi\big)(2^j u_2-k,v) \big| \nonumber \\
& \leq C(\o)c_1 \Big( \check{A}_{j_0}(u_1,v) |u_1-u_2| + \check{B}_{j_0}(u_1,u_2,v) \Big),
\end{align}
where the random variable $C$ has been introduced in Lemma~\ref{omega0} and where for each $J\in\Z$, $(y_1,y_2)\in\R^2$, and $v\in [a,b]$,
\begin{equation}
\label{Aj0:pasymp}
\check{A}_{J}(y_1,v):=\sum_{j\leq J} \sum_{k\in \Z} 2^{j(1-v)} (1+|j|)^{p+1/\alpha+\eta} (1+|k|)^{1/\alpha} \log^{1/\alpha+\eta}(2+|k|) \big(2+|2^j y_1-k|\big)^{-2}  
\end{equation}
and
\begin{align}
\label{Bj0:pasymp}
& \check{B}_{J}(y_1,y_2,v)\\ 
& := \sum_{j>J} \sum_{k\in\Z} 2^{-jv} (1+|j|)^{p+1/\alpha+\eta} (1+|k|)^{1/\alpha} \log^{1/\alpha+\eta}(2+|k|) \left ( \big(3+|2^j y_1-k|\big)^{-2} + \big(3+|2^j y_2-k|\big)^{-2} \right)\nonumber. 
\end{align}
Let us now give an appropriate upper bound for $\check{A}_{j_0}(u_1,v)$. Assume that $j\leq j_0$; using Lemma~\ref{LA3} (in which one takes $\theta=1/\al$, $\zeta=1/\alpha+\eta$ and $u=2^j u_1$) and the inequality $|u_1|\le M$, one obtains that,
\begin{equation}
\sum_{k\in\Z} \frac{ (1+|k|)^{1/\alpha} \log^{1/\alpha+\eta}(2+|k|)}{ \big(2+|2^j u_1-k|\big)^2} \leq c_2 2^{j_0/\alpha} (1+|j_0|)^{1/\alpha+\eta}, \nonumber
\end{equation}
where $c_2$ is a constant only depending on $M$, $\al$ and $\eta$. Next, it follows from the latter inequality, (\ref{Aj0:pasymp}) and Lemma~\ref{LA2} (in which one takes $\theta=1-v$, $\theta_0=1-b$, $\lambda=p+1/\alpha +\eta$, $n_0=-\infty$ and $n_1=j_0$) that,
\begin{align}\label{eq4:Aj0}
\check{A}_{j_0}(u_1,v) & \leq c_2 2^{j_0/\alpha} (1+|j_0|)^{1/\alpha+\eta}\sum_{j\leq j_0} 2^{j(1-v)} (1+|j|)^{p+1/\alpha +\eta} \leq c_3 2^{j_0(1-v+1/\alpha)} (1+|j_0|)^{p+2/\alpha +2\eta} \nonumber \\
& \leq c_4 |u_1-u_2|^{v-1/\alpha-1} \big( 1 + \big|\log |u_1-u_2| \big| \big)^{p+2/\alpha+2\eta},
\end{align}
where the last inequality results from (\ref{eq3:defj0}), and where $c_3$ and $c_4$ are two constants non depending $(u_1,u_2,v)$.
Let us now give an appropriate upper bound for $\check{B}_{j_0}(u_1,u_2,v)$. In view of (\ref{Bj0:pasymp}), this quantity can be expressed as,
\begin{equation}
\label{eq6:Tj0}
\check{B}_{j_0}(u_1,u_2,v) = T_{j_0}(u_1,v) + T_{j_0}(u_2,v), 
\end{equation}
where, for each $J\in\Z$, $y\in\R$ and $v\in [a,b]$,
\begin{equation}
\label{eq:defTJ}
T_{J}(y,v):= \sum_{j> J} \sum_{k\in\Z} 2^{-jv} \frac{ (1+|j|)^{p+1/\alpha+\eta} (1+|k|)^{1/\alpha} \log^{1/\alpha+\eta}(2+|k|) }{ \big(3+|2^j y-k|\big)^2 }. 
\end{equation}
Assume that $j> j_0$ and that $x\in \{u_1,u_2\}$; using Lemma~\ref{LA3} (in which one takes $\theta=1/\al$, $\zeta=1/\alpha+\eta$ and $u=2^j x$) and the inequality 
$|x|\le M$, one gets that, 
\begin{equation}
\sum_{k\in\Z} \frac{ (1+|k|)^{1/\alpha} \log^{1/\alpha+\eta}(2+|k|)}{ \big(2+|2^j x-k|\big)^2} \leq c_2 2^{j/\alpha} (1+|j|)^{1/\alpha+\eta}. \nonumber
\end{equation}
Next, in view of (\ref{eq:defTJ}), it follows from the latter inequality and Lemma~\ref{LA2} (in which one takes $\theta=v-1/\al$, $\theta_0=a-1/\al$, $\lambda=p+2/\alpha +2\eta$, $n_0=j_0+1$ and $n_1=+\infty$) that,
\begin{align}\label{eq5:Tj0}
T_{j_0}(x,v) & \leq c_2 \sum_{j > j_0} 2^{-j(v-1/\al)} (1+|j|)^{p+2/\alpha +2\eta} \leq c_5 2^{-j_0(v-1/\alpha)} (1+|j_0|)^{p+2/\alpha +2\eta} \nonumber \\
& \leq c_6 |u_1-u_2|^{v-1/\alpha} \big( 1 + \big|\log |u_1-u_2| \big| \big)^{p+2/\alpha+2\eta},
\end{align}
where the last inequality results from (\ref{eq3:defj0}), and where $c_5$ and $c_6$ are two constants non depending on $(x,v)$.
Next, (\ref{eq5:Tj0}) and (\ref{eq6:Tj0}) imply that
\begin{equation}\label{eq7:Bj0}
\check{B}_{j_0}(u_1,u_2,v) \leq 2c_6 |u_1-u_2|^{v-1/\alpha} \big( 1+\big| \log |u_1-u_2| \big| \big)^{p+2/\alpha+2\eta}.
\end{equation} 
Next putting together, (\ref{eq4:Aj0}), (\ref{eq7:Bj0})  and (\ref{eq1:decAB}), one obtains that, for all $\eta>0$ and $\o\in\Omega_{0}^*$,
\begin{align}
\label{eq1:majA+B}
& \sum_{(j,k)\in\Z^2} |j|^p 2^{-jv} \big|\epsilon_{j,k}(\omega)\big| \left |\big(\partial_v^{q-p}\Psi\big)(2^ju_1-k,v)-\big(\partial_v^{q-p}\Psi\big)(u_2-k,v)\right|\nonumber\\
& \leq C(\o)c_7 |u_1-u_2|^{v-1/\alpha} \big( 1+\big| \log |u_1-u_2| \big| \big)^{p+2/\alpha+2\eta}, 
\end{align}
where $c_7$ is a constant non depending on $(u_1,u_2,v)$. Finally, (\ref{eq1:TWSEbis}), the triangle inequality and (\ref{eq1:majA+B}) entail that
(\ref{eq1:dX}) holds.
\end{proof}

\begin{proof}[Proof of Proposition \ref{prop:asympinf}] Let $(u,v)\in\R\times [a,b]$ be arbitrary and fixed, in all the sequel we assume that $u\ne 0$. 
Observe that,
in view of (\ref{localisation}), there is a constant $c_1>0$, non depending on $(u,v)$, such that for all $p\in\{0,\ldots,q\}$ and $(j,k)\in\Z^2$, one has,
\begin{equation}\label{eq2:pasymp}
\big| \big(\partial_v^{q-p} \Psi \big)(2^j u-k,v) - \big(\partial_v^{q-p} \Psi \big)(-k,v) \big| \leq c_1 \Big ( \big(3+|2^j u-k|\big)^{-2} + \big(3+|k|\big)^{-2} \Big).
\end{equation}
Also notice that $| (\partial_v^{q-p} \Psi )(2^ju-k,v) - (\partial_v^{q-p} \Psi )(-k,v) |$ can be bounded more sharply when the condition 
\begin{equation}
\label{eq2bis:pasymp}
2^j |u| \le 1
\end{equation}
 holds, namely, using the Mean Value Theorem and (\ref{localisation}), one has,
\begin{align}\label{eq3:pasymp}
\big| \big(\partial_v^{q-p} \Psi \big)(2^j u-k,v) - \big(\partial_v^{q-p} \Psi \big)(-k,v) \big| & \le 2^j |u| \sup_{y\in [u\wedge 0,u\vee 0]}
\big| \big(\partial_x \partial_v^{q-p} \Psi\big)(2^j y-k,v)\big|\nonumber\\
& \leq c_1 2^j |u| \sup_{y\in [u\wedge 0,u\vee 0]} \big( 3+|2^j y-k| \big)^{-2} \nonumber \\
& \leq c_1 2^j |u| \big(2+|k|\big)^{-2},
\end{align}
where the last inequality results from the triangle inequality and (\ref{eq2bis:pasymp}). Denote by $j_1\in \Z$ the unique integer satisfying
\begin{equation}\label{defj0:pasymp}
2^{-1} < 2^{j_1} |u| \leq 1.
\end{equation} 
Then the first inequality in (\ref{eq1:omega0}), (\ref{eq2:pasymp}) and (\ref{eq3:pasymp}) entail that, for all $\eta>0$ and $\o\in\Omega_{0}^*$,

\begin{align}
\label{eq2:decAB}
& \sum_{(j,k)\in\Z^2} |j|^p 2^{-jv}\big|\epsilon_{j,k}(\omega)\big|  \left |\big(\partial_v^{q-p}\Psi\big)(2^ju-k,v)-\big(\partial_v^{q-p}\Psi\big)(-k,v)\right|\nonumber\\
& \le C(\o)\sum_{(j,k)\in\Z^2} 2^{-jv} \big(1+|j|\big)^{p+1/\alpha+\eta} \big(1+|k|\big)^{1/\alpha} \log^{1/\alpha+\eta}\big(2+|k|\big) \big| \big(\partial_v^{q-p} \Psi \big)(2^j u-k,v) - \big(\partial_v^{q-p} \Psi \big)(-k,v) \big|\nonumber \\
& \leq C(\o)c_1 \Big( |u| \check{A}_{j_1}(0,v) + \check{B}_{j_1}(u,0,v) \Big), 
\end{align}
where the random variable $C$ has been introduced in Lemma~\ref{omega0} and where $\check{A}_{j_1}(0,v)$ and $\check{B}_{j_1}(u,0,v)$ are defined respectively by (\ref{Aj0:pasymp}) and 
(\ref{Bj0:pasymp}).
Let us now give an appropriate upper bound for $\check{A}_{j_1}(0,v)$. Observe that 
\begin{equation}
c_{2} := \sum_{k\in\Z} \frac{ \big(1+|k|)^{1/\alpha} \log^{1/\alpha+\eta}\big(2+|k|\big) }{ \big(2+|k|\big)^{2} } < \infty. \nonumber
\end{equation}
Thus, (\ref{Aj0:pasymp}) and Lemma \ref{LA2} (in which one takes $\theta=1-v$, $\theta_0=1-b$, $\lambda=p+1/\alpha +\eta$, $n_0=-\infty$ and $n_1=j_1$) imply that,
\begin{align}
\label{eq4Aj0:psymp}
\check{A}_{j_1}(0,v) & = c_{2} \sum_{j\leq j_1} 2^{j(1-v)} \big(1+|j|\big)^{p+1/\alpha+\eta} \leq c_{3} 2^{j_1(1-v)} \big(1+|j_1| \big)^{p+1/\alpha+\eta} \nonumber \\
& \leq c_{4} |u|^{v-1} \big(1+\big| \log |u| \big| \big)^{p+1/\alpha+\eta},
\end{align}
where the last inequality results from (\ref{defj0:pasymp}) and where $c_3$ and $c_4$ are two constants non depending on $(u,v)$. 
Let us now give an appropriate upper bound for $\check{B}_{j_1}(u,0,v)$. In view of (\ref{Bj0:pasymp}), this quantity can be expressed as,
\begin{equation}
\label{eq2:decompBj0}
\check{B}_{j_1}(u,0,v) := T_{j_1}(u,v) + T_{j_1}(0,v), 
\end{equation}
where $T_{j_1}(u,v)$ and $T_{j_1}(0,v)$ are defined by (\ref{eq:defTJ}).
Assume that $j>j_1$ and that $x\in\{u,0\}$; it follows from Lemma~\ref{LA3} in which one takes $\theta=1/\alpha$ and $\zeta=1/\alpha+\eta$, that, 
$$
\sum_{k\in\Z} \frac{ \big(1+|k|)^{1/\alpha} \log^{1/\alpha+\eta}\big(2+|k| \big) }{\big( 3+ |2^j x-k| \big)^{2} } \leq c_5 \big(1+2^j |x| \big)^{1/\alpha} \log^{1/\alpha+\eta}\big( 2+2^j |x| \big)\le c_6 2^{(j-j_1)/\alpha}(1+j-j_1)^{1/\alpha+\eta},
$$
where the last inequality results from (\ref{defj0:pasymp}) and where $c_5$ and $c_6$ are two constants non depending on $x$, $v$, $j$ and $j_1$. Therefore, in view of (\ref{eq:defTJ}), one obtains that
\begin{equation}\label{eq5:Btj0}
T_{j_1}(x,v)\leq c_6 \sum_{j> j_1} 2^{-jv} 2^{(j-j_1)/\alpha} \big(1+|j| \big)^{p+1/\alpha+\eta} \big( 1+j-j_1 \big)^{1/\alpha+\eta}.
\end{equation}
Next, setting $l=j-j_1$ in the right-hand side of (\ref{eq5:Btj0}) and using Lemma~\ref{LA4}, it follows that,
\begin{align}
\label{eq6:Btj0}
 T_{j_1}(x,v) & \leq c_6 \sum_{l=1}^{+\infty} 2^{-j_1 v} 2^{-l(v-1/\alpha)} (1+l)^{1/\alpha+\eta} \big( 1+|l+j_1| \big)^{p+1/\alpha+\eta} \nonumber \\
& \leq c_7 2^{-j_1 v} \sum_{l=1}^{+\infty} 2^{-l(v-1/\alpha)} (1+l)^{1/\alpha+\eta} \Big( \big(1+l\big)^{p+1/\alpha+\eta} + \big(1+|j_1|\big)^{p+1/\alpha+\eta} \Big) \nonumber\\
& \leq c_7 2^{-j_1 v} \sum_{l=1}^{+\infty} 2^{-l(a-1/\alpha)} (1+l)^{1/\alpha+\eta} \Big( \big(1+l\big)^{p+1/\alpha+\eta} + \big(1+|j_1|\big)^{p+1/\alpha+\eta} \Big)\nonumber\\
&\le c_8 2^{-j_1 v} \big(1+|j_1|\big)^{p+1/\alpha+\eta},\nonumber\\
&\le c_9 |u|^v \left(1+ \big|\log |u|\big| \right)^{p+1/\alpha+\eta},
\end{align}
where the last inequality results from (\ref{defj0:pasymp}) and where the constants $c_7$, $c_8$ and $c_9$ do not depend on $x$, $v$ and $j_1$.
Next, (\ref{eq2:decompBj0}) and (\ref{eq6:Btj0}) imply that,
\begin{equation}
\label{eq:majBj0u0}
\check{B}_{j_1}(u,0,v)\le 2 c_9 |u|^v \left(1+ \big|\log |u|\big| \right)^{p+1/\alpha+\eta}.
\end{equation}
Next, putting together (\ref{eq2:decAB}), (\ref{eq4Aj0:psymp}) and (\ref{eq:majBj0u0}), one gets that,
\begin{align}
\label{eq:majXp0infty}
& \sum_{(j,k)\in\Z^2} |j|^p 2^{-jv}\big|\epsilon_{j,k}(\omega)\big|  \left |\big(\partial_v^{q-p}\Psi\big)(2^ju-k,v)-\big(\partial_v^{q-p}\Psi\big)(-k,v)\right|\nonumber\\
& \le C(\o)c_{10}|u|^v \left(1+ \big|\log |u|\big| \right)^{p+1/\alpha+\eta},
\end{align}
where $c_{10}$ is a constant non depending on $(u,v)$. Finally, (\ref{eq1:TWSEbis}), the triangle inequality and (\ref{eq:majXp0infty}) entail that 
(\ref{eq1:pasymp}) holds.
\end{proof}

Before ending this section, let us stress that, thanks to 
(\ref{eq2:CTWSEbis}) and (\ref{eq1:dX}), for each fixed $\o\in \Omega_{0}^*$, $q\in\Z_+$ and $M,a,b\in\R$ satisfying $M>0$ and $1/\al <a<b<1$, one can derive, a global modulus of continuity of the function $(u,v)\mapsto (\partial_v ^q X)(u,v,\o)$, on the rectangle $[-M,M]\times [a,b]$. More precisely, the following result holds.
\begin{corollary}
\label{cor:modcont}
For each fixed $\o\in \Omega_{0}^*$, $q\in\Z_+$ and $M,a,b,\eta\in\R$ satisfying $M>0$, $1/\al <a<b<1$ and $\eta>0$, one has,
\begin{equation}
\label{eq1:modcont}
\sup_{(u_1,u_2,v_1,v_2) \in [-M,M]^2\times [a,b]^2} \left\{\frac{\big| \big( \partial_v^q X \big)(u_1,v_1,\omega) - \big( \partial_v^q X \big)(u_2,v_2,\omega)\big|}{ |u_1-u_2|^{v_1\vee v_2-1/\alpha}\big(1+\big|\log |u_1-u_2| \big| \big)^{q+2/\alpha+\eta} +|v_1-v_2|} \right\}< \infty.
\end{equation}
\end{corollary}

\begin{proof}[Proof of Corollary \ref{cor:modcont}] For each $(u_1,u_2,v_1,v_2) \in [-M,M]^2\times [a,b]^2$, one sets,
$$
f(u_1,u_2,v_1,v_2):=\frac{\big| \big( \partial_v^q X \big)(u_1,v_1,\omega) - \big( \partial_v^q X \big)(u_2,v_2,\omega)\big|}{ |u_1-u_2|^{v_1\vee v_2-1/\alpha}\big(1+\big|\log |u_1-u_2| \big| \big)^{q+2/\alpha+\eta} +|v_1-v_2|},
$$
with the convention that $0/0=0$. Using the fact that $f(u_1,u_2,v_1,v_2)=f(u_2,u_1,v_2,v_1)$, it follows that,
\begin{align}
\label{eq2:modcont}
& \sup_{(u_1,u_2,v_1,v_2) \in [-M,M]^2\times [a,b]^2} \left\{\frac{\big| \big( \partial_v^q X \big)(u_1,v_1,\omega) - \big( \partial_v^q X \big)(u_2,v_2,\omega)\big|}{ |u_1-u_2|^{v_1\vee v_2-1/\alpha}\big(1+\big|\log |u_1-u_2| \big| \big)^{q+2/\alpha+\eta} +|v_1-v_2|} \right\}\\
& = \sup_{(u_1,u_2,v_1,v_2) \in [-M,M]^2\times [a,b]^2} \left\{\frac{\big| \big( \partial_v^q X \big)(u_1,v_1\vee v_2,\omega) - \big( \partial_v^q X \big)(u_2,v_1\wedge v_2,\omega)\big|}{ |u_1-u_2|^{v_1\vee v_2-1/\alpha}\big(1+\big|\log |u_1-u_2| \big| \big)^{q+2/\alpha+\eta} +|v_1-v_2|} \right\}.\nonumber
\end{align}
Moreover, using the triangle inequality, and the inequality for all $(u_1,u_2,v_1,v_2) \in [-M,M]^2\times [a,b]^2$,
\begin{align*}
& \max\left\{|u_1-u_2|^{v_1\vee v_2-1/\alpha}\big(1+\big|\log |u_1-u_2| \big| \big)^{q+2/\alpha+\eta}, |v_1-v_2|\right\} \\
& \le |u_1-u_2|^{v_1\vee v_2-1/\alpha}\big(1+\big|\log |u_1-u_2| \big| \big)^{q+2/\alpha+\eta} +|v_1-v_2|,
\end{align*}
one gets that,
\begin{align}
\label{eq3:modcont}
& \sup_{(u_1,u_2,v_1,v_2) \in [-M,M]^2\times [a,b]^2} \left\{\frac{\big| \big( \partial_v^q X \big)(u_1,v_1\vee v_2,\omega) - \big( \partial_v^q X \big)(u_2,v_1\wedge v_2,\omega)\big|}{ |u_1-u_2|^{v_1\vee v_2-1/\alpha}\big(1+\big|\log |u_1-u_2| \big| \big)^{q+2/\alpha+\eta} +|v_1-v_2|} \right\}\nonumber\\
& \le \sup_{(u_1,u_2,v_1,v_2) \in [-M,M]^2\times [a,b]^2} \left\{\frac{\big| \big( \partial_v^q X \big)(u_1,v_1\vee v_2,\omega) - \big( \partial_v^q X \big)(u_2,v_1\vee v_2,\omega)\big|}{ |u_1-u_2|^{v_1\vee v_2-1/\alpha}\big(1+\big|\log |u_1-u_2| \big| \big)^{q+2/\alpha+\eta} +|v_1-v_2|} \right\}\nonumber\\
& \hspace{1cm}+\sup_{(u_1,u_2,v_1,v_2) \in [-M,M]^2\times [a,b]^2} \left\{\frac{\big| \big( \partial_v^q X \big)(u_2,v_1\vee v_2,\omega) - \big( \partial_v^q X \big)(u_2,v_1\wedge v_2,\omega)\big|}{ |u_1-u_2|^{v_1\vee v_2-1/\alpha}\big(1+\big|\log |u_1-u_2| \big| \big)^{q+2/\alpha+\eta} +|v_1-v_2|} \right\}\nonumber\\
& \le \sup_{(u_1,u_2,v) \in [-M,M]^2\times [a,b]} \left\{\frac{\big| \big( \partial_v^q X \big)(u_1,v,\omega) - \big( \partial_v^q X \big)(u_2,v,\omega)\big|}{ |u_1-u_2|^{v-1/\alpha}\big(1+\big|\log |u_1-u_2| \big| \big)^{q+2/\alpha+\eta} } \right\}\\
& \hspace{3cm}+\sup_{(u, v_1,v_2) \in [-M,M]\times [a,b]^2} \left\{\frac{\big| \big( \partial_v^q X \big)(u,v_1,\omega) - \big( \partial_v^q X \big)(u, v_2,\omega)\big|}{ |v_1-v_2|} \right\}.\nonumber
\end{align}
Finally, putting together, (\ref{eq2:modcont}), (\ref{eq3:modcont}), (\ref{eq2:CTWSEbis}) and (\ref{eq1:dX}), one obtains (\ref{eq1:modcont}).

\end{proof}

\section{Global and local moduli of continuity of LMSM}
\label{subsec:modcont}

From now on and till the end of the article, LMSM is identified with its modification $\{Y(t) : t\in\R\}$, defined for all $t\in\R$, by,
\begin{equation}
\label{def:LMSMbis}
Y(t)=X(t,H(t)), 
\end{equation}
where $\{X(u,v): (u,v)\in\R\times (1/\al,1)\}$ is the $\stas$ field introduced in Theorem~\ref{TWSE}; recall that
$H(\cdot)$ denotes an arbitrary continuous function defined on the real line and with values in a compact interval $[\underline{H},\overline{H}]\subset(1/\al,1)$. 

First we provide a global modulus of continuity for $\{Y(t) : t\in\R\}$ on an arbitrary nonempty compact interval; there is no restriction to assume 
the latter interval of the form $[-M,M]$ where $M$ is an arbitrary positive real number.
\begin{thm}\label{modcont1Y}
Let $\Omega_0^{*}$ be the event of probability 1 introduced in Lemma~\ref{omega0}. Then for each $\omega\in\Omega_0^{*}$ and for all positive real numbers $M$ and 
$\eta$, one has,
\begin{equation}\label{eq1:modcont1Y}
\sup_{(t,s) \in [-M,M]^2} \left\{\frac{\big| \big( Y(t,\omega) - Y(s,\omega)\big|}{ |t-s|^{H(t)\vee H(s)-1/\alpha}\big(1+\big|\log |t-s| \big| \big)^{2/\alpha+\eta} +\big|H(t)-H(s)\big|} \right\}< \infty.
\end{equation}
\end{thm}

\begin{proof}[Proof of Theorem~\ref{modcont1Y}] The theorem easily results from (\ref{def:LMSMbis}) and Corollary~\ref{cor:modcont} in which one takes $q=0$,
$a=\min_{x\in [-M,M]} H(x)$ and $b=\max_{x\in [-M,M]} H(x)$.
\end{proof}

\begin{rem}
\label{rem:contY}
\begin{itemize}
\item[(i)] Theorem~\ref{modcont1Y} remains valid under the weaker condition that $H(\cdot)$ is a continuous function on the real line with values in the open 
interval $(1/\al,1)$; indeed, even in this case, $H\big ([-M,M]\big)$ is still a compact interval included in $(1/\al,1)$. 
\item[(ii)] A straightforward consequence of Theorem~\ref{modcont1Y} is that: LMSM has a modification with almost surely continuous paths, as soon as its functional Hurst parameter $H(\cdot)$ is a continuous function 
with values in $(1/\al,1)$; this solves and provides a positive answer to the conjecture made by Stoev and Taqqu in Remark~1 at page 166 of \cite{stoev2005path}.
\end{itemize}
\end{rem} 

The following corollary easily follows from Theorem~\ref{modcont1Y}.

\begin{corollary}
\label{modcont2Y}
\begin{itemize}
\item[(i)] Assume that for some real numbers $M_1<M_2$, one has for each $\eta>0$,
\begin{equation}
\label{condeq2:modcont2Y}
\sup_{(t,s)\in [M_1,M_2]^2} \frac{\big|H(t)-H(s)\big|}{|t-s|^{H(t)\vee H(s)-1/\alpha} \big(1+\big|\log |t-s|\big| \big)^{2/\alpha+\eta}}<\infty,
\end{equation}
then it follows that, for all $\omega\in\Omega_{0}^*$ and $\eta>0$, 
\begin{equation}
\label{eq2:modcont2Y}
\sup_{(t,s)\in [M_1,M_2]^2} \left\{\frac{\big| Y(t,\omega)-Y(s,\omega)\big|}{ |t-s|^{H(t)\vee H(s)-1/\alpha} \big(1+\big|\log |t-s|\big| \big)^{2/\alpha+\eta}} \right\}< \infty.
\end{equation}
\item[(ii)] Assume that for some real numbers $M_1<M_2$, one has for each $\eta>0$,
\begin{equation}
\label{condeq3:modcont2Y}
\sup_{(t,s)\in [M_1,M_2]^2} \frac{\big|H(t)-H(s)\big|}{|t-s|^{\min_{x\in [M_1,M_2]} H(x)-1/\alpha} \big(1+\big|\log |t-s|\big| \big)^{2/\alpha+\eta}}<\infty,
\end{equation}
then it follows that, for all $\omega\in\Omega_{0}^*$ and $\eta>0$, 
\begin{equation}
\label{eq3:modcont2Y}
\sup_{(t,s)\in [M_1,M_2]^2} \left\{\frac{\big|Y(t,\omega)-Y(s,\omega)\big|}{ |t-s|^{\min_{x\in [M_1,M_2]} H(x)-1/\alpha} \big(1+\big|\log |t-s| \big| \big)^{2/\alpha+\eta}} \right\}< \infty.
\end{equation}
\end{itemize}
\end{corollary}

\begin{rem}
\label{r1modcont2Y}
\begin{itemize}
\item[(i)] The Condition (\ref{condeq2:modcont2Y}) is satisfied as soon as
$$
H(\cdot)\in\ce^{\max_{x\in [M_1,M_2]} H(x)-1/\al}\big([M_1,M_2],\R\big).
$$
\item[(ii)] The Condition (\ref{condeq3:modcont2Y}) is satisfied as soon as
$$
H(\cdot)\in\ce^{\min_{x\in [M_1,M_2]} H(x)-1/\al}\big([M_1,M_2],\R\big).
$$
\end{itemize}
\end{rem}

Let us now provide a local modulus of continuity for  $\{Y(t) : t\in\R\}$.
\begin{thm}
\label{locmodcont1Y}
Assume that the skewness intensity function $\beta(\cdot)$ of the $\stas$ measure $Z_\al(ds)$ is a constant. Let $t_0\in\R$ be arbitrary and fixed.
Then, one has almost surely, for all positive real numbers $M$ and $\eta$,
\begin{equation}
\label{eq1:locmodcont1Y}
\sup_{t\in [-M,M]} \left\{\frac{\big| Y(t)-Y(t_0)\big|}{ |t-t_0|^{H(t_0)} \big(1+\big|\log |t-t_0|\big| \big)^{1/\alpha+\eta}+\big|H(t)-H(t_0)\big|} \right\}< \infty.
\end{equation}
\end{thm}

\begin{proof}[Proof of Theorem~\ref{locmodcont1Y}] First observe that for any fixed $t_0\in\R$, the process $\big\{X(t,H(t_0)):t\in\R\big\}$ has stationary increments since it is a Linear Fractional Stable Motion of Hurst parameter $H(t_0)$; hence, the processes $\big\{X(t,H(t_0))-X(t_0,H(t_0)) :t\in\R\big\}$ and $\big\{X(t-t_0,H(t_0)): t\in\R\big\}$ have the same finite dimensional distributions. Therefore, using their path continuity, and the fact that the set of the dyadic numbers in $[-M,M]$ is dense in $[-M,M]$, it follows that the random variables,
$$
\sup_{t\in [-M,M]} \left\{\frac{\big| X(t,H(t_0))-X(t_0,H(t_0))\big|}{ |t-t_0|^{H(t_0)} \big(1+\big|\log |t-t_0|\big| \big)^{1/\alpha+\eta}}\right\}
$$
and
$$
\sup_{t\in [-M,M]} \left\{\frac{\big| X(t-t_0,H(t_0))\big|}{ |t-t_0|^{H(t_0)} \big(1+\big|\log |t-t_0|\big| \big)^{1/\alpha+\eta}}\right\},
$$
are equals in law; thus, taking in Proposition~\ref{prop:asympinf}, $q=0$ and $a,b$ such that $H(t_0)\in [a,b]$, one gets that, almost surely,
\begin{equation}
\label{eq2:locmodcont1Y}
\sup_{t\in [-M,M]} \left\{\frac{\big| X(t,H(t_0))-X(t_0,H(t_0))\big|}{ |t-t_0|^{H(t_0)} \big(1+\big|\log |t-t_0|\big| \big)^{1/\alpha+\eta}}\right\}<\infty.
\end{equation}
On the other hand, taking in (\ref{eq2:CTWSEbis}), $q=0$, $a=\underline{H}:=\inf_{x\in\R} H(x)$ and $b=\overline{H}:=\sup_{x\in\R} H(x)$, one obtains that,
\begin{equation}
\label{eq3:locmodcont1Y}
\sup_{t\in [-M,M]} \left\{\frac{\big| X(t,H(t))-X(t,H(t_0))\big|}{\big|H(t)-H(t_0)\big|}\right\}<\infty.
\end{equation}
Finally putting together, (\ref{def:LMSMbis}), (\ref{eq2:locmodcont1Y}) and (\ref{eq3:locmodcont1Y}), it follows that (\ref{eq1:locmodcont1Y}) holds.
\end{proof}

The following result is a straightforward consequence of Theorem~\ref{locmodcont1Y}. 
\begin{corollary}
\label{cor:locmodcont1Y}
Assume that the skewness intensity function $\beta(\cdot)$ of the $\stas$ measure $Z_\al(ds)$ is a constant. Also assume that $t_0\in\R$ is such that,
for each $\eta>0$, one has for all $t\in\R$,
\begin{equation}
\label{eq4:locmodcont1Y}
\big|H(t)-H(t_0)\big|\le c|t-t_0|^{H(t_0)} \big(1+\big|\log |t-t_0|\big| \big)^{1/\alpha+\eta},
\end{equation}
where $c>0$ is a constant only depending on $t_0$ and $\eta$. Then, one has almost surely, for each positive real numbers $M$ and $\eta$,
\begin{equation}
\label{eq5:locmodcont1Y}
\sup_{t\in [-M,M]} \left\{\frac{\big| Y(t)-Y(t_0)\big|}{ |t-t_0|^{H(t_0)} \big(1+\big|\log |t-t_0|\big| \big)^{1/\alpha+\eta}} \right\}< \infty.
\end{equation}
\end{corollary}

\section{Quasi-optimality of global modulus of continuity of LMSM}
\label{subsec:quasiopt}
The goal of this section is to show that, under some conditions, a bit stronger than (\ref{condeq3:modcont2Y}), the global modulus of continuity, given in (\ref{eq3:modcont2Y}), is quasi-optimal, more precisely:
\begin{thm}
\label{optim:modcont}
Assume that $M_1<M_2$ are two arbitrary fixed real numbers such that the condition,
\begin{itemize}
\item[$(\mathcal{A}):$] $H(\cdot)$ belongs to the H\"older space $\ce^{\gamma_*}\big([M_1,M_2],\R\big)$ for some $\gamma_* \in \big (\min_{x\in [M_1,M_2]} H(x)-1/\al,1\big]$,
\end{itemize}
is satisfied. Let us set 
\begin{equation}
\label{eq:defrho}
\rho:=\sup\left\{\theta\in\R_+\,:\,\exists\, t_0\in [M_1,M_2] \mbox{ s.t. }  H(t_0)=\min_{x\in [M_1,M_2]} H(x) \mbox{ and } \sup_{t\in [M_1,M_2]}\frac{|H(t)-H(t_0)|}{|t-t_0|^{\theta}}<\infty\right\}
\end{equation}
and 
\begin{equation}
\label{eq:deftau}
\tau:= \frac{1+2\alpha^{-1}}{\alpha\rho -1},
\end{equation}
with the convention that $\tau:=0$ when $\rho=+\infty$. Assume that 
\begin{equation}
\label{eq:inegalrho}
\al\rho>1,
\end{equation}
then, $\tau$ is a well-defined nonnegative real number, and one has, almost surely, for all $\eta >0$, 
\begin{equation}
\label{eq1:optim}
\sup_{(t,s)\in [M_1,M_2]^2} \left\{\frac{ |Y(t)-Y(s)|}{ |t-s|^{\min_{x\in [M_1,M_2]} H(x)-1/\alpha}\big( 1+ \big| \log |t-s| \big| \big)^{-\tau-\eta}  } \right\}= \infty. 
\end{equation}  
\end{thm}

\begin{rem}
\label{rem:tau}
Notice that the Conditions $(\mathcal{A})$ and (\ref{eq:inegalrho}) are satisfied when $H(\cdot)$ belongs to the H\"older space $\ce^{\gamma}\big([M_1,M_2],\R\big)$, for some 
$\gamma>1/\al$. 
\end{rem}

%
%

In order to prove Theorem~\ref{optim:modcont}, we need some preliminary results. Let us first introduce $\widetilde{\Psi}$ the real-valued deterministic continuous function defined, for all $(x,v)\in \R\times (1/\alpha,1)$, as, 
\begin{equation}\label{PSIder}
\widetilde{\Psi}(x,v):= \frac{1}{\Gamma(v+1-1/\alpha)\Gamma(1/\alpha-v+1)}\int_{\R} (s-x)_{+}^{1/\alpha-v} \psi^{(2)}(s) ds,
\end{equation}
where $\psi^{(2)}$ is the second derivative of the Daubechies mother wavelet $\psi$ introduced at the very beginning of Section~\ref{sec:wavelets}, and where $\Gamma$ is the usual Gamma function; also, recall that the definition of $(\cdot)_{+}^{1/\alpha-v}$ is given in (\ref{eq:adpat+}).  
By using a result in \cite{Samko93} concerning Fourier transforms of right-sided fractional derivatives, one has for each $(\xi,v)\in \R\times (1/\alpha,1)$,  
\begin{equation} 
\label{FPSIder}
 \widehat{\widetilde{\Psi}}(\xi,v)= \frac{1}{\Gamma(v+1-1/\alpha)} |\xi|^{v+1-1/\alpha} e^{-i\textrm{sgn}(\xi)(v+1-1/\alpha)\frac{\pi}{2}}  \widehat{\psi}(\xi),
\end{equation}
where $\widehat{\widetilde{\Psi}}(\cdot,v)$ denotes the Fourier transform of the function $\widetilde{\Psi}(\cdot,v)$.
Let us now give some useful properties of the function $\widetilde{\Psi}$.

\begin{proposition}\label{orthopsi}
The function $\widetilde{\Psi}$ satisfies the following three properties.
\begin{itemize}
\item[(i)] For all real numbers
  $a,b$ such that $1>b>a>1/\alpha$, the function $\widetilde{\Psi}$ is well-localized in the variable $x$ uniformly in the variable $v\in [a,b]$; namely one has,
\begin{equation}
\label{localpsitilde}
\sup_{(x,v)\in\R\times [a,b]} (3+|x|)^2 \big|\widetilde{\Psi}(x,v)\big|< \infty.
\end{equation}
\item[(ii)] For any fixed $v\in \big(1/\alpha,1\big)$, the first moment of the function $\widetilde{\Psi}(\cdot,v)$ vanishes, which means that
\begin{equation}
 \int_{\R} \widetilde{\Psi}(x,v) dx =0. \label{mompsi}
\end{equation}
\item[(iii)] Let $\Psi$ be the function introduced in (\ref{PSI}) then, for each fixed $v \in \big(1/\alpha,1\big)$, the system of functions $\big\lbrace 2^{j/2} \Psi(2^j\cdot-k,v) : (j,k)\in\Z^2\big\rbrace$ and $\big\lbrace 2^{j/2} \widetilde{\Psi}(2^j\cdot-k,v) : (j,k)\in\Z^2\big\rbrace$ is biorthogonal; this means that for any $j\in\Z$, $j'\in\Z$, $k\in\Z$ and $k'\in\Z$, one has, 
\begin{equation}\label{OTN}
2^{(j+j')/2}\int_{\R} \Psi(2^jt-k,v)\widetilde{\Psi}(2^{j'}t-k',v) dt =\delta_{(j,k;j',k')}, 
\end{equation}
where $\delta_{(j,k;j',k')}= 1$ if $(j,k)=(j',k')$ and $0$ otherwise.
\end{itemize}
\end{proposition}

\begin{proof}[Proof of Proposition \ref{orthopsi}] Part $(i)$ can be obtained by using the fact that 
$$
\sup_{v\in [a,b]} \left\{\frac{1}{\Gamma(v+1-1/\alpha)\Gamma(1/\alpha-v+1)}\right\}<\infty
$$
and a method similar to the one used in the proof of Proposition~\ref{regulpsi} Part $(ii)$. In view of the definition of a Fourier transform, taking in (\ref{FPSIder}) $\xi=0$ one gets Part $(ii)$. Let us now give the proof of Part $(iii)$. Using Parseval formula, (\ref{FPSI}) and (\ref{FPSIder}), one obtains for all $(j,k)\in\Z^2$,
\begin{align}
& \int_{\R} 2^{j/2}\Psi(2^jt-k,v) 2^{j'/2}\widetilde{\Psi}(2^{j'}t-k',v) dt \nonumber \\
& = 2^{-(j+j')/2}(2\pi)^{-1} \int_{\R} e^{-i\xi(k/2^j-k'/2^{j'})} \widehat{\Psi}(2^{-j}\xi,v) \overline{\widehat{\widetilde{\Psi}}(2^{-j'}\xi,v)} d\xi \nonumber \\
& = 2^{-(j+j')/2+(j-j')(v+1-1/\alpha)} (2\pi)^{-1}  \int_{\R} e^{-i\xi(k/2^j-k'/2^{j'})} \widehat{\psi}(2^{-j}\xi) \overline{\widehat{\psi}(2^{-j'}\xi)} d\xi \nonumber \\
& = 2^{(j-j')(v+1-1/\alpha)} \int_{\R} 2^{j/2} \psi(2^jt-k) 2^{j'/2} \psi(2^{j'}t-k') dt \nonumber \\
& = \delta_{(j,k;j',k')}, \nonumber
\end{align}
where the last equality results from the fact that $\big\lbrace 2^{j/2}\psi(2^j\cdot-k) : (j,k)\in\Z^2\big\rbrace$ is an orthonormal basis for $L^2(\R)$.
\end{proof}

In all the remaining of this section, $M_1<M_2$ denote two arbitrary real numbers such that the Conditions $(\mathcal{A})$ and (\ref{eq:inegalrho}) hold. For the sake of simplicity, we set,
\begin{equation}
\label{Hstarl}
H_*:=\min_{x\in [M_1,M_2]} H(x).
\end{equation}

\begin{lemma}\label{optim:lem1}
Let $\Omega_0^{*}$ be the event of probability 1 introduced in Lemma~\ref{omega0} and let $\big\{g_{j,k}\,:\, (j,k)\in\N\times\Z\big\}$ be the sequence of the random variables defined on $\Omega_0^{*}$ as, 
\begin{equation}
\label{def:gjk}
g_{j,k} = 2^{j(1+H_{*})} \int_{\R} Y(t) \widetilde{\Psi}(2^jt-k,H_{*}) dt. 
\end{equation}
Assume that there exists $\omega_0 \in \Omega_0^{*}$, $\tau_0 > \tau$ and $\eta_0 >0$ such that
\begin{equation}\label{eq1:l1optim}
\sup_{(t,s)\in [M_1,M_2]^2} \frac{ \big|Y(t,\omega_0)-Y(s,\omega_0) \big|}{ |t-s|^{H_{*}-1/\alpha} \big( 1+ \big| \log |t-s| \big| \big)^{-\tau_0-\eta_0} } <\infty.
\end{equation}
Then one has
\begin{equation}\label{eq2:l1optim}
\limsup_{j \rightarrow +\infty} j^{\tau_0} 2^{-j/\alpha} \max\Big\{ \big|g_{j,k}(\omega_0)\big|: k \in\Z \mbox{ and } M_1+2^{-\frac{j}{2\al}}\le k/2^j\le   M_2-2^{-\frac{j}{2\al}}\Big\} = 0.
\end{equation}
\end{lemma}

\begin{rem}
\label{rem:wedegjk}
Notice that (\ref{localpsitilde}) (in which one takes $a,b$ such that $H_*\in [a,b]$), Proposition~\ref{prop:asympinf} (in which one takes $q=0$, $a=\underline{H}:=\inf_{x\in\R} H(x)$,
$b=\overline{H}:=\sup_{x\in\R} H(x)$ and $\eta$ an arbitrary positive real number) and Relation (\ref{def:LMSMbis}), imply that the random variables $g_{j,k}$ are well-defined and finite on $\Omega_{0}^*$. 
\end{rem}

\begin{proof}[Proof of Lemma \ref{optim:lem1}] In all the sequel, we assume that $j\in\N$ and $k\in\Z$ are arbitrary and
satisfy 
\begin{equation}
\label{eq1bis:l1optim}
M_1+2^{-\frac{j}{2\al}}\le \frac{k}{2^j}\le M_2-2^{-\frac{j}{2\al}}.
\end{equation}
It follows from (\ref{def:gjk}) and (\ref{mompsi}) in which one takes $v=H_*$, that, 
\begin{equation}\label{eq3:l1optim}
g_{j,k}(\omega_0) = 2^{j(1+H_*)} \int_{\R} \big( Y(t,\omega_0)-Y(k2^{-j},\omega_0) \big) \widetilde{\Psi}(2^jt-k,H_*) dt.
\end{equation}
In order to  conveniently bound $\big|g_{j,k}(\omega_0)\big|$, we split the integration domain $\R$ into the following three disjoint subdomains:
\begin{equation}\label{def1:l1optim}
\B_1:=[M_1,M_2]\mbox{, } \B_2:=[-2M_0,2M_0]\setminus [M_1,M_2] \mbox{ and } \B_3:=\R\setminus [-2M_0,2M_0], \mbox{where } M_0:=|M_1|+|M_2|.
\end{equation}
Therefore, (\ref{eq3:l1optim}) implies that,
\begin{equation}\label{eq4:l1optim}
\big|g_{j,k}(\omega_0)\big| \leq \sum_{l=1}^3 A_{j,k}^l(\omega_0),
\end{equation}
where, for each $l\in\{1,2,3\}$, one has set,
\begin{equation}\label{eq5:l1optim}
A_{j,k}^l(\omega_0) = 2^{j(1+H_*)} \int_{\B_l} \big|Y(t,\omega_0)-Y(k2^{-j},\omega_0)\big|  \big| \widetilde{\Psi}(2^jt-k,H_*)\big| dt.
\end{equation}
First, we show that (\ref{eq2:l1optim}) holds when $\big|g_{j,k}(\omega_0)\big|$ is replaced by $A_{j,k}^1(\o_0)$. Relation (\ref{eq1:l1optim}) and the change of variable $u=2^jt-k$, yield
\begin{align}
\label{eq6:l1optim}
A_{j,k}^1(\omega_0) & \leq C_1(\omega_0) 2^{j(1+H_*)} \int_{\B_1} |t-k2^{-j}|^{H_*-1/\alpha} \Big(1+ \big| \log |t-k2^{-j}| \big| \Big)^{-\tau_0-\eta_0} |\widetilde{\Psi}(2^jt-k,H_*)| dt \nonumber \\
& \leq C_1(\omega_0) 2^{j(1+H_*)} \int_{\R} |t-k2^{-j}|^{H_*-1/\alpha} \Big(1+ \big|\log |t-k2^{-j}| \big| \Big)^{-\tau_0-\eta_0} |\widetilde{\Psi}(2^jt-k,H_*)| dt \nonumber \\
& = C_1(\omega_0) 2^{j/\alpha} \int_{\R} |u|^{H_{*}-1/\alpha} \Big(1+ \big| \log |2^{-j}u| \big| \Big)^{-\tau_0-\eta_0} |\widetilde{\Psi}(u,H_*)| du \nonumber \\
& = C_1(\omega_0) j^{-\tau_0-\eta_0}2^{j/\alpha} \int_{\R} |u|^{H_*-1/\alpha} \Bigg( \frac{1}{j} + \bigg|\log(2) - \frac{\log |u|}{j} \bigg| \Bigg)^{-\tau_0-\eta_0} |\widetilde{\Psi}(u,H_*)| du, 
\end{align}
where
$$
C_1(\o_0):=\sup_{(t,s)\in\B_1^2} \frac{ \big|Y(t,\omega_0)-Y(s,\omega_0) \big|}{ |t-s|^{H_{*}-1/\alpha} \big( 1+ \big| \log |t-s| \big| \big)^{-\tau_0-\eta_0} }<\infty.
$$
Let us now show that,
\begin{equation}\label{eq7:l1optim}
\sup_{j\geq 1} \int_{\R} |u|^{H_*-1/\alpha} \Bigg( \frac{1}{j} + \bigg| \log(2) - \frac{\log |u|}{j} \bigg| \Bigg)^{-\tau_0-\eta_0} |\widetilde{\Psi}(u,H_*)| du <\infty.
\end{equation}
In view of (\ref{localpsitilde}) and the inequality, 
$$
\Bigg( \frac{1}{j} + \bigg| \log(2) - \frac{\log |u|}{j} \bigg| \Bigg)^{-\tau_0-\eta_0}\le  \Big( \frac{\log 2}{2} \Big)^{-\tau_0-\eta_0},
$$
which holds for all real number $u$ satisfying $|u|\leq 2^{j/2}$, one gets, for some constants $c_2,\ldots, c_5$ and all integer $j\ge 1$, that,
\begin{align}
& \int_{\R} |u|^{H_*-1/\alpha} \Bigg( \frac{1}{j} + \bigg| \log(2) - \frac{\log |u|}{j}\bigg| \Bigg)^{-\tau_0-\eta_0} |\widetilde{\Psi}(u,H_*)| du \nonumber \\
& \leq c_2 j^{\tau_0+\eta_0} \int_{|u|>2^{j/2}} \frac{|u|^{H_*-1/\alpha}}{(3+|u|)^2}du + c_2 \Big( \frac{\log 2}{2} \Big)^{-\tau_0-\eta_0} \int_{|u|\leq 2^{j/2}} \frac{|u|^{H_*-1/\alpha}}{(3+|u|)^2}du \nonumber \\
& \leq 2 c_2 j^{\tau_0+\eta_0} \int_{u>2^{j/2}} \frac{u^{H_*-1/\alpha}}{(3+u)^2}du + c_3 \int_{\R} \frac{|u|^{H_*-1/\alpha}}{(3+|u|)^2}du \nonumber \\
& \leq c_4 j^{\tau_0+\eta_0} \frac{2^{-j/2(1+1/\alpha-H_*)}}{1+1/\alpha-H_*}+c_5, \nonumber
\end{align}
which shows that (\ref{eq7:l1optim}) is satisfied. Next, (\ref{eq6:l1optim}) and (\ref{eq7:l1optim}) entail that
\begin{equation}
\label{eq8:l1optim}
\limsup_{j \rightarrow +\infty} j^{\tau_0} 2^{-j/\alpha} \max\Big\{ A_{j,k}^1(\omega_0):  k \in\Z \mbox{ and } M_1+2^{-\frac{j}{2\al}}\le k/2^j\le   M_2-2^{-\frac{j}{2\al}}\Big\} = 0.
\end{equation}
Next, we prove that (\ref{eq2:l1optim}) holds when $\big|g_{j,k}(\omega_0)\big|$ is replaced by $A_{j,k}^2(\o_0)$. Let us set,
\begin{equation}
\label{eq8bis:l1optim}
C_6 (\o_0):= \sup_{t\in [-2M_0,2M_0]}\big|Y(t,\omega_0)\big|<\infty;
\end{equation}
observe that $C_6 (\o_0)$ is finite, since the function $t\mapsto Y(t,\omega_0)$ is continuous over the compact interval $[-2M_0,2M_0]$. Also, observe that, in view of 
(\ref{eq1bis:l1optim}) and (\ref{def1:l1optim}), one has that for all $t\in\B_2$,
$$
|2^j t-k|>2^{j(1-\frac{1}{2\al})}.
$$
Therefore, it follows from (\ref{localpsitilde}), that for each $t\in\B_2$,
\begin{equation}
\label{eq8ter:l1optim}
|\widetilde{\Psi}(2^jt-k,H_*)|\le c_7 2^{-j(2-\frac{1}{\al})},
\end{equation}
where $c_7$ is a constant non depending on $t$, $j$ and $k$. Putting together, (\ref{eq5:l1optim}), (\ref{eq8bis:l1optim}) and (\ref{eq8ter:l1optim}), one gets that,
$$
A_{j,k}^2(\omega_0)\le C_8 (\o_0) 2^{-j(1-H_{*}-\frac{1}{\al})},
$$
where $C_8 (\o_0)$ is a constant non depending on $j$ and $k$. The latter inequality and the inequality $H_*<1$, imply that,
\begin{equation}
\label{eq8quat:l1optim}
\limsup_{j \rightarrow +\infty} j^{\tau_0} 2^{-j/\alpha} \max\Big\{ A_{j,k}^2(\omega_0): k \in\Z \mbox{ and } M_1+2^{-\frac{j}{2\al}}\le k/2^j\le   M_2-2^{-\frac{j}{2\al}}\Big\} = 0.
\end{equation}
Next, we prove that (\ref{eq2:l1optim}) holds when $\big|g_{j,k}(\omega_0)\big|$ is replaced by $A_{j,k}^3(\o_0)$. Observe that by using the triangle inequality, (\ref{eq1bis:l1optim}) and (\ref{def1:l1optim}), one has, for each $t\in\B_3$,
$$
|2^j t-k|=2^{j}\left|t-\frac{k}{2^j}\right|\ge 2^j \left (|t|-\frac{|k|}{2^j}\right)> 2^j\left (|t|-M_0\right)>2^{j-1}|t|.
$$
Therefore, it follows from (\ref{localpsitilde}), that for each $t\in\B_3$,
\begin{equation}
\label{eq9:l1optim}
|\widetilde{\Psi}(2^jt-k,H_*)|\le c_{9} 2^{-2j} |t|^{-2},
\end{equation}
where $c_9$ is a constant non depending on $t$, $j$ and $k$. On the other hand, using (\ref{def:LMSMbis}) and Proposition~\ref{prop:asympinf}, in the case where $q=0$, $a=\underline{H}
:=\inf_{x\in\R} H(x)$ and $b=\overline{H}:=\sup_{x\in\R} H(x)$, one obtains that for any fixed $\eta>0$, and for each $t\in\B_3$,
$$
|Y(t,\omega_0)|\le C_{10}(\o_0) |t|^{\overline{H}} \big(1+\big|\log |t| \big|\big)^{1/\alpha+\eta},
$$
where $C_{10}(\o_0)$ is a positive finite constant non depending on $t$. Next, combining the latter inequality with (\ref{eq1bis:l1optim}) and (\ref{eq8bis:l1optim}),
one gets that, for all $j\in\N$ and $k\in\Z$ satisfying (\ref{eq1bis:l1optim}), and for each $t\in\B_3$, one has,
\begin{equation}
\label{eq10:l1optim}
|Y(t,\omega_0)-Y(k2^{-j},\omega_0)|\le C_{11}(\o_0) |t|^{\overline{H}} \big(1+\big|\log |t| \big|\big)^{1/\alpha+\eta},
\end{equation}
where $C_{11}(\o_0)$ is a constant non depending on $j$, $k$ and $t$. Next, (\ref{eq5:l1optim}), (\ref{eq9:l1optim}) and (\ref{eq10:l1optim}), yield
$$
A_{j,k}^3(\omega_0)\le C_{12}(\o_0) 2^{-(1-H_*) j},
$$
where $C_{12}(\o_0)$ is a constant non depending on $j$ and $k$. Moreover, the latter inequality implies that,
\begin{equation}
\label{eq11:l1optim}
\limsup_{j \rightarrow +\infty} j^{\tau_0} 2^{-j/\alpha} \max\Big\{ A_{j,k}^3(\omega_0):  k \in\Z \mbox{ and } M_1+2^{-\frac{j}{2\al}}\le k/2^j\le   M_2-2^{-\frac{j}{2\al}}\Big\} = 0.
\end{equation}
Finally, putting together, (\ref{eq4:l1optim}), (\ref{eq8:l1optim}), (\ref{eq8quat:l1optim}) and (\ref{eq11:l1optim}), it follows that (\ref{eq2:l1optim}) holds.
\end{proof}

\begin{lemma}\label{optim:lem1bis}
Let $\Omega_0^{*}$ be the event of probability 1 introduced in Lemma~\ref{omega0} and let $\big\{\widetilde{g}_{j,k}\,:\, (j,k)\in\N\times\Z\big\}$ be the sequence of the random variables defined on $\Omega_0^{*}$ as, 
\begin{equation}
\label{def:gtjk}
\widetilde{g}_{j,k} = 2^{j(1+H_{*})} \int_{\R} X(t,H(k2^{-j})) \widetilde{\Psi}(2^jt-k,H_{*}) dt.
\end{equation}
Assume that $H(\cdot)$ satisfies the Condition $(\mathcal{A})$. Then, for each $\o\in\Omega_0^{*}$ and all \\
$\th\in \left[0,\min\{\ga_*+1/\al-H_*,1-H_*\}\right)$, one has,
\begin{equation}\label{eq1:l2optim}
\limsup_{j \vers \ii}  2^{j(\th-1/\alpha)} \max\Big\{ \big|g_{j,k}(\omega)-\widetilde{g}_{j,k}(\omega)\big| : k \in\Z \mbox{ and } M_1+2^{-\frac{j}{2\al}}\le k/2^j\le   M_2-2^{-\frac{j}{2\al}} \Big\}=0.
\end{equation}
\end{lemma}

\begin{rem}
\label{rem:wedegjkbis}
Notice that (\ref{localpsitilde}) (in which one takes $a,b$ such that $H_*\in [a,b]$) and Proposition~\ref{prop:asympinf} (in which one takes $q=0$, $a=\underline{H}$,
$b=\overline{H}$ and $\eta$ an arbitrary positive real number), imply that the random variables $\widetilde{g}_{j,k}$ are well-defined and finite on $\Omega_{0}^*$. 
\end{rem}

\begin{proof}[Proof of Lemma \ref{optim:lem1bis}] In all the sequel, we assume that $j\in\N$ and $k\in\Z$ are arbitrary and
satisfy (\ref{eq1bis:l1optim}). Using (\ref{def:LMSMbis}), (\ref{def:gjk}) and (\ref{def:gtjk}), one has,
\begin{equation}\label{eq4:l2optim}
|g_{j,k}(\omega)-\widetilde{g}_{j,k}(\omega)| \leq \sum_{l=1}^3 L_{j,k}^l(\omega),
\end{equation}
where for all $l\in\{1,2,3\}$,
\begin{equation}
\label{eq5:l2optim}
L_{j,k}^l(\omega) = 2^{j(1+H_*)} \int_{\B_l} |X(t,H(t),\omega)-X(t,H(k2^{-j}),\omega)| |\widetilde{\Psi}(2^jt-k,H_*)| dt;
\end{equation}
recall that the sets $\B_1$, $\B_2$ and $\B_3$ have been defined in (\ref{def1:l1optim}). Let us now prove that (\ref{eq1:l2optim}) holds, when $\big|g_{j,k}(\omega)-\widetilde{g}_{j,k}(\omega)\big|$ is replaced by $L_{j,k}^1(\omega)$. It follows from the definition of $\B_1$, (\ref{eq2:CTWSEbis}) (in which one takes $q=0$, $M=M_0$, $a=\underline{H}$ and  $b=\overline{H}$), (\ref{eq1bis:l1optim}), the Condition $({\cal A})$, and the change of variable $u=2^jt-k$, that,
\begin{align}
\label{eq6:l2optim}
L_{j,k}^1(\omega) & \leq C_1(\omega) 2^{j(1+H_*)}\int_{\B_1} |H(t)-H(k2^{-j})| |\widetilde{\Psi}(2^jt-k,H_*)| dt \nonumber \\
& \leq C_2(\omega) 2^{j(1+H_*)} \int_{\B_1} |t-k2^{-j}|^{\ga_*} |\widetilde{\Psi}(2^jt-k,H_*)| dt \nonumber \\
& \leq C_2(\omega) 2^{j(1+H_*)} \int_{\R} |t-k2^{-j}|^{\ga_*} |\widetilde{\Psi}(2^jt-k,H_*)| dt \nonumber \\
& = C_2(\omega) 2^{jH_*} \int_{\R} |2^{-j}u|^{\ga_*} |\widetilde{\Psi}(u,H_*)| du \nonumber \\
& \leq C_3(\omega) 2^{j(H_{*}-\ga_*)},
\end{align}
where the positive and finite constants $C_1(\o)$, $C_2(\o)$ and $C_3(\o)$, do not depend on $j$ and $k$. Then, using (\ref{eq6:l2optim}) and the inequality $\th <\ga_*+1/\al -H_*$, one gets that,
\begin{equation}
\label{eq7:l2optim}
\limsup_{j \vers \ii}  2^{j(\th-1/\alpha)} \max\Big\{ L_{j,k}^1(\omega) : k \in\Z \mbox{ and } M_1+2^{-\frac{j}{2\al}}\le k/2^j\le   M_2-2^{-\frac{j}{2\al}}\Big\}=0.
\end{equation}
Next, let us prove that (\ref{eq1:l2optim}) holds, when $\big|g_{j,k}(\omega)-\widetilde{g}_{j,k}(\omega)\big|$ is replaced by $L_{j,k}^2(\omega)$. Let us set,
\begin{equation}
\label{eq8:l2optim}
C_4 (\o):= \sup_{(u,v)\in [-2M_0,2M_0]\times [\underline{H},\overline{H}]}\big|X(u,v,\omega)\big|<\infty;
\end{equation}
observe that $C_4(\o)$ is finite, since the function $(u,v)\mapsto X(u,v,\omega)$ is continuous over the compact rectangle $[-2M_0,2M_0]\times [\underline{H},\overline{H}]$.
Putting together, (\ref{eq5:l2optim}), (\ref{eq8:l2optim}), and (\ref{eq8ter:l1optim}), one obtains that,
\begin{equation}
\label{eq9:l2optim}
L_{j,k}^2(\omega)\le C_5 (\o) 2^{-j(1-H_{*}-\frac{1}{\al})},
\end{equation}
where $C_5 (\o)$ is a constant non depending on $j$ and $k$. Then, using (\ref{eq9:l2optim}) and the inequality $\th <1-H_*$, it follows that,
\begin{equation}
\label{eq10bis:l2optim}
\limsup_{j \vers \ii}  2^{j(\th-1/\alpha)} \max\Big\{ L_{j,k}^2(\omega) : k \in\Z \mbox{ and } M_1+2^{-\frac{j}{2\al}}\le k/2^j\le   M_2-2^{-\frac{j}{2\al}}\}=0.
\end{equation}
Next, let us prove that (\ref{eq1:l2optim}) holds, when $\big|g_{j,k}(\omega)-\widetilde{g}_{j,k}(\omega)\big|$ is replaced by $L_{j,k}^3(\omega)$. Setting in Proposition~\ref{prop:asympinf}, $q=0$, $a=\underline{H}$ and $b=\overline{H}$, one gets that for any fixed $\eta>0$ and for each $t\in\B_3$,
$$
|X(t,H(t),\omega)-X(t,H(k2^{-j}),\omega)|\le C_{6}(\o) |t|^{\overline{H}} \big(1+\big|\log |t| \big|\big)^{1/\alpha+\eta},
$$
where $C_{6}(\o)$ is a constant non depending on $t$ and $(j,k)$. Next combining the latter inequality with (\ref{eq5:l2optim}) and (\ref{eq9:l1optim}), it follows that,
\begin{equation}
\label{eq11bis:l2optim}
L_{j,k}^3(\omega)\le C_{7}(\o) 2^{-(1-H_*) j},
\end{equation}
where $C_7 (\o)$ is a constant non depending on $j$ and $k$. Then, using (\ref{eq11bis:l2optim}) and the inequality $\th <1-H_*$, it follows that,
\begin{equation}
\label{eq12bisred:l2optim}
\limsup_{j \vers \ii}  2^{j(\th-1/\alpha)} \max\Big\{ L_{j,k}^3(\omega) : k \in\Z \mbox{ and } M_1+2^{-\frac{j}{2\al}}\le k/2^j\le   M_2-2^{-\frac{j}{2\al}}\Big\}=0.
\end{equation}
Finally, putting together, (\ref{eq4:l2optim}), (\ref{eq7:l2optim}), (\ref{eq10bis:l2optim}) and (\ref{eq12bisred:l2optim}), it follows that (\ref{eq1:l2optim}) holds.
\end{proof}


\begin{proposition}
\label{gcjk=ejk}
Let $\Omega_{0}^*$ be the event of probability $1$ which has been introduced in Lemma~\ref{omega0}. Then for all $\o\in \Omega_{0}^*$, $v\in (1/\al,1)$ and $(j,k)\in\Z^2$, one has, 
\begin{equation}
\label{eq1:gcjk=ejk}
2^{j(1+v)} \int_{\R} X(t,v,\o) \widetilde{\Psi}(2^jt-k,v) dt=\epsilon_{j,k}(\o),
\end{equation}
where $\epsilon_{j,k}$ is the random variable defined in (\ref{def:ejk}).
\end{proposition}

\begin{proof}[Proof of the Proposition \ref{gcjk=ejk}]
First observe that by using (\ref{localpsitilde}) and (\ref{eq1:pasymp}) in which one takes $q=0$ and $a,b$ such that $v\in [a,b]$, it follows that, for all $\o\in \Omega_{0}^*$ and $(j,k)\in\Z^2$,
$$
2^{j(1+v)}\left(\sum_{(j',k')\in\Z^2} 2^{-j'v} |\epsilon_{j',k'}(\o)| \big|\Psi(2^{j'}\cdot-k',v)-\Psi(-k',v)\big|\right) \big|\widetilde{\Psi}(2^j\cdot-k,v)\big|\in L_{t}^1(\R);
$$
therefore we are allowed to apply the dominated convergence Theorem, and we obtain, in view of Theorem~\ref{TWSE} Part $(i)$, that 
\begin{align*}
& 2^{j(1+v)} \int_{\R} X(t,v,\o) \widetilde{\Psi}(2^jt-k,v) dt\\
& =2^{j(1+v)}\sum_{(j',k')\in\Z^2} 2^{-j'v} \epsilon_{j',k'}(\o) \int_\R \big(\Psi(2^{j'}t-k',v)-\Psi(-k',v)\big) \widetilde{\Psi}(2^jt-k,v)dt.
\end{align*}
Finally, combining the latter equality with Proposition~\ref{orthopsi} Parts $(ii)$ and $(iii)$, one gets (\ref{eq1:gcjk=ejk}).
\end{proof}

\begin{rem}
\label{rem:existed}
Let $\tau$ and $\rho$ be as in Theorem~\ref{optim:modcont}, also we suppose that (\ref{eq:inegalrho}) holds. We denote by $\tau_0$ 
an arbitrary real number such that $\tau_0>\tau\ge 0$.
\begin{itemize}
\item[(i)] One has,
\begin{equation}
\label{eq:addinrt}
\frac{1+2\alpha^{-1}+\tau_0}{\rho} < \alpha \tau_0.
\end{equation}
\item[(ii)] Denote by $d(\tau_0)$ and $e(\tau_0)$ the positive real numbers defined as,
\begin{equation}
\label{eq:defdetau0}
d(\tau_0):= \frac{2}{3}\left(\frac{1+2\alpha^{-1}+\tau_0}{\rho}\right)+\frac{1}{3} (\alpha \tau_0) \,\,\mbox{ and }\,\, e(\tau_0):= \frac{1}{3}\left(\frac{1+2\alpha^{-1}+\tau_0}{\rho}\right)+\frac{2}{3} (\alpha \tau_0),
\end{equation}
then,
\begin{equation}
\label{ineg:de}
\frac{1+2\alpha^{-1}+\tau_0}{\rho} < d(\tau_0)<e(\tau_0)< \alpha \tau_0.
\end{equation}
\item[(iii)] For any fixed $t_0\in [M_1,M_2]$ and $j\in\N$, denote by $D_{j}(t_0,\tau_0)$ the set of indices, defined as,
\begin{equation}
\label{eq:defDjt0}
D_{j}(t_0,\tau_0):=\Big\{ k\in\Z \,:\, k2^{-j} \in [M_1,M_2] \mbox{ and } j^{-e(\tau_0)} \leq |t_0-k2^{-j}| \leq j^{-d(\tau_0)} \Big\},
\end{equation}
then, for all $j$ big enough, the set $D_{j}(t_0,\tau_0)$ is nonempty and satisfies,
\begin{equation}
\label{eq1:existed}
D_{j}(t_0,\tau_0)\subseteq \Big\{k\in\Z : M_1+2^{-\frac{j}{2\al}}\le k/2^j\le   M_2-2^{-\frac{j}{2\al}}\Big\}.
\end{equation}
\end{itemize}
\end{rem}

\begin{proof}[Proof of Remark~\ref{rem:existed}.] Observe that, in view of (\ref{eq:deftau}), one has 
$$
\frac{1+2\alpha^{-1}+\tau}{\rho}=\al \tau;
$$
therefore, (\ref{eq:inegalrho}), implies that Part $(i)$ holds. Part $(ii)$ easily follows from (\ref{eq:addinrt}) and (\ref{eq:defdetau0}). 
Let us give the proof of Part $(iii)$, for the sake of simplicity we set $d=d(\tau_0)$ and $e=e(\tau_0)$. Observe that the set $D_{j}(t_0,\tau_0)$ is nonempty 
for all $j$ big enough, since $\lim_{j\rightarrow +\infty} 2^j (j^{-d}-j^{-e})=+\infty$. Let $j\ge 1$ and $k$ be two 
arbitrary integers such that $j$ is big enough and $k\in D_{j}(t_0,\tau_0)$. 
In order to show that, they satisfy (\ref{eq1bis:l1optim}), we will study three cases: $t_0\in (M_1,M_2)$, $t_0=M_1$ and $t_0=M_2$. 

Let us first suppose that
$M_1<t_0<M_2$ i.e. $\min\{t_0-M_1, M_2-t_0\}>0$, then in view of the fact that $j$ is big enough, one can assume that $j^{-d}+2^{-\frac{j}{2\al}}\le \min\{t_0-M_1, M_2-t_0\}$; the latter inequality and the inequality $|t_0-k2^{-j}| \leq j^{-d}$ imply that (\ref{eq1bis:l1optim}) holds.

Let us now assume that $t_0=M_1$. It follows from the equality $|t_0-k2^{-j}|=k2^{-j}-M_1$ and the inequalities $j^{-e} \leq |t_0-k2^{-j}| \leq j^{-d}$, that, $M_1+j^{-e}\le k 2^{-j}\le M_1+j^{-d}$.
Moreover, in view of the fact that $j$ is big enough, one can assume $M_1+j^{-e}\ge M_1+ 2^{-\frac{j}{2\al}}$ and $M_1+j^{-d}\le M_2-2^{-\frac{j}{2\al}}$; thus (\ref{eq1bis:l1optim}) holds. At last, the case where $t_0=M_2$, can be treated similarly to the case where $t_0=M_1$.
\end{proof}




\begin{lemma}\label{optim:lem3}
Let $\tau$ be as in Theorem~\ref{optim:modcont}, also we suppose that (\ref{eq:inegalrho}) holds. We denote by $\tau_0$ an arbitrary fixed real number such that $\tau_0>\tau\ge 0$. Then, for all $t_0\in [M_1,M_2]$, there exists $\Omega_{1,\tau_0}^{*}(t_0)$ an event of probability~1 (which a priori depends on $\tau_0$ and $t_0$) included in $\Omega_{0}^*$ (recall that the latter event has been introduced in Lemma~\ref{omega0}), such that, for each $\o\in\Omega_{1,\tau_0}^{*}(t_0)$, one has,
\begin{equation}\label{eq1:l3optim}
\liminf_{j\vers \ii} j^{\tau_0} 2^{-j/\alpha} \max\Big\{ \big| \epsilon_{j,k}(\omega) \big| : k\in D_j(t_0,\tau_0) \Big\} >0,
\end{equation}
where the $\epsilon_{j,k}$'s are the random variables defined in (\ref{def:ejk}) and where $D_j(t_0,\tau_0)$ is the set introduced in (\ref{eq:defDjt0}).
\end{lemma}

\begin{proof}[Proof of Lemma \ref{optim:lem3}.] Let $p$ be a fixed integer such that $p >2R$ (see (\ref{eq:supp-psi-ant}) for the definition of $R$). We assume that $j$ is an arbitrary big enough integer, so that the set,
\begin{equation}\label{eq3:l3optim}
\overline{D_j}(t_0,\tau_0):= \big\{ q\in\Z \,: \,pq\in D_j (t_0,\tau_0)\big\}=\bigg\{ q\in\Z \,: \, \frac{pq}{2^j}\in [M_1,M_2] \mbox{ and } j^{-e(\tau_0)} \leq |pq2^{-j}-t_0| \leq j^{-d(\tau_0)} \bigg\}.
\end{equation}
is nonempty. From now on, for the sake of simplicity, $d(\tau_0)$ and $e(\tau_0)$ are respectively denoted by $d$ and $e$. Notice that, since $j$ is big enough, the cardinality of the set $\overline{D_j} (t_0,\tau_0)$ satisfies, 
\begin{equation}
\label{eq3bis:l3optim}
c_1 j^{-d} 2^{j} \le \mbox{card}\left(\overline{D_j} (t_0,\tau_0)\right)\le c_2 j^{-d} 2^{j},
\end{equation}
where $c_1$ and $c_2$ are two positive constants non depending on $j$. We denote by $\Gamma_j$ the event defined as,
\begin{equation}\label{eq2:l3optim}
\Gamma_j:=\left\{ \omega\in \Omega_0^{*}\,:\,\,\max\Big\{ \big| \epsilon_{j,k}(\omega)\big| : k\in D_j(t_0,\tau_0) \Big\}\leq j^{-\tau_0} 2^{j/\alpha} \right\}.
\end{equation}
Let us provide an appropriate upper bound for the probability $\PR\big(\Gamma_j\big)$; since $j$ is big enough, there is no restriction to suppose that $j^{-\tau_0} 2^{j/\alpha}\ge 1$ and 
that $c_3j^{\alpha\tau_0}2^{-j}<1$, where $c_3$ is the positive constant $c'$ in (\ref{eq4:l3optim}). Next, using, (\ref{eq2:l3optim}), (\ref{eq3:l3optim}), Part $(iv)$ of Remark~\ref{prop:gcjk}, (\ref{eq4:l3optim}), and the first inequality in (\ref{eq3bis:l3optim}), one obtains that,
\begin{align}\label{eq5:l3optim}
 \PR\big(\Gamma_j\big) & \leq \PR\left( \bigcap_{q\in \overline{D_j}(t_0,\tau_0)}\left\{\big|\epsilon_{j,pq}\big|\leq j^{-\tau_0} 2^{j/\alpha}\right\}\right) = \prod_{q \in \overline{D_j}(t_0,\tau_0)} \PR\Big(  \big|\epsilon_{j,pq}\big|\leq  j^{-\tau_0} 2^{j/\alpha}\Big) \nonumber \\
&=\prod_{q \in \overline{D_j}(t_0,\tau_0)} \left (1-\PR\Big(  \big|\epsilon_{j,pq}\big|>  j^{-\tau_0} 2^{j/\alpha}\Big) \right)\nonumber \\
& \leq \Big(1-c_3 j^{\alpha\tau_0}2^{-j}\Big)^{c_1 j^{-d} 2^{j}};
\end{align}
moreover, the inequality $\log(1-x)\le -x$ for all $x\in [0,1)$, allows to prove that,
\begin{equation}
\label{eq7:l3optim}
\Big(1-c_3 j^{\alpha\tau_0}2^{-j}\Big)^{c_1 j^{-d} 2^{j}}:=\exp\left(c_1 j^{-d} 2^{j}\log(1-c_3 j^{\alpha\tau_0}2^{-j})\right)\leq \exp\big( -c_1 c_3 j^{\alpha\tau_0-d} \big).
\end{equation}
Finally putting together (\ref{ineg:de}), (\ref{eq5:l3optim}) and (\ref{eq7:l3optim}), one gets that,
$$
\sum_{j\in\N}\PR\big( \Gamma_j \big) < \infty ;
$$
thus Borel-Cantelli Lemma implies that (\ref{eq1:l3optim}) holds.
\end{proof}

\begin{lemma}\label{optim:lem2}
Let $\tau$ be as in Theorem~\ref{optim:modcont} also we suppose that (\ref{eq:inegalrho}) holds. We denote by $\tau_0$ an arbitrary fixed real number such that $\tau_0>\tau\ge 0$. Then there exists $t_0\in [M_1,M_2]$ (a priori $t_0$ depends on $\tau_0$) such that, for all $\omega \in \Omega_0^*$ (the event of probability 1 introduced in Lemma~\ref{omega0}), one has,
\begin{equation}\label{eq2:l2optim}
\limsup_{j \vers \ii} j^{\tau_0} 2^{-j/\alpha} \max\Big\{ \big|\widetilde{g}_{j,k}(\omega)-\epsilon_{j,k}(\omega)\big| : k\in D_j(t_0,\tau_0)  \Big\}=0;
\end{equation}
recall that the random variables $\widetilde{g}_{j,k}$ and $\epsilon_{j,k}$ have been defined respectively in (\ref{def:gtjk}) and (\ref{def:ejk}), also recall that the 
set $D_j (t_0,\tau_0)$ has been introduced in (\ref{eq:defDjt0}).
\end{lemma}

\begin{proof}[Proof of Lemma \ref{optim:lem2}.] Let $\rho$ be as in (\ref{eq:defrho}).
Assume that $\rho_0\in (1/\al,\rho)$ is arbitrary and such that,
\begin{equation}
\label{inegbis:de}
\frac{1+2\alpha^{-1}+\tau_0}{\rho} <\frac{1+2\alpha^{-1}+\tau_0}{\rho_0}< d(\tau_0)<e(\tau_0)< \alpha \tau_0,
\end{equation}
where $d(\tau_0)$ and $e(\tau_0)$ are defined in (\ref{eq:defdetau0}).
Then, in view of (\ref{eq:defrho}) and (\ref{Hstarl}), there exists $t_0\in[M_1,M_2]$, which satisfies,
\begin{equation}
\label{eq2bis:optim}
\left\{
\begin{array}{l}
H(t_0)=H_*\\
\\
\sup_{t\in [M_1,M_2]}\frac{|H(t)-H(t_0)|}{|t-t_0|^{\rho_0}}<\infty.
\end{array}
\right.
\end{equation}
In all the sequel, we suppose that $j$ is an arbitrary big enough integer, thus the set $D_j(t_0,\tau_0)$ is nonempty and (\ref{eq1:existed}) holds; also we suppose that $k\in D_j(t_0,\tau_0)$ is arbitrary. Using (\ref{eq1:gcjk=ejk}) in which one takes $v=H_*$, (\ref{def:gtjk}), and the equality, for each fixed $t\in\R$, 
$$
X(t,H(k2^{-j}),\o) -X(t,H_*,\o)=\big (H(k2^{-j})-H_*\big)\int_{0}^{1}\big(\partial_v X\big)\big(t,H_*+\theta(H(k2^{-j})-H_*),\o\big) d\th, 
$$
one gets that
\begin{align}
\label{eq10ter:l2optim}
\widetilde{g}_{j,k}(\omega)-\epsilon_{j,k}(\omega) &= 2^{j(1+H_{*})} \int_{\R} \big ( X(t,H(k2^{-j}),\o) -X(t,H_*,\o)\big)\widetilde{\Psi}(2^jt-k,H_{*})dt\nonumber\\
&=2^{j(1+H_{*})} \big (H(k2^{-j})-H_*\big)\int_{\R} \int_{0}^{1}\big(\partial_v X\big)\big(t,H_*+\theta(H(k2^{-j})-H_*),\omega\big)\widetilde{\Psi}(2^jt-k,H_{*})d\theta dt.\nonumber
\end{align}
Therefore, it follows from (\ref{mompsi}) in which one takes $v=H_*$, that,
\begin{equation}
\label{eq10:l2optim}
|\epsilon_{j,k}(\omega)-\widetilde{g}_{j,k}(\omega)| \leq |H(k2^{-j})-H_*|  \sum_{l=1}^{3} F_{j,k}^l(\omega),
\end{equation}
where, for each $l\in\{1,2,3\}$
\begin{align}
\label{eq11:l2optim}
F_{j,k}^l(\omega) & = 2^{j(1+H_*)} \int_{\B_l} \int_{0}^{1} |\widetilde{\Psi}(2^jt-k,H_*)| \\
	& \hspace{2cm}\times \big| \big(\partial_v X\big)\big(t,H_*+\theta(H(k2^{-j})-H_*),\omega\big)-\big(\partial_v X\big)\big(k2^{-j},H_*+\theta(H(k2^{-j})-H_*),\omega\big)\big| d\theta dt;\nonumber
\end{align}
recall that the sets $\B_l$ are defined in (\ref{def1:l1optim}). Observe that, in view of (\ref{eq2bis:optim}) and (\ref{eq:defDjt0}), one has,
\begin{equation}
\label{eq10quat:l2optim}
|H(k2^{-j})-H_*|\le c_1 j^{-d(\tau_0)\rho_0},
\end{equation}
where $c_1$ is a constant non depending on $j$ and $k$. Also, observe that in view of (\ref{inegbis:de}), there
exists $\eta_1$, an arbitrarily small positive real number such that
\begin{equation}
\label{eq:optieta0}
\frac{1+2\alpha^{-1}+\tau_0+\eta_1}{\rho_0}<d(\tau_0).
\end{equation}
Let us now, prove that (\ref{eq2:l2optim}) holds when $\big|\widetilde{g}_{j,k}(\omega)-\epsilon_{j,k}(\omega)\big|$ is replaced by $|H(k2^{-j})-H_*|F_{j,k}^1(\omega)$. 
Using Proposition \ref{pp1:dX} (in which one takes 
$q=1$, $M=M_0$, $a=\underline{H}$, $b=\overline{H}$ and $\eta=\eta_1$), the inequality $H(k2^{-j})\ge H_*$ and the fact that $k2^{-j}\in\B_1\subset[-M_0,M_0]$, one gets that,
\begin{align*}
F_{j,k}^1(\omega) & \leq C_2(\omega) 2^{j(1+H_*)} \int_{\B_1} \int_0^1 |\widetilde{\Psi}(2^jt-k,H_*)| \nonumber \\
& \hspace{2cm}\times |t-k2^{-j}|^{H_*-1/\alpha+\theta(H(k2^{-j})-H_*)} \Big( 1+ \big| \log |t-k2^{-j}| \big| \Big)^{1+2/\alpha+\eta_1} d\theta dt \nonumber\\
&  \leq C_2(\omega) 2^{j(1+H_*)} \int_{\B_1} |\widetilde{\Psi}(2^jt-k,H_*)||t-k2^{-j}|^{H_*-1/\alpha}\Big( 1+ \big| \log |t-k2^{-j}| \big| \Big)^{1+2/\alpha+\eta_1} \nonumber \\
& \hspace{4cm}\times \Big\{\int_0^1 |t-k2^{-j}|^{\theta(H(k2^{-j})-H_*)}  d\theta \Big\}dt \nonumber\\
& \le C_3 (\o) 2^{j(1+H_*)} \int_{\R} |\widetilde{\Psi}(2^jt-k,H_*)||t-k2^{-j}|^{H_*-1/\alpha}\Big( 1+ \big| \log |t-k2^{-j}| \big| \Big)^{1+2/\alpha+\eta_1}dt,
\end{align*}
where $C_2(\o)$ is a constant non depending on $j$ and $k$ and where $C_3(\o)=\big(1+2M_0\big)^{\overline{H}-H_*}C_2(\o)$. Then, setting $u=2^j t-k$ in the last integral, and using Lemma~\ref{LA4}, one obtains that,
\begin{align}
\label{eq12:l2optim}
F_{j,k}^1(\omega) & \leq C_3(\omega) 2^{j/\alpha} \int_{\R} |u|^{H_*-1/\alpha} \Big( 1+ \big| \log |2^{-j}u| \big|  \Big)^{1+2/\alpha+\eta_1} |\widetilde{\Psi}(u,H_*)| du \nonumber \\
& \leq C_4(\omega) 2^{j/\alpha} \int_{\R} |u|^{H_*-1/\alpha} \Big( j^{1+2/\alpha+\eta_1} + \left(1+\big| \log |u| \big| \right)^{1+2/\alpha+\eta_1} \Big) |\widetilde{\Psi}(u,H_*)| du \nonumber \\
& \leq C_5(\omega) j^{1+2/\alpha+\eta_1} \, 2^{j/\alpha},
\end{align}
where $C_4(\o)$ and $C_5 (\o)$ are two constants non depending on $j$ and $k$. Putting together, (\ref{eq10quat:l2optim}), (\ref{eq:optieta0}) and (\ref{eq12:l2optim}),
it follows that 
\begin{equation}
\label{eq12bis:l2optim}
\limsup_{j \vers \ii} j^{\tau_0} 2^{-j/\alpha} \max\Big\{|H(k2^{-j})-H_*|F_{j,k}^1(\omega) : k\in D_j(t_0,\tau_0)  \Big\}=0.
\end{equation}
Let us now prove that (\ref{eq2:l2optim}) holds when $\big|\widetilde{g}_{j,k}(\omega)-\epsilon_{j,k}(\omega)\big|$ is replaced by $|H(k2^{-j})-H_*|F_{j,k}^2(\omega)$. 
We set,
\begin{equation}
\label{eq13:l2optim}
C_6 (\o):= \sup_{(u,v)\in [-2M_0,2M_0]\times [\underline{H},\overline{H}]}\big|(\partial_v X)(u,v,\omega)\big|<\infty;
\end{equation}
observe that $C_6(\o)$ is finite, since the function $(u,v)\mapsto (\partial_v X)(u,v,\omega)$ is continuous over the compact rectangle $[-2M_0,2M_0]\times [\underline{H},\overline{H}]$.
Putting together, (\ref{eq11:l2optim}), (\ref{eq13:l2optim}), (\ref{eq1:existed}) and (\ref{eq8ter:l1optim}), one obtains that,
\begin{equation}
\label{eq14:l2optim}
F_{j,k}^2(\omega)\le C_7 (\o) 2^{-j(1-H_{*}-\frac{1}{\al})},
\end{equation}
where $C_7 (\o)$ is a constant non depending on $j$ and $k$. Then, using (\ref{eq14:l2optim}), the fact that $H(\cdot)$ is a bounded function, and the inequality $0<1-H_*$, it follows that,
\begin{equation}
\label{eq15:l2optim}
\limsup_{j \vers \ii}  j^{\tau_0} 2^{-j/\alpha}\max\Big\{ |H(k2^{-j})-H_*|F_{j,k}^2(\omega) : k\in D_j(t_0,\tau_0) \Big\}=0.
\end{equation}
Let us now prove that (\ref{eq2:l2optim}) holds when $\big|\widetilde{g}_{j,k}(\omega)-\epsilon_{j,k}(\omega)\big|$ is replaced by $|H(k2^{-j})-H_*|F_{j,k}^3(\omega)$.
Setting in Proposition~\ref{prop:asympinf}, $q=1$, $a=\underline{H}$ and $b=\overline{H}$, one gets , in view of (\ref{eq13:l2optim}), that for any fixed $\eta>0$, for each $t\in\B_3$ and for all $\th\in [0,1]$,
$$
\big| \big(\partial_v X\big)\big(t,H_*+\theta(H(k2^{-j})-H_*),\omega\big)-\big(\partial_v X\big)\big(k2^{-j},H_*+\theta(H(k2^{-j})-H_*),\omega\big)\big|\le C_{8}(\o) |t|^{\overline{H}} \big(1+\big|\log |t| \big|\big)^{1+1/\alpha+\eta},
$$
where $C_{8}(\o)$ is a constant non depending on $t$, $\th$ and $(j,k)$. Next combining the latter inequality with (\ref{eq11:l2optim}) and (\ref{eq9:l1optim}), it follows that,
\begin{equation}
\label{eq16:l2optim}
F_{j,k}^3(\omega_0)\le C_{9}(\o) 2^{-(1-H_*) j},
\end{equation}
where $C_9 (\o)$ is a constant non depending on $j$ and $k$. Then, using (\ref{eq16:l2optim}), the fact that $H(\cdot)$ is a bounded function, and the inequality $0<1-H_*$, it follows that,
\begin{equation}
\label{eq17:l2optim}
\limsup_{j \vers \ii}  j^{\tau_0} 2^{-j/\alpha}\max\Big\{ |H(k2^{-j})-H_*|F_{j,k}^3(\omega) : k\in D_j(t_0,\tau_0)\Big\}=0.
\end{equation}
Finally, putting together, (\ref{eq10:l2optim}), (\ref{eq12bis:l2optim}) (\ref{eq15:l2optim}) and (\ref{eq17:l2optim}), it follows that (\ref{eq2:l2optim}) holds.
\end{proof}

\begin{lemma}
\label{optim:lem4}
Let $\tau$ be as in Theorem~\ref{optim:modcont}, also we suppose that the Conditions $(\mathcal{A})$ and (\ref{eq:inegalrho}) hold. We denote by $\tau_0$ an arbitrary fixed real number such that $\tau_0>\tau\ge 0$. Then there exists $\Omega_{2,\tau_0}^*$ an event of probability~1 (which a priori depends on $\tau_0$) included in $\Omega_{0}^*$ (recall that the latter event has been introduced in Lemma~\ref{omega0}), such that, for each $\o\in\Omega_{2,\tau_0}^{*}$, one has,
\begin{equation}\label{eq1:l4optim}
\liminf_{j\vers \ii} j^{\tau_0} 2^{-j/\alpha} \max\Big\{ \big| g_{j,k}(\omega) \big| :  k \in\Z \mbox{ and } M_1+2^{-\frac{j}{2\al}}\le k/2^j\le   M_2-2^{-\frac{j}{2\al}}\Big\} >0,
\end{equation}
where the $g_{j,k}$'s are the random variables defined in (\ref{def:gjk}).
\end{lemma}

\begin{proof}[Proof of Lemma \ref{optim:lem4}.] Putting together (\ref{eq1:existed}) and Lemmas~\ref{optim:lem2}, \ref{optim:lem3}, \ref{optim:lem1bis}; one gets the lemma.
\end{proof}

Now, we are in position to prove Theorem~\ref{optim:modcont}.

\begin{proof}[Proof of Theorem \ref{optim:modcont}.]
Denote by $\Omega_3 ^*$, the event of probability 1, defined as,
$$
\Omega_3 ^*:=\bigcap_{\tau_0\in \Q \,\mbox{{\tiny and }}\tau_0>\tau} \Omega_{2,\tau_0}^*;
$$
recall that the events $\Omega_{2,\tau_0}^*$ have been introduced in Lemma~\ref{optim:lem4}. It is clear that (\ref{eq1:l4optim}) holds
for all $\o\in \Omega_3 ^*$ and for all real number $\tau_0>\tau\ge 0$; therefore, it follows from Lemma~\ref{optim:lem1}, that for each $\o\in \Omega_3 ^*$, $\tau_0>\tau$ and $\eta_0>0$,
$$
\sup_{(t,s)\in [M_1,M_2]^2} \frac{ \big|Y(t,\omega)-Y(s,\omega) \big|}{ |t-s|^{H_{*}-1/\alpha} \big( 1+ \big| \log |t-s| \big| \big)^{-\tau_0-\eta_0} } =\infty.
$$
Then, in view of (\ref{Hstarl}), one gets the theorem.
\end{proof}

\section{Optimality of local modulus of continuity of LMSM}
\label{subsec:optlocmod}

The goal of this section is to show that under a condition a bit stronger than (\ref{eq4:locmodcont1Y}), the local modulus of continuity given in (\ref{eq5:locmodcont1Y}), is optimal, more precisely:

\begin{thm}
\label{locoptim:modcont}  
Let $M$ be an arbitrary positive real number. Assume that $t_0\in (-M,M)$ satisfies for some constant $c>0$ and 
all $t\in\R$, 
\begin{equation}
\label{eq1:tlocmodcont}
\big|H(t)-H(t_0)\big| \le c |t-t_0|^{H(t_0)} \big(1+\big|\log |t-t_0|\big| \big)^{1/\alpha}.
\end{equation}
Then, one has almost surely, 
\begin{equation}
\label{eq2:tlocmodcont}
\sup_{t\in [-M,M]} \left\{\frac{\big| Y(t)-Y(t_0)\big|}{ |t-t_0|^{H(t_0)} \big(1+\big|\log |t-t_0|\big| \big)^{1/\alpha}} \right\}= \infty.
\end{equation}
\end{thm}

\begin{rem}
Let us mention that, even in the quite classical case of Linear Fractional Stable Motion (LFSM) (in other words, in the particular case where the functional parameter $H(\cdot)$ of LMSM is a constant), the optimal lower bound of the power of the logarithmic factor in a local modulus of continuity, was unknown so far; Corollary~\ref{cor:locmodcont1Y}~and Theorem~\ref{locoptim:modcont} in our article, show that, in the more general case of LMSM, this optimal lower bound is in fact $1/\al$. 
\end{rem}

The proof of Theorem~\ref{locoptim:modcont} relies on (\ref{eq2:CTWSEbis}) in which one takes $q=0$, also, more importantly, it relies on the following proposition. 

\begin{proposition}
\label{prop:locmodcont}
Let $M$ be an arbitrary positive real number. For all $t_0\in (-M,M)$, one has almost surely, 
\begin{equation}
\label{eq1:plocmodcont}
\sup_{t\in [-M,M]} \left\{\frac{\big| X(t,H(t_0))-X(t_0,H(t_0))\big|}{ |t-t_0|^{H(t_0)} \big(1+\big|\log |t-t_0|\big| \big)^{1/\alpha}} \right\}= \infty.
\end{equation}
\end{proposition}

In order to show that Proposition~\ref{prop:locmodcont} holds, we need to introduce some additional notations, also we need to derive some preliminary results. Let $m_0$ be the positive integer defined as,
\begin{equation}
\label{eq:defm0}
m_0:=\big [\log_2 (3R+2)\big]+1;
\end{equation}
recall that, $R$ is a fixed real number strictly bigger than $1$, such that (\ref{eq:supp-psi-ant}) holds. For all $j\in\N$, 
one sets, 
\begin{equation}
\label{eq:ljm0}
r(j,m_0):=jm_0 \mbox{ and } l(j,m_0):=\big[2^{r(j,m_0)} t_0 +R+2\big];
\end{equation}
observe that, the inequalities,
\begin{equation}
\label{eq:inegrl}
(R+1)2^{-r(j,m_0)}< l(j,m_0)2^{-r(j,m_0)}-t_0 < (R+1)2^{1-r(j,m_0)}<4/5,
\end{equation}
hold. One denotes by $\check{\eps}_j$ the $\stas$ random variable,
\begin{equation}
\label{eq:defchecesp}
\check{\eps}_j :=\eps_{r(j,m_0),l(j,m_0)}\,,
\end{equation}
in other words, $\check{\eps}_j$ is defined through (\ref{def:ejk}) in which $j$ and $k$ are replaced, respectively by $r(j,m_0)$ and $l(j,m_0)$. 

\begin{lemma}
\label{lem1:locmodcont}
The $\stas$ random variables $\check{\eps}_j$, $j\in\N$ are independent and they all have the same scale parameter, namely, for each $j\in\N$, 
\begin{equation}
\label{eq:spchepsjk}
\|\check{\eps}_j \|_{\alpha} =  \Bigg\{ \int_{\R} |\psi(t)|^{\alpha} dt \Bigg\}^{1/\alpha}. 
\end{equation}
\end{lemma}

\begin{proof}[Proof of Lemma~\ref{lem1:locmodcont}.] First observe that (\ref{eq:spchepsjk}) is a straightforward consequence of (\ref{eq:defchecesp}) and (\ref{eq:spepsjk}). Let us now prove that the random variables $\check{\eps}_j$, $j\in\N$ are independent. Notice that (\ref{eq:supp-psi-ant}), entails that,
$$
\mbox{supp}\, \psi \big(2^{r(j,m_0)}\cdot-l(j,m_0)\big)\subseteq \big [l(j,m_0)2^{-r(j,m_0)}-R2^{-r(j,m_0)},l(j,m_0)2^{-r(j,m_0)}+R2^{-r(j,m_0)}\big];
$$
therefore, in view of (\ref{def:ejk}) and the fact that the $\stas$ random measure $Z_{\alpha}(ds)$ is independently scattered, it is sufficient to show that the intervals $\big [l(j,m_0)2^{-r(j,m_0)}-R2^{-r(j,m_0)},l(j,m_0)2^{-r(j,m_0)}+R2^{-r(j,m_0)}\big]$, $j\in\N$ are disjoint.
The latter result can be obtained by proving that, the inequality,
\begin{equation}
\label{eq1lem1:locmodcont}
R 2^{-r(j,m_0)}+R2^{-r(j+p,m_0)}<\big|l(j,m_0)2^{-r(j,m_0)}-l(j+p,m_0)2^{-r(j+p,m_0)}\big|,
\end{equation}
holds for all $(j,p)\in\N^2$. By using the triangle inequality, (\ref{eq:inegrl}) and the first equality in (\ref{eq:ljm0}), one has,
\begin{align}
\label{eq2lem1:locmodcont}
& \big|l(j,m_0)2^{-r(j,m_0)}-l(j+p,m_0)2^{-r(j+p,m_0)}\big|  \ge \big|l(j,m_0)2^{-r(j,m_0)}-t_0\big|-\big|l(j+p,m_0)2^{-r(j+p,m_0)}-t_0\big|\nonumber\\
& > (R+1)2^{-r(j,m_0)}-(R+1)2^{1-r(j+p,m_0)}= (R+1)2^{-jm_0}\big(1-2^{1-p m_0}\big)\ge (R+1)2^{-jm_0}\big(1-2^{1-m_0}\big).
\end{align}
On the other hand, the first equality in (\ref{eq:ljm0}), imply that 
\begin{equation}
\label{eq3lem1:locmodcont}
R 2^{-r(j,m_0)}+R2^{-r(j+p,m_0)}=R2^{-jm_0}\big(1+2^{-p m_0}\big)\le R 2^{-j m_0}\big(1+2^{-m_0}\big).
\end{equation}
Next, notice that (\ref{eq:defm0}), implies that $2^{-m_0}<(3R+2)^{-1}$ and consequently that,
\begin{equation}
\label{eq4lem1:locmodcont}
R \big(1+2^{-m_0}\big)<\frac{3R(R+1)}{3R+2}< (R+1)\big(1-2^{1-m_0}\big).
\end{equation}
Finally, putting together (\ref{eq2lem1:locmodcont}), (\ref{eq3lem1:locmodcont}) and (\ref{eq4lem1:locmodcont}), one gets (\ref{eq1lem1:locmodcont})
\end{proof}

\begin{lemma}
\label{lem2:locmodcont}
One has, almost surely,
\begin{equation}
\label{eq1lem2:locmodcont}
\limsup_{j\rightarrow +\infty} \frac{|\check{\eps}_j|}{j^{1/\al}\log^{1/\al}(j)}\ge 1.
\end{equation}
\end{lemma}
\begin{proof}[Proof of Lemma~\ref{lem2:locmodcont}.] Notice that, in view of Lemma~\ref{lem1:locmodcont}, the events $\left\{|\check{\eps}_j|>j^{1/\al}\log^{1/\al}(j)\right\}$, $j\in\N$, are independent; moreover, (\ref{eq:defchecesp}) and the first inequality in (\ref{eq4:l3optim}), imply that,
$$
\sum_{j=2}^{+\infty} \PR\left(|\check{\eps}_j|>j^{1/\al}\log^{1/\al}(j)\right)\ge c'\sum_{j=2}^{+\infty} j^{-1}\log^{-1}(j)=+\infty.
$$
Thus, applying the second Borel-Cantelli Lemma, one gets (\ref{eq1lem2:locmodcont}).
\end{proof}

\begin{lemma}
\label{lem3:locmodcont}
Let $\Omega_{0}^*$ be the event of probability 1 introduced in Lemma~\ref{omega0}. Assume that for some $t_0\in (-M,M)$ and $\o_0 \in \Omega_{0}^*$, one has,
\begin{equation}
\label{eq1lem3:locmodcont}
\sup_{t\in [-M,M]} \left\{\frac{\big| X(t,H(t_0),\o_0)-X(t_0,H(t_0),\o_0)\big|}{ |t-t_0|^{H(t_0)} \big(1+\big|\log |t-t_0|\big| \big)^{1/\alpha}} \right\}< \infty.
\end{equation}
Then, it follows that, 
\begin{equation}
\label{eq2lem3:locmodcont}
\limsup_{j\rightarrow +\infty} \frac{|\check{\eps}_j(\o_0)|}{j^{1/\al}}<\infty.
\end{equation}
\end{lemma}

\begin{proof}[Proof of Lemma~\ref{lem3:locmodcont}.] First notice that, (\ref{eq:defchecesp}), (\ref{eq1:gcjk=ejk}) in which one takes $v=H(t_0)$, (\ref{mompsi}), and the change of variable $x=t-l_j2^{-r_j}$, imply that,
\begin{align}
\label{eq3lem3:locmodcont}
\check{\eps}_j(\o_0)&=2^{r_j(1+H(t_0))} \int_{\R} \Big(X(t,H(t_0),\o_0)-X(t_0,H(t_0),\o_0)\Big) \widetilde{\Psi}(2^{r_j}t-l_j,H(t_0)) dt\nonumber\\
&=2^{r_j(1+H(t_0))} \int_{\R} \Big(X\big(x+l_j2^{-r_j},H(t_0),\o_0\big)-X\big(t_0,H(t_0),\o_0\big)\Big) \widetilde{\Psi}(2^{r_j}x,H(t_0)) dx,
\end{align}
where, for the sake of simplicity, we have set $r_j=r(j,m_0)$ and $l_j=l(j,m_0)$. Let $s_*:=|t_0|+2$, observe that, in view of (\ref{eq:inegrl}), one has,
\begin{equation}
\label{eq4lem3:locmodcont}
\forall\, x\in\R,\, |x|\ge s_*\,\Longrightarrow \big|x+l_j2^{-r_j}\big|\ge 1.
\end{equation}
Also, observe that (\ref{eq3lem3:locmodcont}) entails that,
\begin{equation}
\label{eq5lem3:locmodcont}
\big|\check{\eps}_j(\o_0)\big| \le S_j+Z_j,
\end{equation}
where 
\begin{equation}
\label{eq6lem3:locmodcont}
S_j=2^{r_j(1+H(t_0))} \int_{|x|<s_*} \Big|X\big(x+l_j2^{-r_j},H(t_0),\o_0\big)-X\big(t_0,H(t_0),\o_0\big)\Big| \big|\widetilde{\Psi}(2^{r_j}x,H(t_0))\big| dx
\end{equation}
and 
\begin{equation}
\label{eq7lem3:locmodcont}
Z_j=2^{r_j(1+H(t_0))} \int_{|x|\ge s_*} \Big|X\big(x+l_j2^{-r_j},H(t_0),\o_0\big)-X\big(t_0,H(t_0),\o_0\big)\Big| \big|\widetilde{\Psi}(2^{r_j}x,H(t_0))\big| dx.
\end{equation}
Let us now give an appropriate upper bound for $S_j$. Notice that, the fact that $t\mapsto X(t,H(t_0),\o_0)$ is a continuous function over $\R$, entails that, (\ref{eq1lem3:locmodcont}) remains valid, when $[-M,M]$ is replaced by any other compact interval; also notice that, in view of (\ref{eq:inegrl}), when $|x|<s_*$, then $x+l_j2^{-r_j}$ belongs to the compact interval $\big[-s_*-|t_0|-4/5, s_*+|t_0|+4/5\big]$. Thus, using (\ref{eq1lem3:locmodcont}) in which $M$ is replaced by 
$s_*+|t_0|+4/5$, one gets that,
\begin{align}
\label{eq8lem3:locmodcont}
S_j & \le C_1 (\o_0) 2^{r_j(1+H(t_0))}\int_{|x|<s_*} |\nu_j+x|^{H(t_0)}\Big (1+\big|\log|\nu_j+x|\big|\Big)^{1/\al}\big|\widetilde{\Psi}(2^{r_j}x,H(t_0))\big| dx\nonumber\\
&\le  C_1 (\o_0) 2^{r_j(1+H(t_0))}\int_{\R} |\nu_j+x|^{H(t_0)}\Big (1+\big|\log|\nu_j+x|\big|\Big)^{1/\al}\big|\widetilde{\Psi}(2^{r_j}x,H(t_0))\big| dx,
\end{align}
where, $C_1 (\o_0)$ is a constant non depending on $j$, and
\begin{equation}
\label{eq9lem3:locmodcont}
\nu_j:=l_j2^{-r_j}-t_0;
\end{equation}
observe that (\ref{eq:inegrl}) implies that,
\begin{equation}
\label{eq10lem3:locmodcont}
R+1 < 2^{r_j}\nu_j < 2R+2.
\end{equation}
For the sake of convenience, let us set,
\begin{equation}
\label{eq11lem3:locmodcont}
c_2:=\sup_{y\in\R}\,(3+|y|)^2 \big|\widetilde{\Psi}(y,H(t_0))\big|<\infty;
\end{equation}
observe that the inequality in (\ref{eq11lem3:locmodcont}) results from (\ref{localpsitilde}).
Next, making in (\ref{eq8lem3:locmodcont}) the change of variable $u=x/\nu_j$, and using the triangle inequality, (\ref{eq10lem3:locmodcont}),
(\ref{eq11lem3:locmodcont}), (\ref{eq9lem3:locmodcont}), the last two inequalities in (\ref{eq:inegrl}), and the the first equality in (\ref{eq:ljm0}), it follows that,
\begin{align}
\label{eq12lem3:locmodcont}
S_j & \le C_1 (\o_0)\,2^{r_j(1+H(t_0))}\nu_j \int_{\R} |\nu_j+\nu_j u|^{H(t_0)}\Big(1+\big|\log|\nu_j+\nu_j u|\big|\Big)^{1/\al}\big|\widetilde{\Psi}(2^{r_j}\nu_j u,H(t_0))\big| du\nonumber\\
& = C_1 (\o_0)\, 2^{r_j(1+H(t_0))}\nu_{j}^{1+H(t_0)}\int_{\R} |1+u|^{H(t_0)}\Big(1+\big|\log (\nu_j)+\log|1+u|\big|\Big)^{1/\al}\big|\widetilde{\Psi}(2^{r_j}\nu_j u,H(t_0))\big| du
\nonumber\\
&\le C_3 (\o_0)\big( 2^{r_j}\nu_{j}\big)^{1+H(t_0)}\,\big|\log (\nu_j)\big|^{1/\al},\nonumber\\
&\le C_4(\o_0) \, j^{1/\al}
\end{align}
where, the constant,
$$
C_3 (\o_0):=c_2\big(\log(5/4)\big)^{-1/\al} C_1 (\o_0)\int_{\R} \frac{\big|2+|\log|1+u||\big|^{1/\al}}
{(3+|u|)^{2-H(t_0)}}du<\infty,
$$
and the constant $C_4(\o_0)=C_3 (\o_0)\big(2R+2\big)^{1+H(t_0)} m_{0}^{1/\al}$. Let us now give an appropriate upper bound for $Z_j$. Using (\ref{eq7lem3:locmodcont}), (\ref{eq11lem3:locmodcont}) and the triangle inequality, one obtains that,
\begin{equation}
\label{eq13lem3:locmodcont}
Z_j \le c_2 2^{-r_j (1-H(t_0))} \int_{|x|\ge s_*} \big|X\big(x+l_j2^{-r_j},H(t_0),\o_0\big)\big|\, x^{-2}dx+C_5(\o_0) 2^{-r_j (1-H(t_0))},
\end{equation}
where, the constant
$$
C_5(\o_0):=c_2 \big|X\big(t_0,H(t_0),\o_0\big)\big|\int_{|x|\ge s_*} x^{-2}dx<\infty.
$$
Next, observe that, (\ref{eq4lem3:locmodcont}) and (\ref{eq:inegrl}) imply that for all real number $x$ which satisfies $|x|\ge s_*$, and for each $j\in\N$,
one has, 
$$
1\le \big|x+l_j2^{-r_j}\big|\le |x|+|t_0|+1;
$$
thus, taking in (\ref{eq1:pasymp}), $q=0$, $a,b$ such that $H(t_0)\in [a,b]$, and $\eta$ an arbitrary fixed positive real number, it follows that,
\begin{equation}
\label{eq14lem3:locmodcont}
\big|X\big(x+l_j2^{-r_j},H(t_0),\o_0\big)\big|\le C_6(\o_0) \big (|x|+|t_0|+1)^{H(t_0)}
\Big (1+\log\big(|x|+|t_0|+1\big)\Big)^{1/\al+\eta},
\end{equation}
where the finite constant $C_6(\o_0)$ does not depend on $x$ and $j$. Next, combining (\ref{eq13lem3:locmodcont}) with (\ref{eq14lem3:locmodcont}),
one gets, that,
\begin{equation}
\label{eq15lem3:locmodcont}
Z_j \le C_7(\o_0)2^{-r_j (1-H(t_0))},
\end{equation}
where the finite constant
$$
C_7(\o_0):=C_5(\o_0)+c_2\int_{|x|\ge s_*}\big (|x|+|t_0|+1)^{H(t_0)}\Big (1+\log\big(|x|+|t_0|+1\big)\Big)^{1/\al+\eta}\, x^{-2}dx.
$$
Finally, putting together, (\ref{eq5lem3:locmodcont}), (\ref{eq12lem3:locmodcont}), (\ref{eq15lem3:locmodcont}) and the first equality in (\ref{eq:ljm0}), one obtains (\ref{eq2lem3:locmodcont})
\end{proof}

Now, we are in position to prove Proposition~\ref{prop:locmodcont} and Theorem~\ref{locoptim:modcont}.

\begin{proof}[Proof of Proposition~\ref{prop:locmodcont}.] The proposition is a straightforward consequence of Lemmas~\ref{lem2:locmodcont}~and~\ref{lem3:locmodcont}.
\end{proof}

\begin{proof}[Proof of Theorem~\ref{locoptim:modcont}.] Using (\ref{def:LMSMbis}) and the triangle inequality, one has, for all $t\in [-M,M]$,
$$
\big| X(t,H(t_0))-X(t_0,H(t_0))\big|\le \big|Y(t)-Y(t_0)\big|+\big|X(t,H(t))-X(t,H(t_0))\big|,
$$
and, as a consequence, 
\begin{align*}
& \sup_{t\in [-M,M]} \left\{\frac{\big| X(t,H(t_0))-X(t_0,H(t_0))\big|}{ |t-t_0|^{H(t_0)} \big(1+\big|\log |t-t_0|\big| \big)^{1/\alpha}} \right\}\\
& \hspace{0.5cm} \le \sup_{t\in [-M,M]} \left\{\frac{\big|Y(t)-Y(t_0)\big|}{ |t-t_0|^{H(t_0)} \big(1+\big|\log |t-t_0|\big| \big)^{1/\alpha}} \right\}
+\sup_{t\in [-M,M]} \left\{\frac{\big|X(t,H(t))-X(t,H(t_0))\big|}{ |t-t_0|^{H(t_0)} \big(1+\big|\log |t-t_0|\big| \big)^{1/\alpha}} \right\}.
\end{align*}
Thus, in view of (\ref{eq1:plocmodcont}), in order to show that (\ref{eq2:tlocmodcont}) holds, it is sufficient to prove that,
\begin{equation}
\label{eq3:tlocmodcont}
\sup_{t\in [-M,M]} \left\{\frac{\big|X(t,H(t))-X(t,H(t_0))\big|}{ |t-t_0|^{H(t_0)} \big(1+\big|\log |t-t_0|\big| \big)^{1/\alpha}} \right\}<\infty.
\end{equation}
Taking in (\ref{eq2:CTWSEbis}) $q=0$, $a=\underline{H}:=\inf_{x\in\R} H(x)$ and $b:=\overline{H}:=\sup_{x\in\R} H(x)$, one gets that,
\begin{equation}
\label{eq4:tlocmodcont}
\sup_{t\in [-M,M]} \left\{\frac{\big|X(t,H(t))-X(t,H(t_0))\big|}{|H(t)-H(t_0)|} \right\}<\infty.
\end{equation}
Finally, combining (\ref{eq1:tlocmodcont}) with (\ref{eq4:tlocmodcont}), it follows that (\ref{eq3:tlocmodcont}) holds.
\end{proof}

\section{Local H\"older exponent of LMSM}
\label{subsec:lheLMSM}
The goal of this section is to determine the local H\"older exponent of a typical path of LMSM. Let us first recall, in a general framework, the definition of this exponent.

Denote by $f$ an arbitrary deterministic real-valued continuous function defined on the real line. The critical global H\"older regularity of $f$, over an arbitrary nonempty compact interval $[M_1,M_2]$, can be measured through, 
\begin{equation}
\label{eq1:ppintro}
\rho_{f}^{\mbox{{\tiny unif}}}\big([M_1,M_2]\big):=\sup\left\{\rho\ge 0: \sup_{s',s''\in [M_1,M_2]}\frac{|f(s')-f(s'')|}{|s'-s''|^\rho} <\infty\right\},
\end{equation}
the uniform (or global) H\"older exponent of $f$ over $[M_1,M_2]$; observe that one has 
\begin{equation}
\label{eq1bis:ppintro}
\rho_{f}^{\mbox{{\tiny unif}}}\big([M'_1,M'_2]\big)\ge \rho_{f}^{\mbox{{\tiny unif}}}\big([M_1,M_2]\big),
\end{equation}
 when $[M'_1,M'_2]\subseteq [M_1,M_2]$. The local H\"older regularity of $f$ in a neighborhood of some point $t_0\in\R$, can be measured through,
\begin{equation}
\label{eq2:ppintro}
\rho_{f}^{\mbox{{\tiny unif}}}(t_0):=\sup\left\{\rho_{f}^{\mbox{{\tiny unif}}}\big([M_1,M_2]\big): M_1\in\R,\, M_2\in\R \mbox{ and } M_1<t_0<M_2\right\},
\end{equation}
the local H\"older exponent of $f$ at $t_0$; notice that the latter exponent is sometime called the uniform pointwise H\"older exponent of $f$ at $t_0$ (see \cite{stoev2005path}).

Let $t\mapsto Y(t,\o)$ be a continuous path of the LMSM $\{Y(t):t\in\R\}$. The uniform H\"older exponent of $t\mapsto Y(t,\o)$ over $[M_1,M_2]$, is denoted by $\rho_{Y}^{\mbox{{\tiny unif}}}\big([M_1,M_2],\o\big)$; the local H\"older exponent of $t\mapsto Y(t,\o)$ at $t_0$, is denoted by $\rho_{Y}^{\mbox{{\tiny unif}}}(t_0,\o)$.

Thanks to Part $(ii)$ of Corollary~\ref{modcont2Y} and thanks to Theorem~\ref{optim:modcont}, under some H\"older condition on $H(\cdot)$, one can, almost surely for all 
$t_0\in\R$, completely determine $\rho_{Y}^{\mbox{{\tiny unif}}}(t_0,\o)$, more precisely:
\begin{thm}
\label{theo:lheLMSM}
There is $\Omega_4 ^*$ an event of probability 1 (non depending on $t_0$), such that for all $\o\in\Omega_4 ^*$ and for each $t_0\in\R$ satisfying,
\begin{equation}
\label{eq1:tlheLMSM}
\rho_{H}^{\mbox{{\tiny unif}}}(t_0)>1/\al,
\end{equation}
one has,
\begin{equation}
\label{eq2:tlheLMSM}
\rho_{Y}^{\mbox{{\tiny unif}}}(t_0,\o)=H(t_0)-1/\al. 
\end{equation}
\end{thm}

Notice that the latter theorem is a more precise result than Theorem~4.1 in \cite{stoev2005path}.

\begin{proof}[Proof of Theorem~\ref{theo:lheLMSM}.] The theorem does not make sense if there is no $t_0\in\R$ which satisfies (\ref{eq1:tlheLMSM}), so in all
the sequel, we assume that (\ref{eq1:tlheLMSM}) is satisfied for some $t_0\in\R$. In view of (\ref{eq1bis:ppintro}) and (\ref{eq2:ppintro}), this assumption implies 
that the set,
$$
\Lambda:=\left\{(\mu_1,\mu_2)\in\Q^2 : \mu_1<\mu_2 \mbox{ and } \rho_{H}^{\mbox{{\tiny unif}}}\big([\mu_1,\mu_2]\big)>1/\al\right\},
$$
is nonempty. Next, observe that, (\ref{eq1:ppintro}), Part $(ii)$ of Corollary~\ref{modcont2Y}, Theorem~\ref{optim:modcont} and Remark~\ref{rem:tau}, implies that, for all $(\mu_1,\mu_2)\in \Lambda$, one has, almost surely, 
\begin{equation}
\label{eq3bis:tlheLMSM}
\rho_{Y}^{\mbox{{\tiny unif}}}\big([\mu_1,\mu_2]\big)=\min_{x\in [\mu_1,\mu_2]} H(x)-1/\al;
\end{equation}
moreover, the fact that $\Lambda$ is a countable set, entails that (\ref{eq3bis:tlheLMSM}) even holds on $\Omega_4 ^*$, an event of probability 1 which does not depend
on $(\mu_1,\mu_2)$. Also, observe that, for each $t_0\in \R$ which satisfies (\ref{eq1:tlheLMSM}), one has, for all $\o\in\Omega_{4}^*$,
\begin{equation}
\label{eq3ter:tlheLMSM}
\rho_{Y}^{\mbox{{\tiny unif}}}(t_0,\o)=\sup\left\{\rho_{Y}^{\mbox{{\tiny unif}}}\big([\mu_1,\mu_2]\big): (\mu_1,\mu_2)\in\Lambda \mbox{ and } \mu_1 < t_0 <\mu_2\right\};
\end{equation}
the latter equality can be obtained by using (\ref{eq1bis:ppintro}), (\ref{eq2:ppintro}) and the fact that the set of the rational numbers is dense in the set of the 
real numbers. Finally, since $H(\cdot)$ is a continuous function, combining (\ref{eq3bis:tlheLMSM}) with (\ref{eq3ter:tlheLMSM}), one gets (\ref{eq2:tlheLMSM}).
\end{proof}

\section{Appendix}
\label{sec:appendix}


The following lemma is a standard result.

\begin{lemma}\label{LA4}
For all fixed positive real number $\lambda$, there exists a finite constant $c$ which only depends on $\lambda$, such that for each nonnegative real numbers $x$ and $y$, one has,
\begin{equation}
(x+y)^{\lambda} \leq c (x^{\lambda}+y^{\lambda}), \nonumber
\end{equation}
with the convention that $0^\lambda=0$.
\end{lemma}

The following technical lemma plays a crucial role in the proof of Part $(iii)$ of Proposition~\ref{prop:cauchy} as well as in those of other important results in our article.

\begin{lemma}\label{LA5}
Let $(p,q)\in\{0,1,2\}\times\Z_+$ be arbitrary and fixed. We set $\phi:=\partial_{x}^p \partial_{v}^q \Psi$, where $\Psi$ is the function introduced in (\ref{PSI}). Let $M$, $\nu$, $a$, $b$ and $\kappa$ be arbitrary and fixed real numbers satisfying $M>0$,  $1>b>a>1/\al$, $a-1/\al>\kappa$ and $a-1/\al-\kappa>\nu\ge 0$. At last, let $i$ be an arbitrary and fixed nonnegative integer.
For all $n\in\Z_+$ and $(t,s,v)\in \R^2\times (1/\al,1)$ we set,
\begin{align}\label{SAN}
& A_n(t,s,v;M,\kappa,\nu,i,\phi)\nonumber \\
&:=  \sum_{|j| \leq n} \sum_{|k| > M2^{n+1}} 2^{-jv} \frac{ \big| \phi(2^jt-k,v) - \phi(2^js-k,v)\big| }{|t-s|^{\kappa}}(3+|j|)^{i+1/\alpha+\nu} (3+|k|)^{1/\alpha+\nu}
\end{align}
and 
\begin{align}\label{SBN}
& B_n(t,s,v;M,\kappa,\nu,i,\phi)\nonumber \\
&:= \sum_{|j| \geq n+1} \sum_{k\in \Z} 2^{-jv} \frac{ \big|\phi(2^jt-k,v) - \phi(2^js-k,v)\big| }{|t-s|^{\kappa}}(3+|j|)^{i+1/\alpha+\nu} (3+|k|)^{1/\alpha+\nu}, 
\end{align}
with the convention that $A_n(t,t,v;M,\kappa,\nu,i,\phi)=B_n(t,t,v;M,\kappa,\nu,i,\phi)=0$ for any $t\in\R$. Then, when $n$ goes to $+\infty$, $A_n(t,s,v;M,\kappa,\nu,i,\phi)$ and $B_n(t,s,v;M,\kappa,\nu,i,\phi)$ converge to $0$, uniformly in $(t,s,v) \in [-M,M]^2 \times [a,b]$.
\end{lemma}

In order to prove Lemma~\ref{LA5}, we need some preliminary results.

\begin{lemma} \label{LA1}
For all fixed real numbers $\xi >0$ and $M>0$, there exists a constant $c>0$ such that 
for each integer $n\ge 0$,  
\begin{equation}
\sum_{k>M2^{n+1}} (1+k)^{-1-\xi} \leq c 2^{-n\xi}. \nonumber
\end{equation}
\end{lemma}

\begin{proof}[Proof of Lemma \ref{LA1}] Clearly, one has for all integer $k\ge 1$, $(1+k)^{-1-\xi}\le \int_{k-1}^k (1+x)^{-1-\xi}dx$. Therefore,
$$
\sum_{k>M2^{n+1}} (1+k)^{-1-\xi}\le \int_{M2^{n+1}-1}^{+\infty}(1+x)^{-1-\xi}dx = \xi^{-1} M^{-\xi} 2^{-(n+1)\xi}.
$$
\end{proof}

\begin{lemma}\label{LA2}
Let $\lambda \in \R$ and $\theta_0 >0$ be fixed. Set $c:=\sum_{m=0}^{+\infty} 2^{-m\theta_0}(1+m)^{|\lambda|} < +\infty.$ Then for all real number $\theta$ such that $|\theta| \geq \theta_0$ and each $n_0,n_1 \in \lbrace 0,\pm 1, \ldots, \pm \infty \rbrace$ satisfying $n_0<n_1$, one has,
\begin{equation}
\label{eq1:LA2-ant}
\sum_{n=n_0}^{n_1} 2^{n\theta} (1+|n|)^{\lambda} \leq c
	\begin{cases}
        2^{n_0 \theta} (1+|n_0|)^{\lambda} & \mbox{if } \theta < 0 \\
        2^{n_1 \theta} (1+|n_1|)^{\lambda} & \mbox{if } \theta >0,
    \end{cases}
\end{equation}
with the convention that $2^{-\infty} (1+\infty)^{\lambda}=0$ and $2^{+\infty} (1+\infty)^{\lambda}=+\infty$
\end{lemma}

\begin{proof}[Proof of Lemma \ref{LA2}]
First, notice that the lemma clearly holds in the following three cases:
\begin{itemize}
\item $n_0=-\infty$ and $n_1=+\infty$;
\item $n_0=-\infty$ and $\theta<0$;
\item $n_1=+\infty$ and $\theta>0$.
\end{itemize}
Indeed, in the latter three cases, (\ref{eq1:LA2-ant}) becomes $+\infty\le +\infty$.
 
Let us study the case where $\theta < 0$ and $-\infty < n_0 < n_1 \leq +\infty$, the case where $\theta > 0$ and $-\infty \le n_0 < n_1 < +\infty$ can be treated similarly. One has 
\begin{align*}
& \sum_{n=n_0}^{n_1} 2^{n\theta} (1+|n|)^{\lambda} \leq  \sum_{m=0}^{+\infty} 2^{(m+n_0)\theta} (1+|m+n_0|)^{\lambda} =  2^{n_0\theta} (1+|n_0|)^{\lambda} \sum_{m=0}^{+\infty} 2^{m\theta} \left( \frac{1+|m+n_0|}{1+|n_0|} \right)^{\lambda} \\
& \leq  2^{n_0\theta} (1+|n_0|)^{\lambda} \sum_{m=0}^{+\infty} 2^{-m\theta_0} \left( \frac{1+|m+n_0|}{1+ |n_0|} \right)^{\lambda}.
\end{align*}
Thus, it remains to show that 
\begin{equation}
\label{eq2:LA2-ant}
\sum_{m=0}^{+\infty} 2^{-m\theta_0} \left( \frac{1+|m+n_0|}{1+ |n_0|} \right)^{\lambda}\le c:=\sum_{m=0}^{+\infty} 2^{-m\theta_0}(1+m)^{|\lambda|}.
\end{equation}
In fact, (\ref{eq2:LA2-ant}) can be obtained by proving the following: for every integer $m\ge 0$, one has,
\begin{equation}
\label{eq3:LA2-ant}
\frac{1}{1+m} \leq \frac{1+|m+n_0|}{1+|n_0|} \leq 1+ m. 
\end{equation}
Clearly the second inequality in (\ref{eq3:LA2-ant}) is satisfied. To prove that the first one holds, we argue by cases:
\begin{itemize}
\item if $n_0 \geq 0$, one gets $\frac{1+|m+n_0|}{1+|n_0|}=\frac{1+m+n_0}{1+n_0}=1+\frac{m}{1+n_0}\geq 1 \geq \frac{1}{1+m}$;
\item if $n_0<0$ and $m\geq -n_0=|n_0|$, then $\frac{1+|m+n_0|}{1+|n_0|} \ge \frac{1}{1+|n_0|}\geq \frac{1}{1+m}$;
\item if  $n_0<0$ and $m<-n_0=|n_0|$, then $$\frac{1+|m+n_0|}{1+|n_0|}=\frac{1-m+|n_0|}{1+|n_0|}=1-\frac{m}{1+|n_0|} \geq 1-\frac{m}{1+m} = \frac{1}{1+m}.$$
\end{itemize}
\end{proof}

The following lemma is a more or less classical result, we refer for instance to \cite{ayache2009linear} for its proof.
\begin{lemma}\label{LA3} \cite{ayache2009linear}
For all fixed real numbers $\theta \in [0,1)$ and $\zeta \geq 0,$ there exists
a constant $c>0$ such that for any $u\in\R$, one has, 
\begin{equation}
\sum_{k\in\Z} \frac{(1+|k|)^{\theta} \log^{\zeta}(2+|k|)}{(2+|u-k|)^2 } \leq c(1+|u|)^{\theta} \log^{\zeta}(2+|u|). \nonumber
\end{equation}
\end{lemma}


Now, we are in position to prove Lemma~\ref{LA5}.

\begin{proof}[Proof of Lemma \ref{LA5}]
Let $t,s \in [-M,M]$ be arbitrary and fixed; there is no restriction to assume that $s\ne t$. We denote by $j_0 > -\log_2(2M)-1$  the unique integer such that 
\begin{equation}\label{defj0}
2^{-j_0-1} < |t-s| \leq 2^{-j_0}.
\end{equation}
From now on, for the sake of simplicity we set: 
$$
A_n(t,s,v):=A_n(t,s,v;M,\kappa,\nu,i,\phi) \mbox{ and } B_n(t,s,v):=B_n(t,s,v;M,\kappa,\nu,i,\phi).
$$
Let us first prove that, when $n\rightarrow+\infty$, $A_n(t,s,v)$ converges to $0$, uniformly in $(t,s,v) \in [-M,M]^2 \times [a,b]$. So, in the sequel, we assume that $j$ is an arbitrary integer satisfying $|j| \leq n$. We need to derive suitable upper bounds for the quantity 
\begin{equation}\label{SANk}
A_n^{j}(t,s,v) := \sum_{ |k| > M2^{n+1}}  \frac{ \big|\phi(2^jt-k,v) - \phi(2^js-k,v)\big| }{ |t-s|^{\kappa}}(3+|k|)^{1/\alpha+\nu}.
\end{equation}
For this purpose, we consider two cases $j \leq j_0$ and $j\geq j_0+1$ separately. First, we suppose that
\begin{equation}\label{jj0}
j \leq j_0.
\end{equation}
Using the Mean Value Theorem, (\ref{localisation}), (\ref{defj0}) and (\ref{jj0}), one obtains that 
\begin{align}\label{eq1:ant-lem8}
\nonumber
\big| \phi(2^jt-k,v) - \phi(2^js-k,v) \big| & \leq c_1 2^j |t-s| \sup_{u\in I}(3+|u|)^{-2} \\
									  & \leq c_1 2^j |t-s| (2+|2^jt-k|)^{-2}, 
\end{align}
where $I$ denotes the compact interval with end-points $2^jt-k$ and $2^js-k$. It is worth noticing that, in view of (\ref{defj0}) and (\ref{jj0}), the length of $I$ is at most $1$; this is why the last inequality holds. Next, (\ref{eq1:ant-lem8}) and (\ref{SANk}) entail that
\begin{equation}\label{m1SANk}
A_n^{j}(t,s,v) \leq c_1 2^j |t-s|^{1-\kappa} \sum_{|k| > M2^{n+1}} (3+|k|)^{1/\alpha+\nu} (2+|2^jt-k|)^{-2}.
\end{equation}
Moreover, using the inequalities $|t| \leq M, |j|\leq n$ and $|k| > M2^{n+1},$ one gets 
\begin{equation}\label{m2ker}
(3+|k|)^{1/\alpha+\nu} (2+|2^jt-k|)^{-2} \leq (3+|k|)^{1/\alpha+\nu} (2+|k|- 2^jM)^{-2} \leq c_2 (1+|k|)^{-(2-1/\alpha-\nu)}.
\end{equation}
Putting together (\ref{m1SANk}) and (\ref{m2ker}), one obtains that
\begin{equation*}
A_n^{j}(t,s,v) \leq c_3 2^j |t-s|^{1-\kappa} \sum_{|k| > M2^{n+1}} (1+|k|)^{-(2-1/\alpha-\nu)}.
\end{equation*}
Then Lemma~\ref{LA1} (in which one takes $\xi=1-1/\alpha-\nu$) and Relation (\ref{defj0}), imply that 
\begin{equation}\label{m3SANk}
A_n^{j}(t,s,v) \leq c_4 2^{j_0(\kappa-1)+j-n(1-1/\alpha-\nu)}.
\end{equation}
Let us now study the second case where
\begin{equation}\label{j01j}
j_0+1 \leq j.
\end{equation}
It follows from (\ref{SANk}), (\ref{defj0}) and (\ref{j01j}) that 
\begin{equation}\label{m4SANk}
A_n^j(t,s,v) \leq 2^{j\kappa} \sum_{|k| > M2^{n+1}} \left\lbrace \big| \phi(2^jt-k,v) \big| + \big| \phi(2^js-k,v) \big| \right\rbrace (3+|k|)^{1/\alpha+\nu}.
\end{equation}
Moreover, using (\ref{localisation}) and the fact that $|j| \leq n$, one has for all $(u,v)\in [-M,M]\times [a,b]$ and $k\in\Z$ satisfying $|k|> M2^{n+1},$
\begin{align}\label{m5ker}
\big| \phi(2^ju-k,v)\big| & \leq c_5(3+|2^j u-k|)^{-2} \leq c_5(3+|k| -2^j|u|)^{-2} \nonumber \\
                    & \leq c_5 (3+|k|-2^nM)^{-2} \leq c_6 (3+|k|)^{-2}. 
\end{align}
Combining (\ref{m5ker}) with (\ref{m4SANk}), one gets that
\begin{equation*}
A_n^{j}(t,s,v) \leq c_6 2^{j\kappa+1} \sum_{|k|> M2^{n+1}} (3+|k|)^{-(2-1/\alpha-\nu)}.
\end{equation*}
Thus, it follows from Lemma \ref{LA1} (in which one takes $\xi=1-1/\alpha-\nu$), that
\begin{equation}\label{m6SANk}
A_n^{j}(t,s,v) \leq c_7 2^{j\kappa-n(1-1/\alpha-\nu)}.
\end{equation}
Putting together (\ref{SAN}), (\ref{SANk}), (\ref{m3SANk}) and (\ref{m6SANk}), one obtains that
\begin{equation}\label{m7SAN}
A_n(t,s,v) \leq c_8 2^{-n(1-1/\alpha-\nu)} \left[ 2^{j_0(\kappa-1)} \sum_{j=-\infty}^{j_0} 2^{j(1-v)}(3+|j|)^{i+1/\alpha+\nu} + \sum_{j=j_0+1}^{+\infty} 2^{j(\kappa-v)} (3+|j|)^{i+1/\alpha+\nu}  \right].
\end{equation}
Next, using Lemma~\ref{LA2} with $n_0=-\infty$, $n_1=j_0$, $\theta=1-v>0$, $\theta_0=1-b$ and $\lambda=i+1/\alpha+\nu$, one 
gets that 
\begin{equation}
\label{eq2:ant-lem8}
\sum_{j=-\infty}^{j_0} 2^{j(1-v)}(3+|j|)^{i+1/\alpha+\nu}\le c_9 2^{j_0 (1-v)}(1+|j_0|)^{i+1/\alpha+\nu}
\end{equation}
and using again the same lemma with $n_0=j_0+1$, $n_1=+\infty$, $\theta=\kappa-v<0$, $\theta_0=1/\al$ and $\lambda=i+1/\alpha+\nu$,
one obtains that
\begin{equation}
\label{eq3:ant-lem8}
\sum_{j=j_0+1}^{+\infty} 2^{j(\kappa-v)} (3+|j|)^{i+1/\alpha+\nu} \le c_{10} 2^{j_0(\kappa-v)}(1+|j_0|)^{i+1/\alpha+\nu}.
\end{equation}
Putting together (\ref{m7SAN}), (\ref{eq2:ant-lem8}), (\ref{eq3:ant-lem8}),
the inequality $v-\kappa\ge 1/\al$ and the inequality $j_0> -\log_2 (2M)-1$, 
one obtains that
\begin{equation}
\label{eq4:ant-lem8}
A_n(s,t,v)\le c_{11} 2^{-j_0(v-\kappa)}(1+|j_0|)^{i+1/\alpha+\nu}
2^{-n(1-1/\alpha-\nu)}\le c_{12} 2^{-n(1-1/\alpha-\nu)},
\end{equation}
where
$$c_{12}:=c_{11}\sup\left\{2^{-j/\al}(1+|j|)^{i+1/\alpha+\nu}\,:\,j\in\Z\mbox{
  and } j> -\log_2 (2M)-1\right\}<+\infty.$$ 
The last inequality in (\ref{eq4:ant-lem8}) implies that when $n\rightarrow +\infty$, $A_n(t,s,v)$ converges to $0$, uniformly in $(t,s,v) \in [-M,M]^2 \times [a,b]$.

From now on our goal is to prove that $B_n(t,s,v)$ converges to $0$ uniformly
in $t,s,v$, when $n$ goes to infinity. So, in all the sequel $j$ denotes an arbitrary integer satisfying $|j| \geq n+1.$ First we derive a suitable upper bound for the quantity
\begin{equation}\label{SBNk}
B^{j}(t,s,v) := \sum_{k\in\Z} \frac{ \big|\phi(2^jt-k,v) - \phi(2^js-k,v)\big| }{ |t-s|^{\kappa}}(3+|k|)^{1/\alpha+\nu}.
\end{equation}
As above, we distinguish two cases : $j\leq j_0$ and $j\ge j_0+1$. First, we
suppose that (\ref{jj0}) is verified. Similarly to (\ref{m1SANk}), one has
that
$$
B^j(t,s,v) \leq c_{13} 2^j |t-s|^{1-\kappa} \sum_{k\in\Z} (3+|k|)^{1/\alpha+\nu}
(2+|2^jt-k|)^{-2}.
$$
Then, using (\ref{defj0}), Lemma~\ref{LA3} (in which we take
 $\theta=1/\al+\nu$ and $\zeta=0$), and the fact that $|t| \leq M,$ one
 obtains that,
\begin{equation}\label{m8SBNk}
B^{j}(t,s,v) \leq c_{14} 2^{j+j_0(\kappa-1)} (1+2^j)^{1/\alpha+\nu}.
\end{equation}
Now let us suppose that (\ref{j01j}) is satisfied. By using this relation,
(\ref{defj0}), the triangle inequality, (\ref{localisation}), Lemma~\ref{LA3}
(in which one takes $\theta=1/\alpha+\nu$ and $\zeta=0$) and the fact that $t,s
\in [-M,M],$ one gets that,
\begin{align}\label{m9SBNk}
B^j(t,s,v) & \leq 2^{j\kappa} \sum_{k\in\Z} (3+|k|)^{1/\alpha+\nu} \left\lbrace \big| \phi(2^jt-k,v) \big| +  \big| \phi(2^js-k,v) \big| \right\rbrace  \nonumber \\
& \leq c_{15} 2^{j\kappa} \sum_{k\in \Z} (1+|k|)^{1/\alpha+\nu} \left\lbrace (2+|2^jt-k|)^{-2} + (2+|2^js-k|)^{-2} \right\rbrace \nonumber \\
& \leq c_{16} 2^{j\kappa} \left\lbrace (1+2^j |t|)^{1/\alpha+\nu} + (1+2^j |s|)^{1/\alpha+\nu} \right\rbrace \nonumber \\
& \leq c_{17} 2^{j(\kappa+1/\alpha+\nu)}.
\end{align}
There is no restriction to assume that $n\ge\log_2(M)+2$, then in view of the
inequality $j_0> -\log_2(M)-2$, one has that $-n-1< j_0$ and thus
(\ref{m8SBNk}) entails that, 
\begin{align}
\sum_{j=-\infty}^{-n-1} 2^{-jv} (3+|j|)^{i+1/\alpha+\nu} B^{j}(t,s,v) & \leq c_{14} \sum_{j=-\infty}^{-n-1} 2^{j(1-v)+j_0(\kappa-1)} (1+2^j)^{1/\alpha+\nu} (3+|j|)^{i+1/\alpha+\nu} \nonumber \\
& \leq c_{18} 2^{j_0(\kappa-1)} \sum_{j=-\infty}^{-n-1} 2^{j(1-v)} (3+ |j| )^{i+1/\alpha+\nu}. \nonumber 
\end{align}
Next, using Lemma~\ref{LA2} (in which one takes $n_0=-\infty$, $n_1=-n-1$, $\theta=1-v$, $\theta_0=1-b$
and $\lambda=i+1/\alpha+\nu$), the inequality $2^{j_0(\kappa-1)} < (4M)^{1-\kappa}$ and the
inequality $v\le b$, one gets that
\begin{equation}\label{m10SBN}
\sum_{j=-\infty}^{-n-1} 2^{-jv} (3+ |j|)^{i+1/\alpha+\nu} B^{j}(t,s,v) \leq c_{19} 2^{-n(1-b)} (4+n)^{i+1/\alpha+\nu}.
\end{equation}
Let us now give a suitable upper bound for $\sum_{j\geq n+1}
2^{-jv}(3+|j|)^{i+1/\alpha+\nu} B^j(t,s,v)$. First we assume that $j_0\ge n+1$;
then, using (\ref{m8SBNk}), one has that,
\begin{align}\label{m11SBN}
\sum_{j=n+1}^{j_0} 2^{-jv}(3+|j|)^{i+1/\alpha+\nu} B^{j}(t,s,v) & \leq c_{14} 2^{j_0(\kappa-1)} \sum_{j=-\infty}^{j_0} 2^{j(1-v)}(3+|j|)^{i+1/\alpha+\nu} (1+2^j)^{1/\alpha+\nu} \nonumber \\
& \leq  c_{20} 2^{j_0(\kappa-1+1/\alpha+\nu)} \sum_{j=-\infty}^{j_0} 2^{j(1-v)}(3+|j|)^{i+1/\alpha+\nu} \nonumber \\
& \leq c_{21} 2^{-j_0(a-1/\al-\kappa-\nu)}(3+|j_0|)^{i+1/\alpha+\nu} \nonumber \\
& \leq c_{22} 2^{-n(a-1/\al-\kappa-\nu)}(3+n)^{i+1/\alpha+\nu}.
\end{align}
Observe that the third inequality in (\ref{m11SBN}) follows from Lemma~\ref{LA2} (in which we take $n_0=-\infty$, $n_1=j_0$, $\theta=1-v$, $\theta_0=1-b$ and $\lambda=i+1/\alpha+\nu$) as well as from the inequality $v\ge a$. Also observe that the last inequality in (\ref{m11SBN}), results from the fact that the function $x\mapsto
2^{-x(a-1/\al-\kappa-\nu)}(3+x)^{i+1/\alpha+\nu}$ is continuous over $\R_{+}$ and decreasing for $x$ big enough.

On the other hand, by making use of (\ref{m9SBNk}), one has that 
\begin{align}\label{m12SBN}
\sum_{j=j_0+1}^{+\infty} 2^{-jv}(3+|j|)^{i+1/\alpha+\nu} B^j(t,s,v) & \leq c_{17} \sum_{j=j_0+1}^{+\infty} 2^{j(\kappa+1/\alpha+\nu-v)} (3+|j|)^{i+1/\alpha+\nu}  \nonumber \\
& \leq c_{23} 2^{-j_0(a-1/\al-\kappa-\nu)} (3+|j_0|)^{i+1/\alpha+\nu} \nonumber \\
& \leq c_{24} 2^{-n(a-1/\al-\kappa-\nu)} (3+n)^{i+1/\alpha+\nu}.
\end{align}
Observe that the second inequality in (\ref{m12SBN}) follows from Lemma~\ref{LA2} (in which we take $n_0=j_0+1$, $n_1=+\infty$,  $\theta=\kappa+1/\alpha+\nu-v$, $\theta_0= a-1/\al-\kappa-\nu$ and $\lambda=i+1/\alpha+\nu$) as well as from the inequality $v\ge a$. Also observe that the last inequality in (\ref{m12SBN}), results from the fact that the function $x\mapsto
2^{-x(a-1/\al-\kappa-\nu)}(3+x)^{i+1/\alpha+\nu}$ is continuous over $\R_{+}$ and decreasing for $x$ big enough.

Combining (\ref{m11SBN}) with (\ref{m12SBN}), it follows that one has, in the case where $j_0\ge n+1$,
\begin{equation}
\label{m13SBN}
\sum_{j=n+1}^{+\infty} 2^{-jv}(3+|j|)^{i+1/\alpha+\nu} B^j(t,s,v) \le c_{25} 2^{-n(a-1/\al-\kappa-\nu)} (3+n)^{i+1/\alpha+\nu}.
\end{equation}

Let us now assume that $j_0<n+1$, then, by making use of (\ref{m9SBNk}), one has that
\begin{align}\label{m14SBN}
\sum_{j=n+1}^{+\infty} 2^{-jv}(3+|j|)^{i+1/\alpha+\nu} B^j(t,s,v) & \leq c_{17} \sum_{j=n+1}^{+\infty} 2^{j(\kappa+1/\alpha+\nu-v)} (3+|j|)^{i+1/\alpha+\nu}  \nonumber \\
& \leq c_{26} 2^{-n(a-1/\al-\kappa-\nu)} (3+n)^{i+1/\alpha+\nu},
\end{align}
where the last inequality follows from Lemma~\ref{LA2} (in which we take $n_0=n+1$, $n_1=+\infty$,  $\theta=\kappa+1/\alpha+\nu-v$, $\theta_0= a-1/\al-\kappa-\nu$ and $\lambda=i+1/\alpha+\nu$) as well as from the inequality $v\ge a$.

Finally, (\ref{SBN}), (\ref{SBNk}), (\ref{m10SBN}), (\ref{m13SBN}) and (\ref{m14SBN}) imply that, for all $n\ge \log_2(M)+2$,
\begin{equation}
\label{m15SBN}
B_n(t,s,v)\le c_{27} \big(2^{-n(1-b)}+2^{-n(a-1/\al-\kappa-\nu)}\big) (4+n)^{i+1/\alpha+\nu};
\end{equation}
which in turn entails that when $n\rightarrow +\infty$, $B_n(t,s,v)$ converges to $0$, uniformly in $(t,s,v) \in [-M,M]^2 \times [a,b]$.
\end{proof}

\noindent{\bf Acknowlegment}
\\

We would like to thank Professor Yves Meyer for his interest in our results. His valuable advice greatly 
improved the introduction of the previous version of the manuscript.

\bibliographystyle{plain}
\bibliography{mabibliodethese}
\end{document}